\documentclass[10pt]{article}

\usepackage[T1]{fontenc}
\usepackage{lmodern}
\usepackage[utf8]{inputenc}
\usepackage{microtype}
\usepackage{framed}
\usepackage{listings}
\usepackage{vmargin}
\usepackage{setspace}
\usepackage{mathrsfs, mathenv}
\usepackage{amsmath, amsthm, amssymb, amsfonts, amscd}
\usepackage{graphicx}
\usepackage{epstopdf}
\usepackage[svgnames]{xcolor}
\usepackage{hyperref}
\hypersetup{citecolor=blue, colorlinks=true, linkcolor=black}
\usepackage[capitalise]{cleveref}
\usepackage{subcaption}
\usepackage{bbm}
\usepackage{cite}
\usepackage{verbatim}
\usepackage{pgfplots}
\usepackage{tikz}
\usetikzlibrary{arrows,decorations.pathmorphing,backgrounds,positioning,fit,matrix}
\usepackage{etoolbox}
\usepackage{color}
\usepackage{lipsum}
\usepackage{ifthen}
\usepackage[ruled, vlined]{algorithm2e}
\usepackage[title]{appendix}
\usepackage{autonum}

\theoremstyle{plain}
\newtheorem{theorem}{Theorem}[section]
\newtheorem{corollary}[theorem]{Corollary}
\newtheorem{lemma}[theorem]{Lemma}
\newtheorem{proposition}[theorem]{Proposition}
\numberwithin{equation}{section}

\theoremstyle{definition}

\theoremstyle{remark}
\newtheorem{remark}[theorem]{Remark}
\newtheorem{assumption}[theorem]{Assumption}

\ifpdf
  \DeclareGraphicsExtensions{.eps,.pdf,.png,.jpg}
\else
  \DeclareGraphicsExtensions{.eps}
\fi

\usepackage{mathtools}

\usepackage{booktabs}

\usepackage{pgfplots}
\usepackage{tikz}
\usetikzlibrary{patterns,arrows,decorations.pathmorphing,backgrounds,positioning,fit,matrix}
\usepackage[labelfont=bf]{caption}
\usepackage{here}
\usepackage[font=normal]{subcaption}

\usepackage{enumitem}
\setlist[itemize]{leftmargin=.5in}
\setlist[enumerate]{leftmargin=.5in,topsep=3pt,itemsep=3pt,label=(\roman*)}

\newcommand*\samethanks[1][\value{footnote}]{\footnotemark[#1]}

\newcommand{\email}[1]{\href{#1}{#1}}
\newcommand{\TheTitle}{Robust Estimation of Effective Diffusions \\ from Multiscale Data} 
\newcommand{\TheAuthors}{G. Garegnani, A. Zanoni}
\title{\TheTitle}
\author{Giacomo Garegnani \thanks{ANMC, Institute of Mathematics, École Polytechnique Fédérale de Lausanne, 1015 Lausanne, Switzerland, \email{giacomo.garegnani@epfl.ch}, \email{andrea.zanoni@epfl.ch}. The authors are partially supported by the Swiss National Science Foundation, under grant No. 200020\_172710.}
		\and Andrea Zanoni \samethanks
}
\date{}

\usepackage{amsopn}

\DeclareMathOperator*{\argmin}{arg\,min}
\newcommand{\abs}[1]{\left\lvert#1\right\rvert}
\newcommand{\norm}[1]{\left\|#1\right\|}
\renewcommand{\phi}{\varphi}
\renewcommand{\theta}{\vartheta}

\newcommand{\R}{\mathbb{R}}

\newcommand{\epl}{\varepsilon}

\newcommand{\defeq}{\coloneqq}

\newcommand{\E}{\operatorname{\mathbb{E}}}

\renewcommand{\d}{\mathrm{d}}
\newcommand{\dd}{\,\mathrm{d}}
\definecolor{shade}{RGB}{100, 100, 100}
\definecolor{bordeaux}{RGB}{128, 0, 50}

\usepackage[usestackEOL]{stackengine}

\definecolor{leg1}{RGB}{0,114,189}
\definecolor{leg2}{RGB}{217,83,25}
\definecolor{leg3}{RGB}{237,177,32}
\definecolor{leg4}{RGB}{126,47,142}
\definecolor{leg5}{RGB}{119,172,48}

\definecolor{leg21}{RGB}{62,38,169}
\definecolor{leg22}{RGB}{46,135,247}
\definecolor{leg23}{RGB}{55,200,151}
\definecolor{leg24}{RGB}{254,195,56}

\ifpdf
\hypersetup{
	pdftitle={\TheTitle},
	pdfauthor={\TheAuthors}
}
\fi

\begin{document}
	
\maketitle	

\begin{abstract} We present a novel methodology based on filtered data and moving averages for estimating effective dynamics from observations of multiscale systems. We show in a semi-parametric framework of the Langevin type that our approach is asymptotically unbiased with respect to the theory of homogenization. Moreover, we demonstrate on a range of challenging numerical experiments that our method is accurate in extracting coarse-grained dynamics from multiscale data. In particular, the estimators we propose are more robust and require less knowledge of the full model than the standard technique of subsampling, which is widely employed in practice in this setting.
\end{abstract}
 


\section{Introduction}

Inferring simple effective models from observations of complex phenomena characterized by multiple scales is a problem of interest in several fields. Examples range from chemical models of molecular dynamics \cite{LeS16b,LeL12,PaS07,PPS09}, where reactions may occur at widely separated time scales, to the modelling of financial markets characterized by market-microstructure noise \cite{AiJ14,AMZ05,ZMA05,OSP10}. Extracting from data a simple surrogate of multiscale models is of the utmost relevance also in oceanography, meteorology, and marine biology \cite{CoP09,KPP15}.

In this paper, we are interested in inferring coarse-grained equations from observations of diffusion processes evolving on multiple time scales. Given a positive integer $d$, a drift function $b^\epl \colon \R^d \times \R^d \to \R^d$ periodic with respect to its second argument, a multiscale parameter $\epl > 0$, and a diffusion coefficient $\sigma > 0$, we consider the $d$-dimensional multiscale SDE
\begin{equation}\label{eq:SDE_MS}
\d X^\epl(t) = b^\epl\left(X^\epl(t), \frac{X^\epl(t)}\epl\right) \dd t + \sqrt{2\sigma} \dd W(t),
\end{equation}
where $W \defeq (W(t), t \geq 0)$ is a standard $d$-dimensional Brownian motion, and where $X^\epl(0)$ is a given initial condition. Assuming that continuous-time data $X^\epl \defeq (X^\epl(t), 0 \leq t \leq T)$ are provided, with $T$ a finite time horizon, our goal is then inferring a coarse-grained equation, independent of the fastest scale $\mathcal O(\epl^{-1})$, which reads
\begin{equation}\label{eq:SDE_HOM}
\d X^0(t) = b^0\left(X^0(t)\right) \dd t + \sqrt{2\Sigma} \dd W(t),
\end{equation}
where $b^0 \colon \R^d \to \R^d$ and $\Sigma \in \R^{d\times d}$ are the effective drift function and diffusion matrix, respectively. Knowledge of the full model \eqref{eq:SDE_MS} yields, in specific instances, a single-scale model \eqref{eq:SDE_HOM} which is effective in the sense of the theory of homogenization. In particular, one can prove in these cases that $X^\epl \to X^0$ for $\epl \to 0$ in a weak sense (see \cite[Chapter 18]{PaS08} or \cite[Chapter 3]{BLP78}). In this work, we consider $b^\epl$ and $\sigma$ to be unknown and wish to infer the parameters $b^0$ and $\Sigma$ of \eqref{eq:SDE_HOM} from multiscale data. Hence, the problem we consider here could be framed into the setting of data-driven homogenization. 

\subsection{Setup: Multiscale Overdamped Langevin Dynamics}\label{sec:Setting}

The class of multiscale SDEs which can be written as \eqref{eq:SDE_MS} is vast, and can be employed for modeling a wide range of physical and social phenomena. In this work, we narrow the scope by considering a semi-parametric framework and a gradient structure, inspired by simple models of molecular dynamics. Let $N$ be a positive integer, consider smooth functions $\{V_i \colon \R^d \to \R\}_{i=1}^N$, and a periodic function $p \colon \R^d \to \R$ with period $L_i$ in the $i$-th direction in $\R^d$ for $i = 1, \ldots, d$. We then let the drift function in the multiscale dynamics \eqref{eq:SDE_MS} be given by
\begin{equation}
b^\epl\left(x, y\right) = -\sum_{i=1}^N \alpha_i \nabla V_i(x) - \frac1\epl \nabla p(y),
\end{equation}
where $\{\alpha_i\}_{i=1}^N$ are scalar drift coefficients. With this choice, equation \eqref{eq:SDE_MS} reads
\begin{equation}\label{eq:SDE_MS_Lang}
\d X^\epl(t) = -\sum_{i=1}^N \alpha_i \nabla V_i(X^\epl(t)) \dd t - \frac1\epl \nabla p\left(\frac{X^\epl(t)}\epl\right) \dd t + \sqrt{2\sigma} \dd W(t),
\end{equation}
and the stochastic model we consider is of the overdamped Langevin type. There exists for equation \eqref{eq:SDE_MS_Lang} a model of the form \eqref{eq:SDE_HOM} which is effective in the homogenization limit $\epl \to 0$. Let $X^\epl \defeq (X^\epl(t), 0 \leq t \leq T)$ denote the solution of \eqref{eq:SDE_MS_Lang} for a finite time horizon $T$. Then, it holds $X^\epl \to X^0$ in law in $C^0([0,T]; \R^d)$ for $\epl \to 0$, where $X^0 \defeq (X^0(t), 0 \leq t \leq T)$ is the solution of the overdamped Langevin equation
\begin{equation}\label{eq:SDE_HOM_Lang}
\d X^0(t)  = -\sum_{i=1}^N A_i \nabla V_i(X^0(t)) \dd t + \sqrt{2\Sigma} \dd W(t).
\end{equation}
Here, the matrices $A_i \defeq \alpha_i \mathcal K$ and $\Sigma \defeq \sigma \mathcal K$, where $\mathcal K \in \R^{d\times d}$ is the symmetric positive semidefinite matrix defined by
\begin{equation}\label{eq:HomogenizationCoeff}
\mathcal K = \int_{\mathbb{L}} (I + D\Phi(y))(I + D\Phi(y))^\top \dd \nu(y), \quad  \mathbb L \defeq \bigotimes_{i=1}^d [0, L_i],
\end{equation}
where $D\Phi$ is the Jacobian of the solution $\Phi \colon \R^d \to \R^d$ of the vector-valued PDE, or cell problem
\begin{equation}\label{eq:Cell}
\begin{aligned}
&\mathcal L_0 \Phi = \nabla p, &&\text{in } \mathbb L, \quad + \text{ periodic b.c. on } \partial \mathbb L, \\
&\int_{\mathbb L} \Phi(y) \dd \nu(y) = 0,
\end{aligned}
\end{equation}
and where the differential operator $\mathcal L_0$, applied component-wise to $\Phi$, is defined as
\begin{equation}
\mathcal L_0 = - \nabla p \cdot \nabla + \sigma \Delta.
\end{equation}
The measure $\nu$ introduced in \eqref{eq:HomogenizationCoeff} is the probability measure on $\mathbb L$ given by
\begin{equation}
\nu(\d y) = \frac1{C_\nu}\exp\left(-\frac{p(y)}{\sigma}\right) \dd y, \qquad C_\nu = \int_{\mathbb L} \exp\left(-\frac{p(y)}{\sigma}\right) \dd y,
\end{equation}
which accounts for the fluctuations at infinity of the fast-scales of the solution of \eqref{eq:SDE_MS_Lang}. We refer the reader to \cite[Chapters 11 and 18]{PaS08} for a complete derivation and proof of this homogenization result.

\subsection{Failure of Standard Estimators}

We consider from now on for clarity the case $d = 1$, for which equation \eqref{eq:SDE_MS_Lang} reads
\begin{equation}\label{eq:SDE_MS_Lang_d1}
\d X^\epl(t) = -\alpha \cdot V'(X^\epl(t)) \dd t - \frac1\epl p'\left(\frac{X^\epl(t)}\epl\right) \dd t + \sqrt{2\sigma} \dd W(t),
\end{equation}
where $V \colon \R\to\R^N$ is defined as $V(x) = (V_1(x), V_2(x), \ldots, V_N(x))^\top$, the derivative $V'$ is computed component-wise, and $\alpha = (\alpha_1, \alpha_2, \ldots, \alpha_N)^\top$. Let us assume that we are exposed to a continuous-time stream of data $X^\epl$ solution of \eqref{eq:SDE_MS_Lang} for a finite time horizon $T$. Moreover, let us assume that the periodic function $p$, as well as the scale-separation parameter $\epl$, the drift coefficients $\{\alpha_i\}_{i=1}^N$, and the diffusion coefficient $\sigma$ are unknown. Conversely, we assume that the functions $\{V_i\}_{i=1}^N$ are known. In the semi-parametric framework, indeed, these functions often coincide with the first $N$ elements of the basis of some appropriate function space. Our goal is then to infer the effective drift coefficient $A = (A_1, A_2, \ldots, A_N)^\top \in \R^N$ and diffusion coefficient $\Sigma > 0$ that define the single-scale dynamics
\begin{equation}\label{eq:SDE_Hom_Lang_d1}
\d X^0(t) = -A \cdot V'(X^0(t)) \dd t + \sqrt{2\Sigma} \dd W(t).
\end{equation}
We consider the inferred coefficients to be \textit{asymptotically unbiased} if they converge to the true effective coefficients in the homogenization limit $\epl \to 0$ and for infinite data, i.e., for $T \to \infty$. Since the coarse-grained dynamics are inferred from data instead of being computed using the homogenization formulas, we are in the setting of \textit{data-driven homogenization}.

Let us first consider the drift coefficient. Applying Girsanov change of measure formula to the effective dynamics (see e.g. \cite{LiS01,LiS01b,Pav14,PSV09}) yields the likelihood function
\begin{equation}
L_T\left(X \mid A\right) = \exp\left(-\frac{I_T\left(X \mid A\right)}{2\Sigma}\right),
\end{equation}
where $X = (X(t), 0 \leq t \leq T)$ is a continuous-time stream of data and where
\begin{equation}
I_T\left(X \mid A\right) = \int_0^T A \cdot V'(X(t)) \dd X(t) + \frac12 \int_0^T (A \cdot V'(X(t)))^2 \dd t.
\end{equation}
Minimizing the function $I_T$ with respect to $A$ yields the maximum likelihood estimator (MLE) $\widehat A(X, T)$, which is defined as the solution of the linear system
\begin{equation} \label{eq:MLE}
-M(X, T) \widehat A_{\mathrm{MLE}}(X, T) = v(X, T),
\end{equation}
where
\begin{equation}
M(X, T) \defeq \frac1T \int_0^T V'(X(t)) \otimes V'(X(t)) \dd t, \quad v(X, T) \defeq \frac1T \int_0^T V'(X(t)) \dd X(t).
\end{equation}
The standard estimator of the diffusion coefficient $\Sigma$ given the stream of data $X$ is obtained by computing the quadratic variation $\langle X \rangle_T$ of the path $X$ and by defining
\begin{equation}
\widehat \Sigma(X, T) = \frac{\langle X \rangle_T}{2T} .
\end{equation}
We remark that in case the data $X$ would originate from the model \eqref{eq:SDE_Hom_Lang_d1}, we would have by definition $\widehat \Sigma(X, T) = \Sigma$. 

Let us now consider the stream of data $X = X^\epl$, i.e., the framework of data-driven homogenization. In this case, both the estimators $\widehat A(X^\epl, T)$ and $\widehat \Sigma(X^\epl, T)$ for the effective drift and diffusion coefficients are not asymptotically unbiased. In particular, it holds by \cite[Theorem 3.4]{PaS07}
\begin{equation} \label{eq:failure_classic}
\begin{aligned}
&\lim_{\epl \to 0}\lim_{T\to\infty} \widehat A_{\mathrm{MLE}}(X^\epl, T) = \alpha, \quad \text{a.s.}, \\
&\widehat \Sigma(X^\epl, T) = \sigma,
\end{aligned}
\end{equation}
where $\alpha$ and $\sigma$ are the coefficients of the multiscale dynamics \eqref{eq:SDE_MS_Lang_d1}. In this setting, the standard estimators fail and it is necessary to employ homogenization-informed techniques to infer the effective equation. 

\begin{remark} In case $\epl$ is known and due to \eqref{eq:failure_classic}, it would be possible to infer directly the full multiscale model \eqref{eq:SDE_Hom_Lang_d1} employing a periodic parametrisation of the function $p$. We argue that this would be less useful, at least for predictive purposes, than estimating directly the effective model. Indeed, numerical integration of \eqref{eq:SDE_MS_Lang_d1} is possible only by choosing critically small time step for most numerical schemes, which in turn yields dramatically high computational cost.
\end{remark}

\subsection{Literature Review}

The literature on statistical inference of stochastic models modelled by SDEs is vast. Introductory references on the topic are \cite{Bis08,Kut04,BaP80,Pra99}. The articles \cite{PSV09,PSZ13} give a complete overview on recent frequentist and Bayesian perspectives on the topic. A series of methods have been proposed in recent years for multiscale models, in different settings and with different purposes. We refer the reader to \cite{PPS12} for a recent survey, and summarize the approaches which are related to the one presented in this article in the following.

The focus of the early papers \cite{PaS07,PPS09} are simple multiscale models of molecular dynamics. In these works, the authors show that multiscale continuous-time data should be subsampled in order to infer single-scale equations which are effective in the sense of the theory of homogenization. In particular, let $\delta > 0$ be the subsampling rate, chosen for simplicity such that $T = n\delta$ for a positive integer $n$. The subsampled drift estimator is computed with a Euler--Maruyama-type discretisation of the MLE with spacing $\delta$, i.e.,
\begin{equation} \label{eq:DriftEstimatorSub}
-M^\delta_{\mathrm{sub}}(X^\epl, T) \widehat A^\delta_{\mathrm{sub}}(X^\epl, T) = v^\delta_{\mathrm{sub}}(X^\epl, T),
\end{equation}
where
\begin{equation}
\begin{aligned}
&M^\delta_{\mathrm{sub}}(X^\epl, T) \defeq \frac{\delta}T \sum_{i=0}^{n-1} V'(X^\epl(i\delta)) \otimes V'(X^\epl(i\delta)), \\
&v^\delta_{\mathrm{sub}}(X^\epl, T) \defeq \frac1T \sum_{i=0}^{n-1} V'(X^\epl(i\delta)) \left(X^\epl((i+1)\delta) - X^\epl(i\delta)\right).
\end{aligned}
\end{equation}
For the diffusion coefficient, the same subsampling and discretization procedure yields the estimator
\begin{equation}\label{eq:DriftSubsampling}
\widehat \Sigma_{\mathrm{sub}}^\delta = \frac{1}{2T} \sum_{i=0}^{n-1} \left(X^\epl((i+1)\delta) - X^\epl(i\delta)\right)^2.
\end{equation}
The subsampling rate which guarantees asymptotically unbiased estimators lays between the time scales of the multiscale and the effective models, i.e., for $\delta = \epl^\zeta$, with $\zeta \in (0, 1)$ \cite[Theorems 3.5 and 3.6]{PaS07}. The optimal subsampling rate is conjectured in \cite{PaS07} to be $\delta = \epl^{2/3}$. The necessity of subsampling the data have been further observed in \cite{ABT10,ABT11,ABJ13} for unobservable multiscale processes. In econometrics, the works \cite{AiJ14,AMZ05,ZMA05,OSP10} show that subsampling the data is indispensable in order to infer integrated volatility in the presence of market-microstructure noise. The disadvantages of subsampling are mainly two. First, it has been demonstrated numerically \cite{PaS07,AGP21} that inference results based on subsampling highly depend on $\delta$ for $\epl > 0$ and finite $T$. Second, knowledge of the scale-separation parameter $\epl$ is necessary to build asymptotically unbiased estimators, which in practice could be a severe limitation.

The methodologies proposed in \cite{KPK13,KKP15} and further applied in \cite{KPP15} for model selection with paleoclimatic and biological data bypass subsampling by exploiting a martingale property of the likelihood function, and by computing appropriate estimators for conditional expectations. The estimators proposed in \cite{KPK13} are not well posed on a single trajectory, which is overcome by averaging over a set of short trajectories. For \cite{KKP15}, estimators are obtained through a computationally expensive procedure employed to approximate conditional  expectations via Nadarya--Watson techniques. Let us remark that, to our knowledge, unbiasedness for these estimators of the effective dynamics is not theoretically justified and is just conjectured in \cite{KPK13}. In \cite{KKP15}, theoretical analysis is restricted to the one-dimensional case and when the effective dynamics is of the Ornstein--Uhlenbeck type.

Another alternative to subsampling has been proposed in \cite{AGP21}. In this work, maximum likelihood and Bayesian estimators for the effective dynamics are modified by smoothening the continuous-time data with a low-pass exponential filter. Under the assumptions of clear separation between the fast and the slow time scales, and of periodicity for the fast-scale drift function, asymptotic unbiasedness of the estimators is shown. An enhanced robustness with respect to subsampling is demonstrated via numerical experiments. The setting presented in our current paper closely relates to the work \cite{AGP21}, which we simplify here without any loss of accuracy.

In their series of works \cite{GaS17,GaS18,SpC13}, the authors propose estimators for the parameters of a multiscale SDE. The setting is similar to the one studied in this paper, with the difference that the stochastic dynamics are driven by a small noise, which vanishes in the homogenization limit. Hence, the effective equation is in this case a deterministic dynamical system. Theoretical difficulties are due in this framework to the likelihood function induced by the effective dynamics, which is singular. Closely related work is \cite{PaS14}, where estimation of multiscale stochastic dynamics is obtained by dimensionality reduction of an appropriate posterior distribution.

In the case of discrete-time data, it is possible to employ information based on the eigenpairs of the generator of the dynamics to obtain asymptotically unbiased estimators. In \cite{KeS99} and in the series of works \cite{CrV06,CrV06b}, the authors develop the theory of this class of spectral estimators for single-scale SDEs, which is further applied to the multiscale case in \cite{CrV11}. Recently, the technique based on filtered data of \cite{AGP21} has been combined with eigenfuction martingale estimating functions for inferring effective dynamics from discrete-time data from the full model in \cite{APZ22}.

The recent work \cite{CNN22} deals with inference of a similar multiscale equation with Kalman filtering methodologies. In this work, though, the authors focus on retrieving the coefficients of the multiscale dynamics given misspecified data from the reduced model, while we are interested in the opposite direction.

In the related field of multiscale partial differential equations (PDEs), the issue of model misspecification due to the mismatch of multiscale data and homogenized models has been studied in \cite{NPS12} and in the series of works \cite{AbD19,AbD20,AGZ20}.

\subsection{Our Contributions}

In this paper, we build on the techniques introduced in our previous work \cite{AGP21} and design a new class of estimators for effective diffusions based on moving averages. The methodologies we introduce here are easy to implement, computationally cheap, robust, and unbiased with respect to the theory of homogenization. Furthermore, we complete the numerical analysis of \cite{AGP21} by testing our method against a complex instance of multi-dimensional Langevin equations.

The basic idea underlying both our previous work \cite{AGP21} and our manuscript is similar: in both works we propose to smoothen the data through some filtering kernel, and to compute modified estimators which employ -- in some appropriate fashion -- the filtered data. Nevertheless, the theoretical and numerical analysis presented in this paper extend relevantly our previous work and have their own originality. In particular:
\begin{enumerate}[leftmargin=20pt]
	\item Numerical experiments show that the straightforward application of a moving average to the data yields extremely robust inference of effective dynamics when presented with multiscale data. The accuracy of the inference procedure is remarkably higher than previously existing techniques. 
	\item We propose a novel straightforward estimator for the effective diffusion coefficient, which was not previously analysed in our previous work, or elsewhere in the literature.
	\item Implementing our methodology does not require prior knowledge of the scale-separation parameter, and is simple and efficient even for data and parameters in multiple dimensions.
	\item We present an original analysis of asymptotic unbiasedness based on the ergodic properties of an appropriate system of stochastic delay differential equations (SDDEs). To our knowledge, this analysis is a novel theoretical contribution to the literature of statistical inference for SDEs.
\end{enumerate} 

\subsection{Outline}

The remainder of this paper is organized as follows. In \cref{sec:Filter} we present our methodology and introduce the main results of unbiasedness for the estimators based on filtered data. We then present numerical experiments in \cref{sec:NumExp} corroborating our theoretical findings and showing how to apply our methodology to a complex multi-dimensional scenario. \cref{sec:Proof} is dedicated to the proof of unbiasedness for our estimators, and in \cref{sec:Conclusion} we draw our conclusions and present some possible future directions of research.

\section{The Filtered Data Approach}\label{sec:Filter}

In this work, we propose an alternative to subsampling for preprocessing the data and infer effective dynamics. In particular, we smoothen the data with a continuous moving average, identified by the kernel
\begin{equation} \label{eq:filter_ma}
k_{\mathrm{ma}}^\delta(r) = \frac1\delta \chi_{[0,\delta]}(r),
\end{equation}
where $\chi_{[a,b]}$ denotes the indicator function of the interval $[a,b]$ for real numbers $a < b$, and where $\delta > 0$ is the size of the moving average window. We then take a time convolution of the kernel and the data and obtain the filtered trajectory 
\begin{equation}\label{eq:FilteredData}
Z_{\mathrm{ma}}^{\delta,\epl}(t) = \int_0^t k_{\mathrm{ma}}^\delta(t-s) X^\epl(s) \dd s = \frac1\delta \int_{t-\delta}^t X^\epl(s) \dd s, \qquad t \ge \delta.
\end{equation}
The process $Z_{\mathrm{ma}}^{\delta,\epl}$ is fully defined on the interval $t \in [0,T]$ by defining for small times $t \leq \delta$
\begin{equation}
Z_{\mathrm{ma}}^{\delta,\epl}(t) = \frac1t \int_0^t X^\epl(s) \dd s, \qquad 0 \leq t < \delta. 
\end{equation}
We remark that with this choice the process $Z_{\mathrm{ma}}^{\delta,\epl}$ is continuous on $[0, T]$, and can be seen as a smoothed version of the original trajectory where the fast oscillations are damped. 

The idea of smoothening the data with a low-pass filter has already been introduced in our previous work \cite{AGP21}. Instead of a moving average, we consider in \cite{AGP21} an exponential filtering kernel of the form
\begin{equation}\label{eq:filter}
k_{\mathrm{exp}}^{\delta,\beta}(r) = \frac{C_\beta}{\delta^{1/\beta}} \exp\left(-\frac{r^\beta}{\delta}\right), \qquad C_\beta =  \frac{\beta}{\Gamma(1/\beta)},
\end{equation}
where the constant $C_\beta$ normalizes the kernel in the sense
\begin{equation}
\int_0^\infty k_{\mathrm{exp}}^{\delta,\beta}(r) \dd r = 1.
\end{equation}
Filtered data $Z_{\mathrm{exp}}^{\delta,\beta,\epl}(t)$ are then obtained similarly to \eqref{eq:FilteredData} by taking a time convolution. We remark that the parameters $\delta$ and $\beta$ in \eqref{eq:filter} have two different roles. In particular, we have that $\beta$ is a \textit{shape parameter}, and $\delta$ is the \textit{filtering width}. 
The approaches presented here and in our previous work are closely related. In fact, it is simple to deduce that for almost every $r \geq 0$
\begin{equation}\label{eq:filter_new}
\lim_{\beta \to \infty} k_{\mathrm{exp}}^{\delta,\beta}(r) = \chi_{[0,1]}(r),
\end{equation}
independently of $\delta$. We remark that the theoretical analysis in \cite{AGP21} is restricted to the case where the shape parameter $\beta = 1$, despite numerical experiments suggest that choosing $\beta > 1$ yields better estimators. Studying the moving average kernel and due to \eqref{eq:filter_new} therefore partially fills the theoretical gap of our previous work.

\subsection{Estimating the Drift Coefficient}

We now present how in practice one employs filtered data to obtain asymptotically unbiased estimators of the effective drift coefficient. The main idea is modifying the classical MLE \eqref{eq:MLE} by replacing one occurrence of the original process $X^\epl$ with the filtered process $Z_{\mathrm{ma}}^{\delta,\epl}$ both in $M$ and $v$. In particular, the drift estimator $\widehat A_{\mathrm{ma}}^\delta(X^\epl, T)$ is defined as the solution of the linear system
\begin{equation} \label{eq:DriftEstimator}
-M_{\mathrm{ma}}^\delta(X^\epl, T) \widehat A_{\mathrm{ma}}^\delta(X^\epl, T) = v_{\mathrm{ma}}^\delta(X^\epl, T),
\end{equation}
where
\begin{equation}
\begin{aligned}
&M_{\mathrm{ma}}^\delta(X^\epl, T) \defeq \frac1T \int_0^T V'(Z_{\mathrm{ma}}^{\delta,\epl}(t)) \otimes V'(X^\epl(t)) \dd t, \\
&v_{\mathrm{ma}}^\delta(X^\epl, T) \defeq \frac1T \int_0^T V'(Z_{\mathrm{ma}}^{\delta,\epl}(t)) \dd X^\epl(t).
\end{aligned}
\end{equation}
An asymptotically unbiased estimator $\widehat A_{\mathrm{exp}}^{\delta,\beta}(X^\epl, T)$ of the drift coefficient based on exponential filtering kernels is obtained in the same way, by replacing $Z_{\mathrm{ma}}^{\delta, \epl}$ with $Z_{\mathrm{exp}}^{\delta, \beta, \epl}$, as shown in \cite{AGP21}
We remark that it is fundamental to keep $X^\epl(t)$ in the differential of the right-hand side $v(X^\epl, T)$, as discussed in \cite[Remark 3.7]{AGP21}. For the well-posedness of the estimator $\widehat A_{\mathrm{ma}}^\delta(X^\epl, T)$ the matrix $M_{\mathrm{ma}}^\delta(X^\epl, T)$ needs to be invertible. We take this property as an assumption in the theoretical analysis but we observe that it holds in practice in all the numerical experiments.

\subsection{Estimating the Diffusion Coefficient}

We now focus on inferring the effective diffusion coefficient and we propose two different estimators. The first one is given by 
\begin{equation} \label{eq:DiffusionEstimator}
\widehat \Sigma_{\mathrm{ma}}^\delta(X^\epl, T) = \frac1{\delta T} \int_\delta^T \left(X^\epl(t) - Z_{\mathrm{ma}}^{\delta,\epl}(t) \right) \left(X^\epl(t) - X^\epl(t-\delta)\right) \dd t,
\end{equation}
which is analogous to the diffusion estimator presented in \cite[Section 3]{AGP21}.

A different methodology for estimating the diffusion coefficient of the homogenized equation can be derived from the particular form of \eqref{eq:SDE_Hom_Lang_d1}. We know that $\Sigma = \sigma \mathcal K$ and therefore an estimation of $\Sigma$ can be obtained by first estimating $\sigma$ and $\mathcal K$. The former can be computed exactly due to \eqref{eq:failure_classic}, indeed we have
\begin{equation}
\sigma = \frac{\langle X^\epl \rangle_T}{2T},
\end{equation}
while for the latter we use the fact that $A = \alpha \mathcal K$. Given an estimator $\widehat A$ for the effective drift coefficient and since by \eqref{eq:failure_classic} we know that $\widehat \alpha = \widehat A_{\mathrm{MLE}}(X^\epl,T)$ approximates $\alpha$, we write
\begin{equation} \label{eq:system_diffusion}
\widehat \alpha \widehat{\mathcal K} = \widehat A,
\end{equation}
which is an overdetermined linear system with $N$ equations and one unknown, and where $\widehat{\mathcal K}$ denotes the estimator of the coefficient $\mathcal K$ which we aim to infer. The least squares solution to \eqref{eq:system_diffusion} is given by
\begin{equation}
\widehat{\mathcal K} = \frac{\widehat \alpha^\top \widehat A}{\widehat \alpha^\top \widehat \alpha}.
\end{equation}
Assuming that an estimator $\widehat A(X^\epl, T)$ of the effective drift coefficient has already been computed, the effective diffusion coefficient can then be estimated as
\begin{equation} \label{eq:DiffusionTildeEstimator}
\widetilde \Sigma(X^\epl,T) = \frac{\langle X^\epl \rangle_T (\widehat A_{\mathrm{MLE}}(X^\epl,T)^\top \widehat A(X^\epl,T))}{2T(\widehat A_{\mathrm{MLE}}(X^\epl,T)^\top \widehat A_{\mathrm{MLE}}(X^\epl,T))}.
\end{equation}

\begin{remark} In practice, the stream $X^\epl$ consists of high-frequency discrete data, and not of a continuous process. Hence, the filtered data and all our estimators are computed in practice with the usual appropriate discretizations. We notice that if $X^\epl$ consists of $n$ data points, the time complexity needed to compute the filtered trajectory $Z^{\delta,\epl}_{\mathrm{ma}}$ is of order $\mathcal O(n)$.
\end{remark}

\subsection{Statement of Asymptotic Unbiasedness Results}

In this section we present the main theoretical results of this work, i.e., the asymptotic unbiasedness of the proposed estimators. We first introduce the assumptions which will be employed in the analysis. In particular, we restrict the analysis to the same dissipative framework as \cite{AGP21,PaS07}.
\begin{assumption}\label{as:regularity} The functions $p$ and $V$ satisfy:
	\begin{enumerate}[leftmargin=1cm]
		\item\label{as:periodicity} $p \in \mathcal C^\infty(\R)$ and is $L$-periodic for some $L > 0$,
		\item\label{as:regularity_diss} $V_i \in \mathcal C^\infty(\R)$ for all $i=1, \ldots, N$ and are polynomially bounded from above and bounded from below. Moreover, the potential is dissipative, i.e., there exist $a,b > 0$ such that
		\begin{equation}
		-\alpha \cdot V'(x) x \leq a - bx^2,
		\end{equation} 
		\item\label{as:regularity_Lip} $V'$ and $V''$ are Lipschitz continuous, i.e., there exists a constant $C > 0$ such that
		\begin{equation}
		\norm{V'(x) - V'(y)} \leq C\abs{x - y}, \qquad \text{and} \qquad \norm{V''(x) - V''(y)} \leq C\abs{x - y},
		\end{equation} 
		where $\norm{\cdot}$ denotes the Euclidean norm, and the components $V'_i$ and $V''_i$ are polynomially bounded for all $i = 1, \ldots, N$.
	\end{enumerate}
\end{assumption}
\begin{remark}\cref{as:regularity} has the following consequences:
	\begin{enumerate}[leftmargin=1cm]
		\item \cref{as:regularity}\ref{as:periodicity} allows to employ the theory of periodic homogenization to conclude that \eqref{eq:SDE_Hom_Lang_d1} is effective for equation \eqref{eq:SDE_MS_Lang_d1}.
		\item \cref{as:regularity}\ref{as:regularity_diss} implies that the solutions of \eqref{eq:SDE_MS_Lang_d1} and \eqref{eq:SDE_Hom_Lang_d1} are geometrically ergodic, and thus the existence of unique invariant measures (see \cite[Propositions 5.1 and 5.2]{PaS07}).
		\item \cref{as:regularity}\ref{as:regularity_Lip} is technical, and guarantees sufficient regularity for our main results to hold.
	\end{enumerate}
\end{remark}

We first consider the drift estimator $\widehat A_{\mathrm{ma}}^\delta(X^\epl,T)$, which is asymptotically unbiased due to the following result.
\begin{theorem}\label{thm:Drift} Let $\widehat A_{\mathrm{ma}}^\delta(X^\epl, T)$ be defined in \eqref{eq:DriftEstimator} with $\delta$ independent of $\epl$ or $\delta = \epl^\zeta$ where $\zeta \in (0, 2)$. Under \cref{as:regularity} and if $M_{\mathrm{ma}}^\delta(X^\epl, T)$ is invertible, it holds
	\begin{equation}\label{eq:ThmDrift}
	\lim_{\epl\to 0} \lim_{T\to\infty} \widehat A_{\mathrm{ma}}^\delta(X^\epl, T) = A, \quad \text{a.s.},
	\end{equation}
	where $A$ is the drift coefficient of the homogenized equation \eqref{eq:SDE_Hom_Lang_d1}.
\end{theorem}

The following result of asymptotic unbiasedness holds for the estimator $\widehat \Sigma_{\mathrm{ma}}^\delta(X^\epl,T)$ of the effective diffusion.

\begin{theorem}\label{thm:Diffusion} Let $\widehat \Sigma_{\mathrm{ma}}^\delta(X^\epl, T)$ be defined in \eqref{eq:DiffusionEstimator} with $\delta = \epl^\zeta$ where $\zeta \in (0, 2)$. Under \cref{as:regularity}, it holds
	\begin{equation}
	\lim_{\epl\to 0} \lim_{T\to\infty} \widehat \Sigma_{\mathrm{ma}}^\delta(X^\epl, T) = \Sigma, \quad \text{a.s.},
	\end{equation}
	where $\Sigma$ is the diffusion coefficient of the homogenized equation \eqref{eq:SDE_Hom_Lang_d1}.
\end{theorem}

We notice from the statement of \cref{thm:Diffusion} that the main limitation of $\widehat \Sigma_{\mathrm{ma}}^\delta$ is that knowledge of the scale-separation parameter $\epl$ is necessary for unbiasedness. Conversely, the estimator $\widetilde \Sigma(X^\epl,T)$ of the effective diffusion can be asymptotically unbiased even if $\delta$ is independent of $\epl$. Indeed, unbiasedness of $\widetilde \Sigma(X^\epl, T)$ solely relies on the accuracy of the corresponding drift estimator, as shown by the following result.

\begin{theorem}\label{thm:Diffusion_tilde} Let $\widehat A$ be an asymptotically unbiased estimator of the effective drift coefficient, i.e.,
	\begin{equation} \label{eq:hypothesis_convergence_drift}
	\lim_{\epl \to 0} \lim_{T \to \infty} \widehat A(X^\epl,T) = A, \quad \text{a.s.},
	\end{equation}
	where $A$ is the drift coefficient of the homogenized equation \eqref{eq:SDE_Hom_Lang_d1}, and let $\widetilde \Sigma(X^\epl, T)$ be defined in \eqref{eq:DiffusionTildeEstimator}. Under \cref{as:regularity}, it holds
	\begin{equation}
	\lim_{\epl\to 0} \lim_{T\to\infty} \widetilde \Sigma(X^\epl, T) = \Sigma, \quad \text{a.s.},
	\end{equation}
	where $\Sigma$ is the diffusion coefficient of the homogenized equation \eqref{eq:SDE_Hom_Lang_d1}.
\end{theorem}

The following corollary is a direct consequence of \cref{thm:Diffusion_tilde}.

\begin{corollary} \label{cor:Diffusion_tilde}
	Let $\widetilde \Sigma_{\mathrm{ma}}^\delta(X^\epl,T)$, $\widetilde \Sigma_{\mathrm{exp}}^{\delta,\beta}(X^\epl,T)$, $\widetilde \Sigma_{\mathrm{sub}}^\delta(X^\epl,T)$ be defined by \eqref{eq:DiffusionTildeEstimator} with $\widehat A(X^\epl,T)$ replaced by $\widehat A_{\mathrm{ma}}^\delta(X^\epl,T)$, $\widehat A_{\mathrm{exp}}^{\delta,\beta}(X^\epl,T)$, and $\widehat A_{\mathrm{sub}}^\delta(X^\epl,T)$, respectively. If $\delta$ is independent of $\epl$ or $\delta = \epl^\zeta$ with $\zeta \in (0,2)$, then
	\begin{equation}
	\lim_{\epl \to 0} \lim_{T \to \infty} \widetilde \Sigma_{\mathrm{ma}}^\delta = \Sigma, \quad \text{a.s.}, \qquad \lim_{\epl \to 0} \lim_{T \to \infty} \widetilde \Sigma_{\mathrm{exp}}^{\delta,\beta} = \Sigma, \quad \text{a.s.},
	\end{equation}
	and if  $\delta = \epl^\zeta$ with $\zeta \in (0,1)$
	\begin{equation}
	\lim_{\epl \to 0} \lim_{T \to \infty} \widetilde \Sigma_{\mathrm{sub}}^\delta = \Sigma, \quad \text{a.s.},
	\end{equation}
	where $\Sigma$ is the diffusion coefficient of the homogenized equation \eqref{eq:SDE_Hom_Lang_d1}.
\end{corollary}

Two remarks are due.
\begin{remark} In our approach, the scale-separation parameter $\epl$ need not be known. In particular, \cref{thm:Drift,cor:Diffusion_tilde} show that our estimators are asymptotically unbiased if $\delta$ is independent of $\epl$, i.e., the filtering width equals the speed of the homogenized process. Since $\epl$ is in general unknown in practice, this constitutes an advantage with respect to, e.g., subsampling, for which knowledge of $\epl$ is necessary. Moreover, as we demonstrate numerically in \cref{sec:NumExp}, modifying the filtering width has a weak impact on the inference results as long as $\delta \in [\epl, 1]$ when $T < \infty$ and $\epl > 0$.
\end{remark}

\begin{remark} Despite being more accurate in practice, as demonstrated by our numerical experiments, and not requiring knowledge of $\epl$, the estimator $\widetilde \Sigma_{\mathrm{ma}}^\delta$ is computationally more expensive to obtain than $\widehat \Sigma_{\mathrm{ma}}^\delta$. Indeed, computing $\widetilde \Sigma_{\mathrm{ma}}^\delta$ requires estimators for the parameters $\alpha$ and $\sigma$ of the multiscale equation \eqref{eq:SDE_MS_Lang_d1}, as well as the drift coefficient $A$ of the homogenized model \eqref{eq:SDE_Hom_Lang_d1}. Hence, if a very accurate estimate for the whole effective equation is needed, we recommend to employ $\widetilde \Sigma_{\mathrm{ma}}^\delta$, while $\widehat \Sigma_{\mathrm{ma}}^\delta$ can be used in case only the diffusion coefficient is required.
\end{remark}

The proof of \cref{thm:Drift,thm:Diffusion,thm:Diffusion_tilde} is obtained by applying techniques similar to the ones employed in \cite{PaS07,AGP21,APZ22}, and is presented in detail in \cref{sec:Proof}.

\subsection{The Multi-Dimensional Case}

In this section we report the expression of the estimators for higher dimensions. Indeed, the choice $d = 1$ in the remainder of the paper is made only for economy of notation, and for clarity. Moreover, the proofs of \cref{thm:Drift,thm:Diffusion,thm:Diffusion_tilde} would be conceptually unchanged by considering $d > 1$, up to tedious technical details. Let $\nabla V \colon \R^d \to \R^{dN}$ be the vector obtained by stacking the gradients of the components $V_i$, $i = 1, \dots, N$, of the slow-scale potential, i.e.,
\begin{equation}
\nabla V(x) = \begin{pmatrix} \nabla V_1(x)^\top & \cdots & \nabla V_N(x)^\top \end{pmatrix}^\top.
\end{equation}
Then, the drift estimator for the block matrix $A \in \R^{dN \times d}$ given by
\begin{equation}
A = \begin{pmatrix} A_1^\top & \cdots & A_N^\top \end{pmatrix}^\top,
\end{equation}
is the solution of the linear system
\begin{equation}
-M_{\mathrm{ma}}^\delta(X^\epl, T) \widehat A_{\mathrm{ma}}^\delta(X^\epl, T) = v_{\mathrm{ma}}^\delta(X^\epl, T),
\end{equation}
where the matrices $M_{\mathrm{ma}}^\delta(X^\epl, T) \in \R^{dN \times dN}$ and $v_{\mathrm{ma}}^\delta(X^\epl, T) \in \R^{dN \times d}$ are defined as
\begin{equation}
\begin{aligned}
M_{\mathrm{ma}}^\delta(X^\epl, T) &\defeq \frac1T \int_0^T \nabla V(Z_{\mathrm{ma}}^{\delta,\epl}(t)) \otimes \nabla V(X^\epl(t)) \dd t, \\ v_{\mathrm{ma}}^\delta(X^\epl, T) &\defeq \frac1T \int_0^T \nabla V(Z_{\mathrm{ma}}^{\delta,\epl}(t)) \otimes \dd X^\epl(t).
\end{aligned}
\end{equation}
Concerning the diffusion coefficient, it is natural to impose that it is symmetric and positive definite, so that the square root $\sqrt{2\Sigma}$ is well defined. For this reason, we define the estimator $\widehat \Sigma_{\mathrm{ma}}^\delta(X^\epl, T)$ for the effective diffusion coefficient $\Sigma \in \R^{d \times d}$ as
\begin{equation}
\widehat \Sigma_{\mathrm{ma}}^\delta(X^\epl, T) = \mathcal S \left( \frac1{\delta T} \int_\delta^T \left(X^\epl(t) - Z_{\mathrm{ma}}^{\delta,\epl}(t)\right) \otimes \left(X^\epl(t) - X^\epl(t-\delta)\right) \dd t \right),
\end{equation}
where $\mathcal S(B)$ denotes the symmetric part of a matrix $B \in \R^{d \times d}$, i.e., $\mathcal S(B) = (B + B^\top)/2$. Remark that positive definiteness is not guaranteed for $\widehat \Sigma_{\mathrm{ma}}^\delta(X^\epl, T)$, but due to asymptotic unbiasedness it is natural to expect that for $T$ large enough and $\epl$ small enough the estimator $\widehat \Sigma_{\mathrm{ma}}^\delta(X^\epl, T)$ is positive definite as its limit. For the estimators of the form $\widetilde \Sigma(X^\epl,T)$, we first estimate the homogenization matrix as
\begin{equation}
\widehat{\mathcal K} = \argmin_{\mathcal K \in \mathrm{Sym}^+_d} \norm{\widehat A_{\mathrm{MLE}}(X^\epl, T) \mathcal K - \widehat A(X^\epl, T)}_F,
\end{equation}
where $\mathrm{Sym}^+_d$ is the space of symmetric positive definite matrices of size $d\times d$, $\norm{\cdot}_F$ is the Frobenius norm, and 
$\widehat A(X^\epl,T)$ is any estimator of the effective drift coefficient $A \in \R^{dN \times d}$. It is simple to compute the minimum in practice by imposing $\mathcal K = LL^\top$, and then minimizing directly over $L \in \R^{d\times d}$. Then, we define
\begin{equation}
\widetilde \Sigma(X^\epl,T) = \frac{\langle X^\epl \rangle_T}{2T} \widehat{\mathcal K}.
\end{equation}
In this case, the estimator is symmetric and positive definite by construction.

As we present in the numerical experiment of \cref{sec:NumExp_2d}, our methodology based on moving average can be naturally and successfully applied to higher-dimensional SDEs.

\section{Numerical Experiments} \label{sec:NumExp}

In this section, we present a series of numerical experiments which have the twofold goal of validating our theoretical analysis, and of showcasing the effectiveness of our technique on challenging academic examples.

\subsection{Sensitivity Analysis with a Ornstein--Uhlenbeck Model} \label{sec:NumExp_OU}

\begin{figure}[t!]
	\centering
	\hspace{2cm}\includegraphics[]{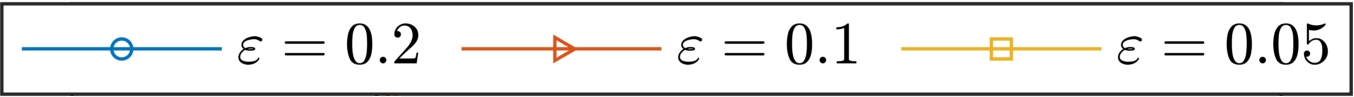}\vspace{0.25cm}
	\begin{tabular}{m{1.3cm}m{4cm}m{4cm}m{4cm}}		
		\vspace{-1.2cm}$\sigma=1$ & \includegraphics[]{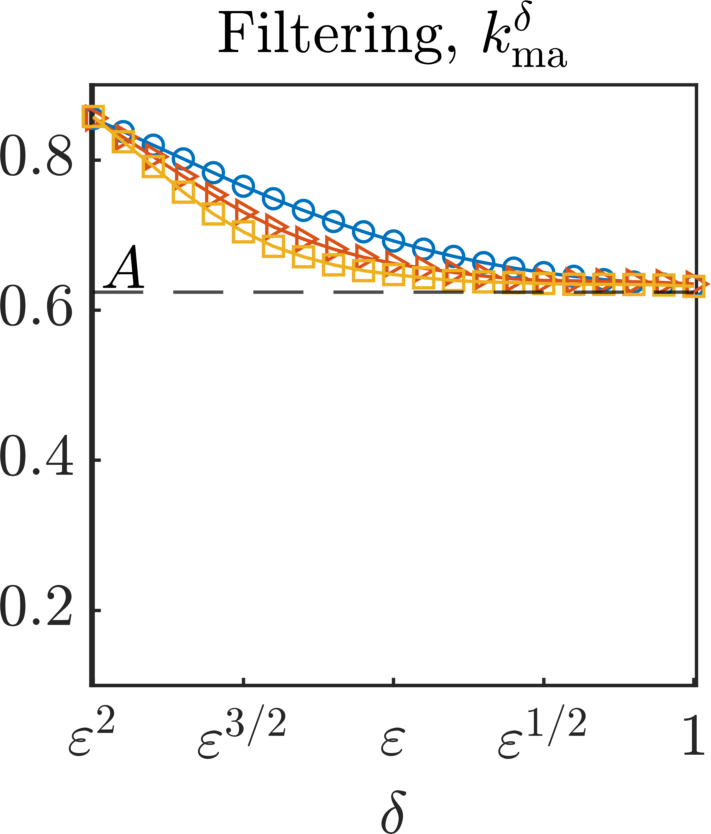} & \includegraphics[]{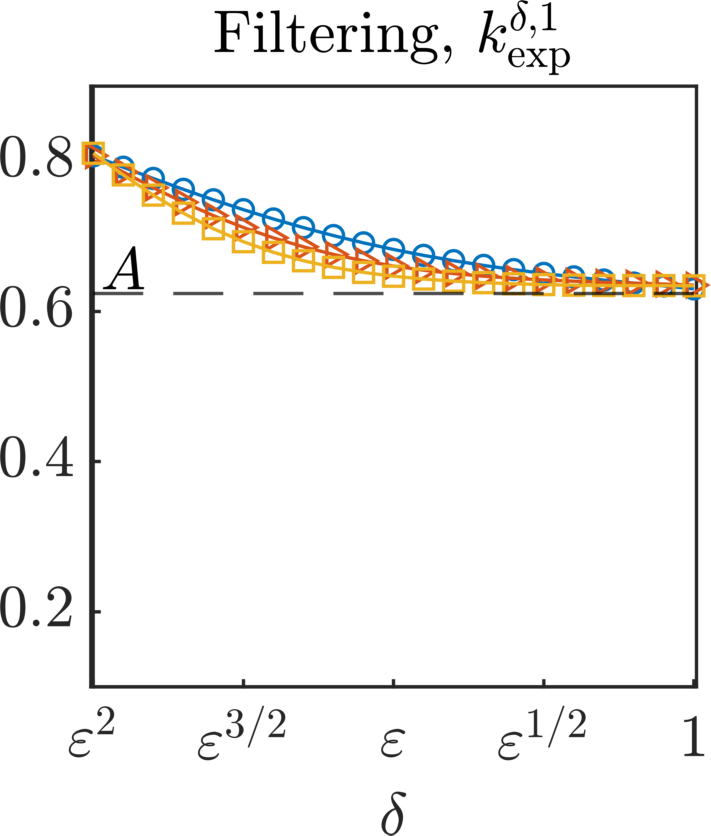}  & \includegraphics[]{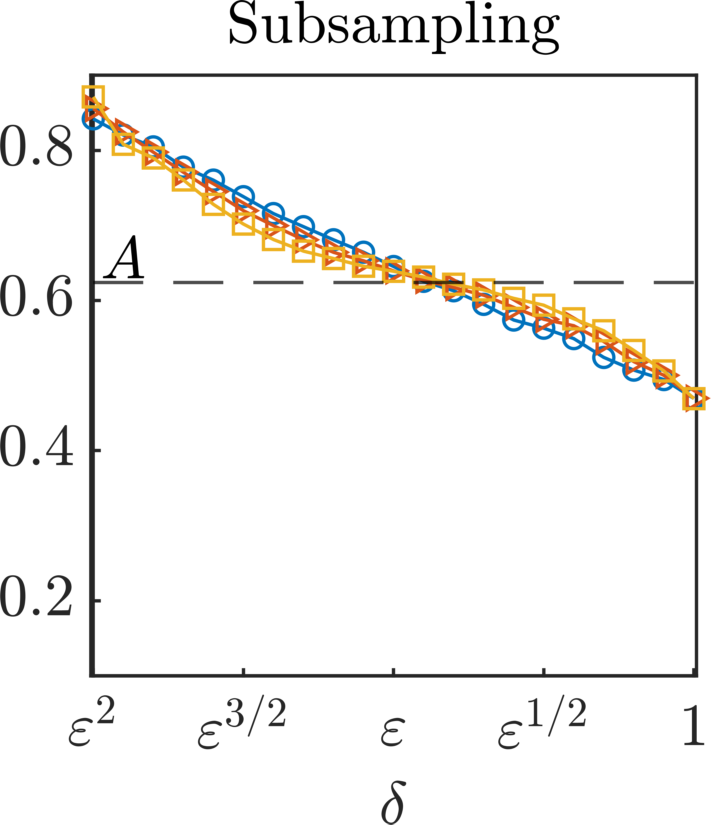} \\ 
		\vspace{-1.2cm}$\sigma=0.75$ & \includegraphics[]{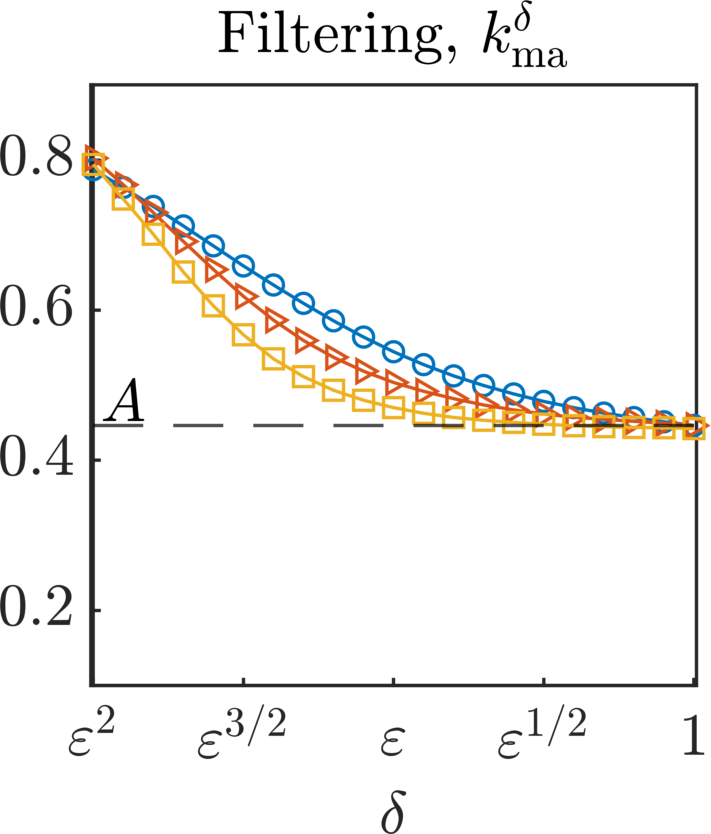} & \includegraphics[]{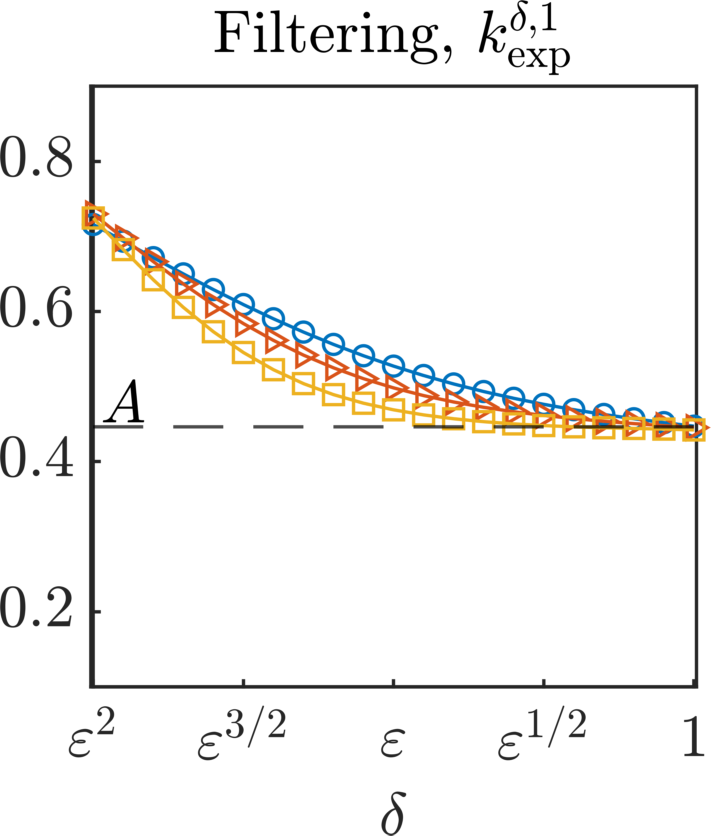}  & \includegraphics[]{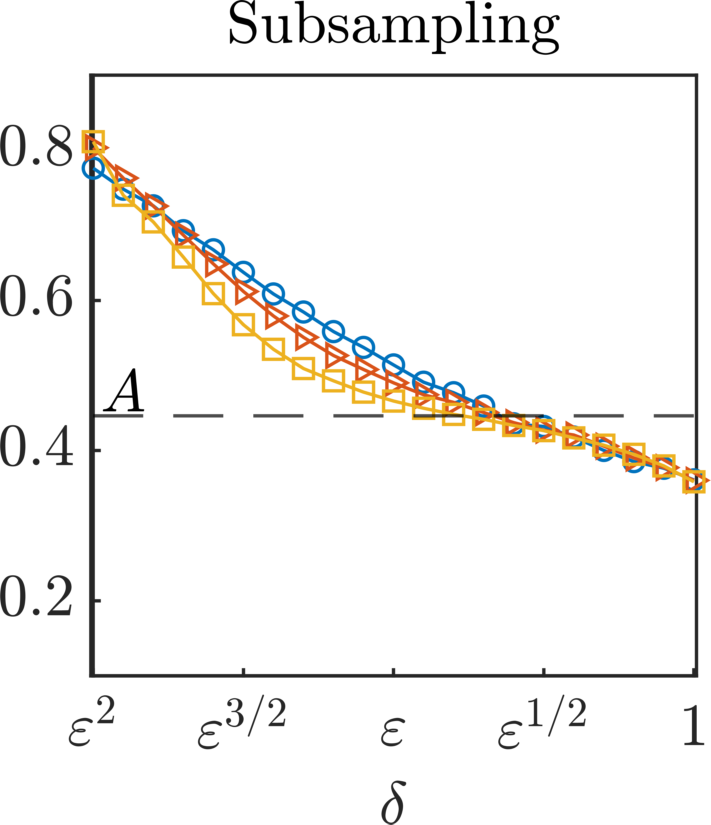} \\ 
		\vspace{-1.2cm}$\sigma=0.5$ & \includegraphics[]{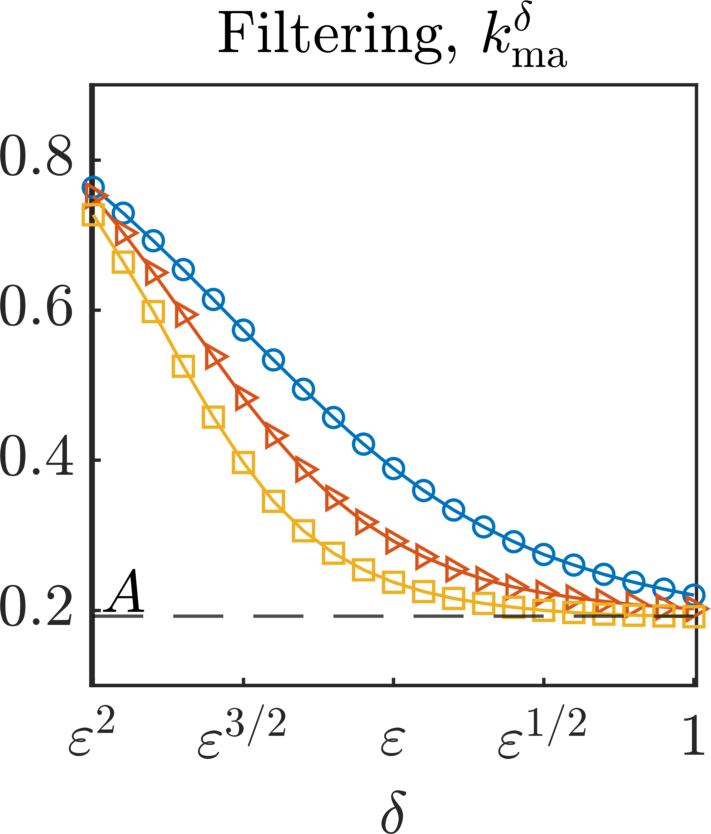} & \includegraphics[]{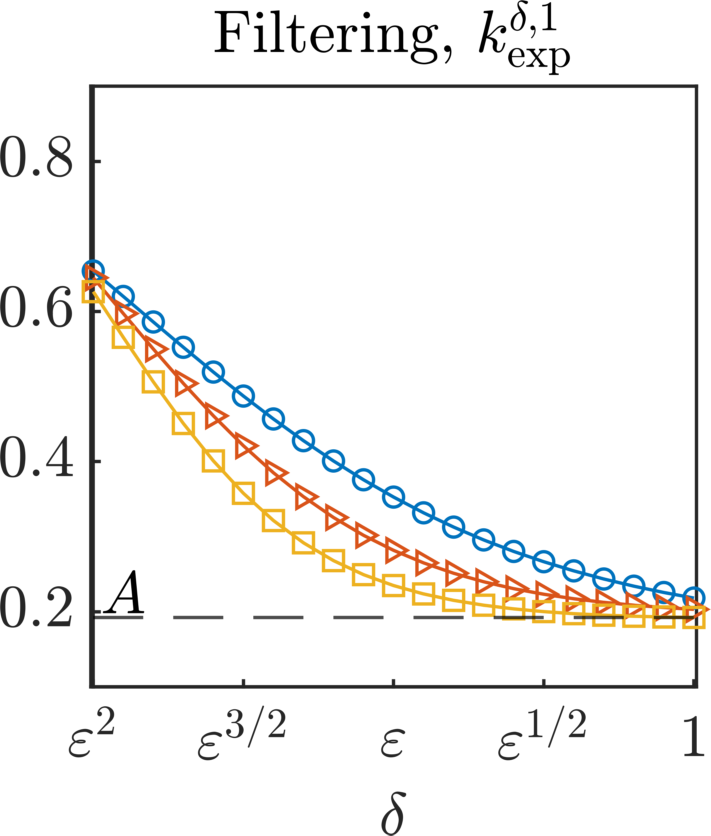}  & \includegraphics[]{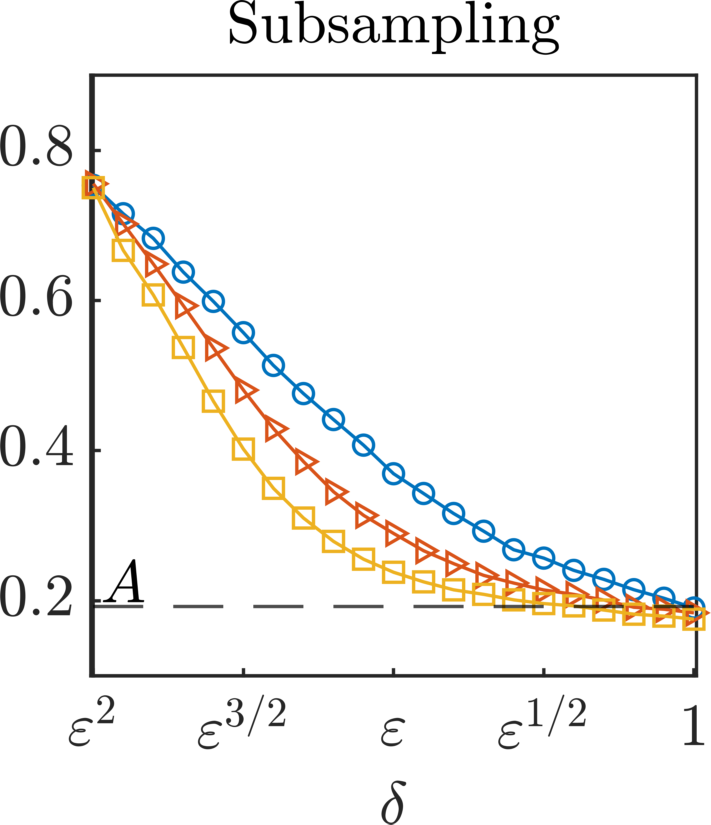} \\ 
	\end{tabular}
	
	\vspace{0.3cm}
	\includegraphics[]{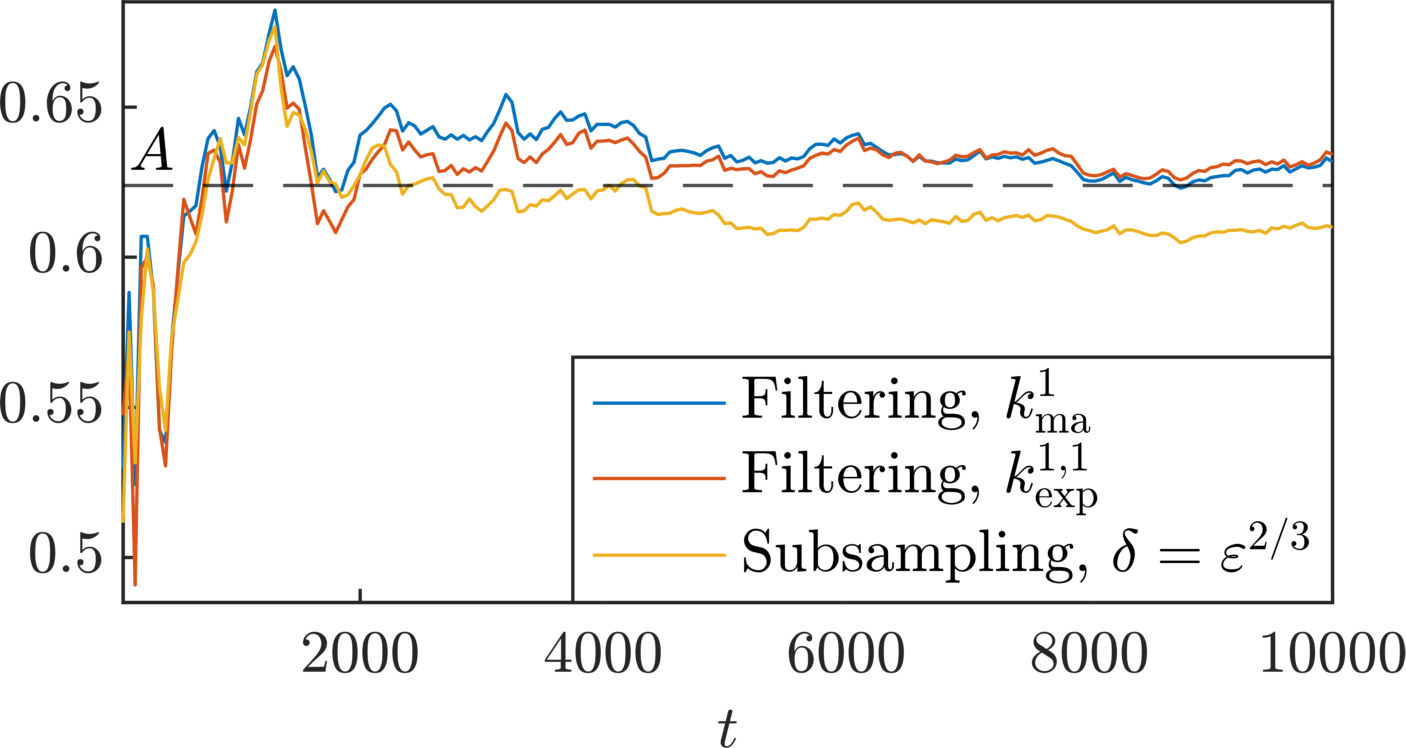}
	\caption{Estimation of the drift coefficient of an effective Ornstein--Uhlenbeck equation. Top: Numerical results at final time $T = 10^4$ for variable $\sigma = 1, 0.75, 0.5$ (per row), for $\epl = 0.2, 0.1, 0.05$ (blue, red, and yellow lines respectively), and for variable filtering/subsampling width $\delta$ (horizontal axis in all figures). Comparison between filtering with moving average and exponential filters (first two columns) and subsampling (last column). Bottom: Convergence with respect to $t \in [0, 10^4]$ of the two estimators based on filtered data with $\delta = 1$, and of the subsampling estimator with $\delta = \epl^{2/3}$, for fixed $\sigma = 1$ and $\epl = 0.05$. Remark: The legend on top is valid for all plots, except the last row.}
	\label{fig:Drift}
\end{figure}
\begin{figure}[t!]
	\centering
	\hspace{2cm}\includegraphics[]{Figures/legend_epsilon}\vspace{0.25cm}
	\begin{tabular}{m{1.3cm}m{4cm}m{4cm}m{4cm}}		
		\vspace{-1.2cm}$\sigma=1$ & \includegraphics[]{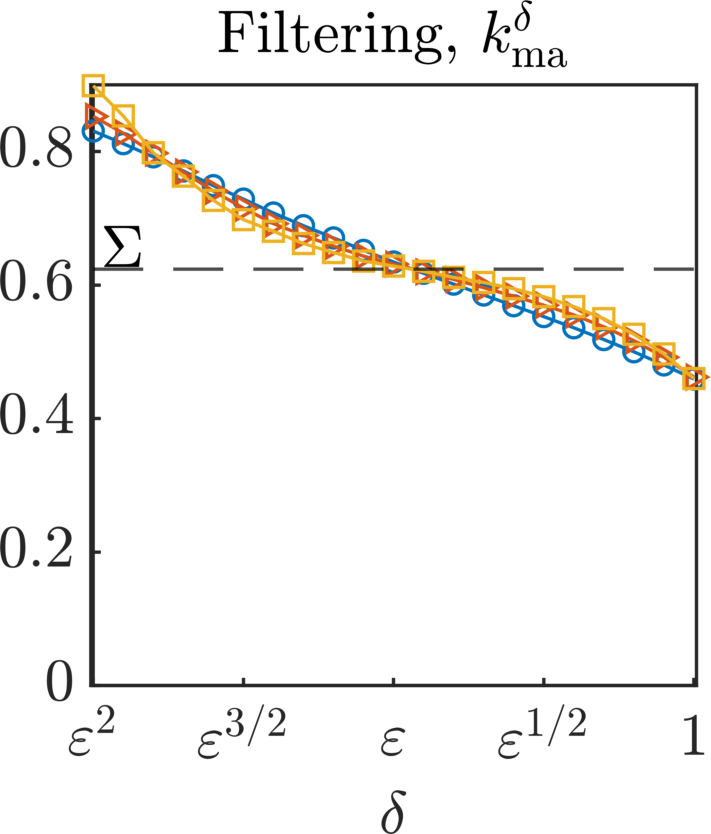} & \includegraphics[]{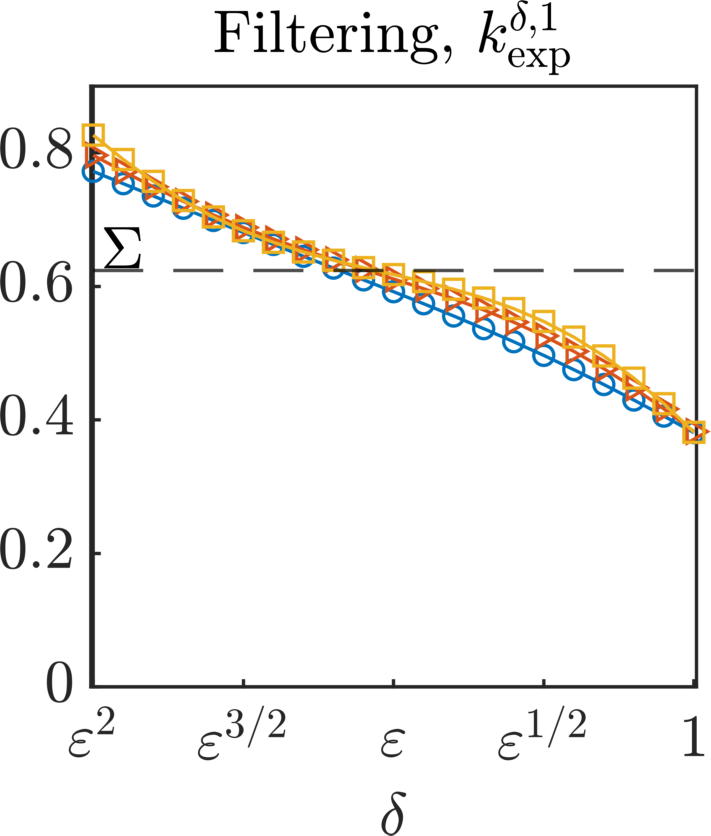}  & \includegraphics[]{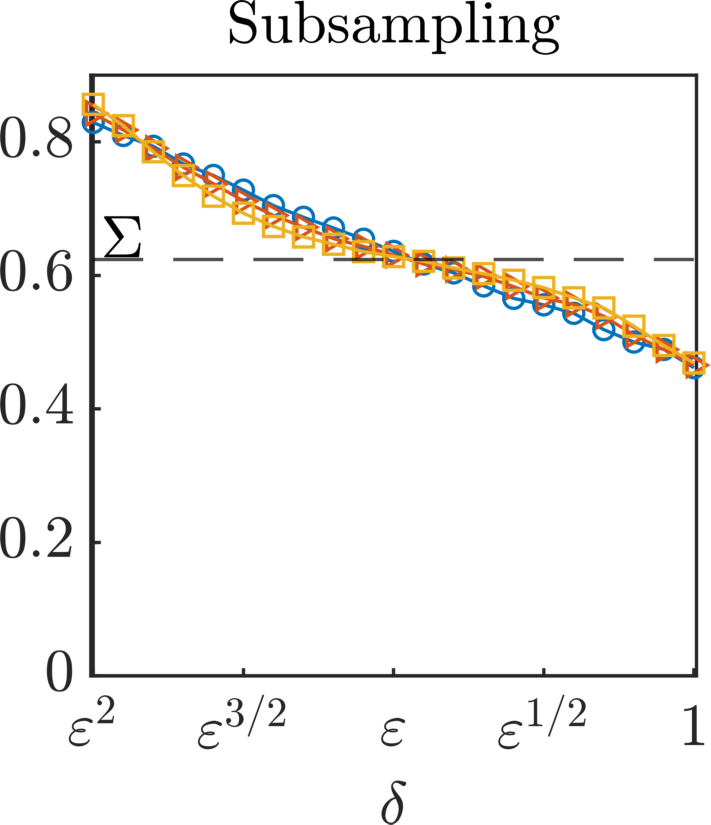} \\ 
		\vspace{-1.2cm}$\sigma=0.75$ & \includegraphics[]{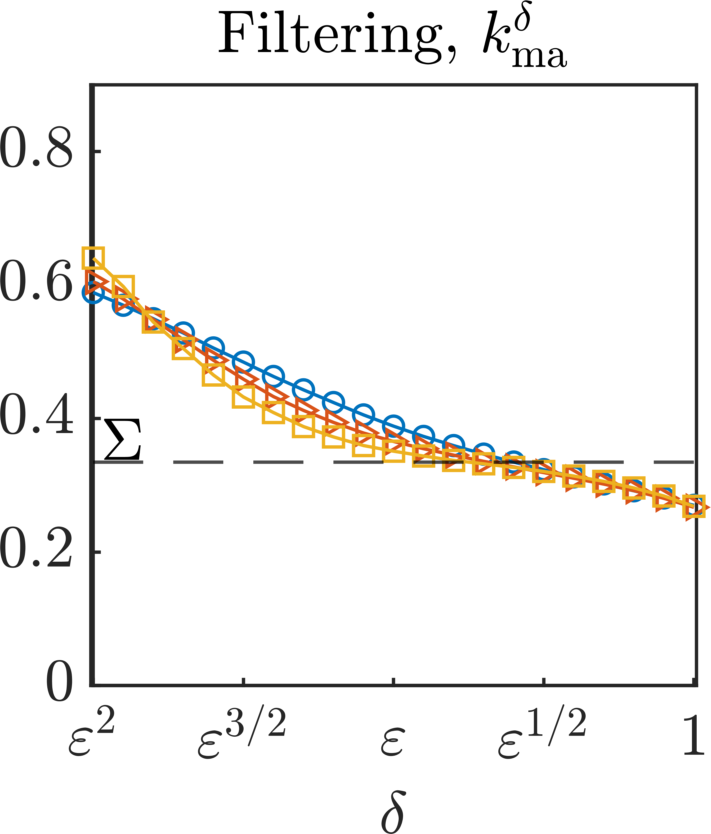} & \includegraphics[]{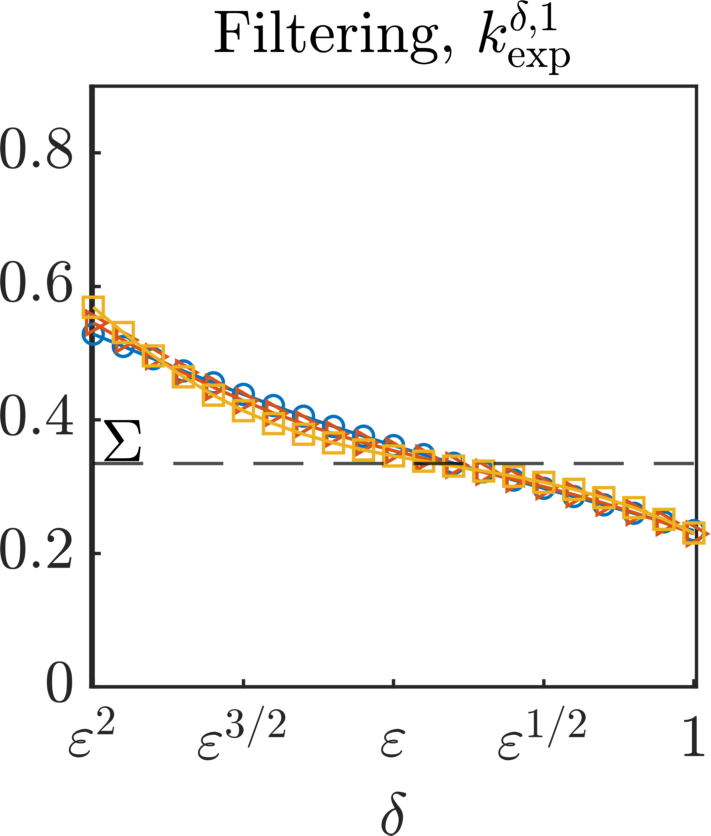}  & \includegraphics[]{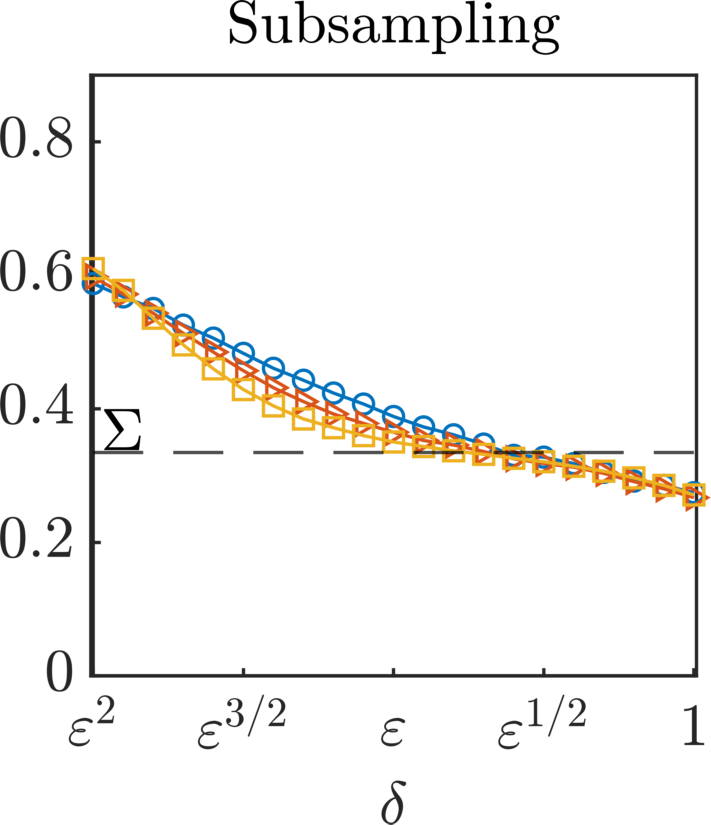} \\ 
		\vspace{-1.2cm}$\sigma=0.5$ & \includegraphics[]{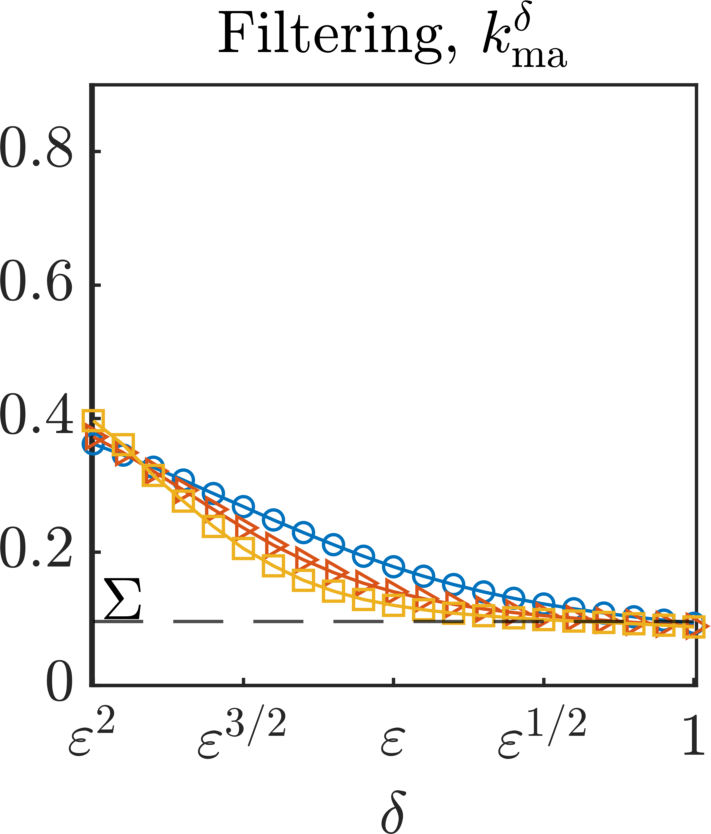} & \includegraphics[]{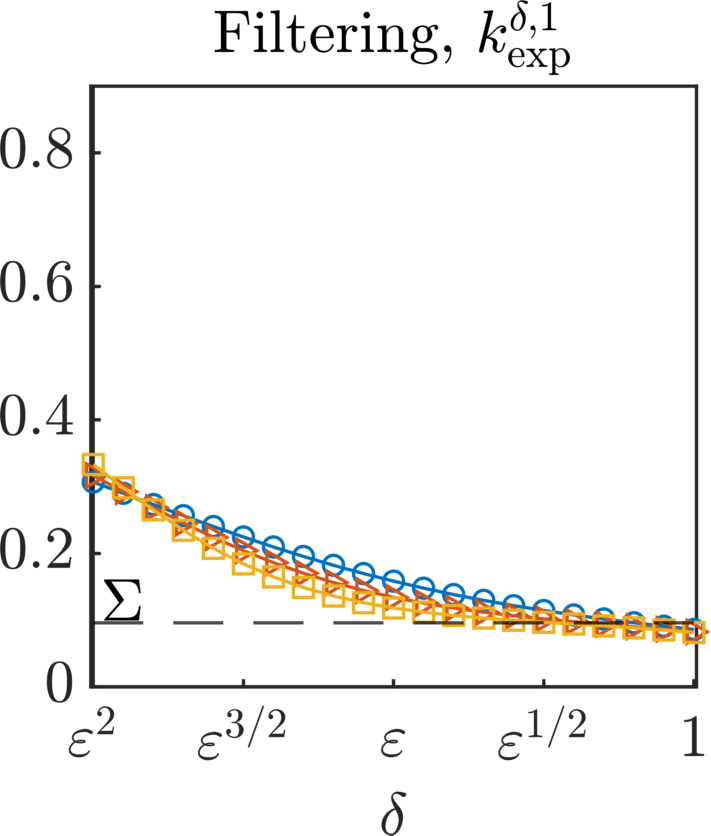}  & \includegraphics[]{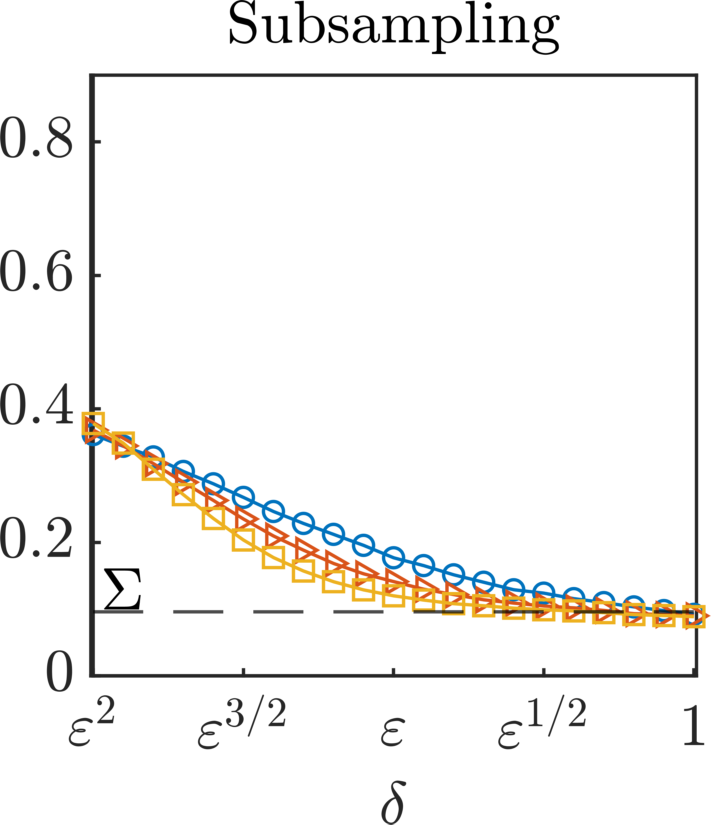} \\ 
	\end{tabular}
	
	\vspace{0.25cm}
	\includegraphics[]{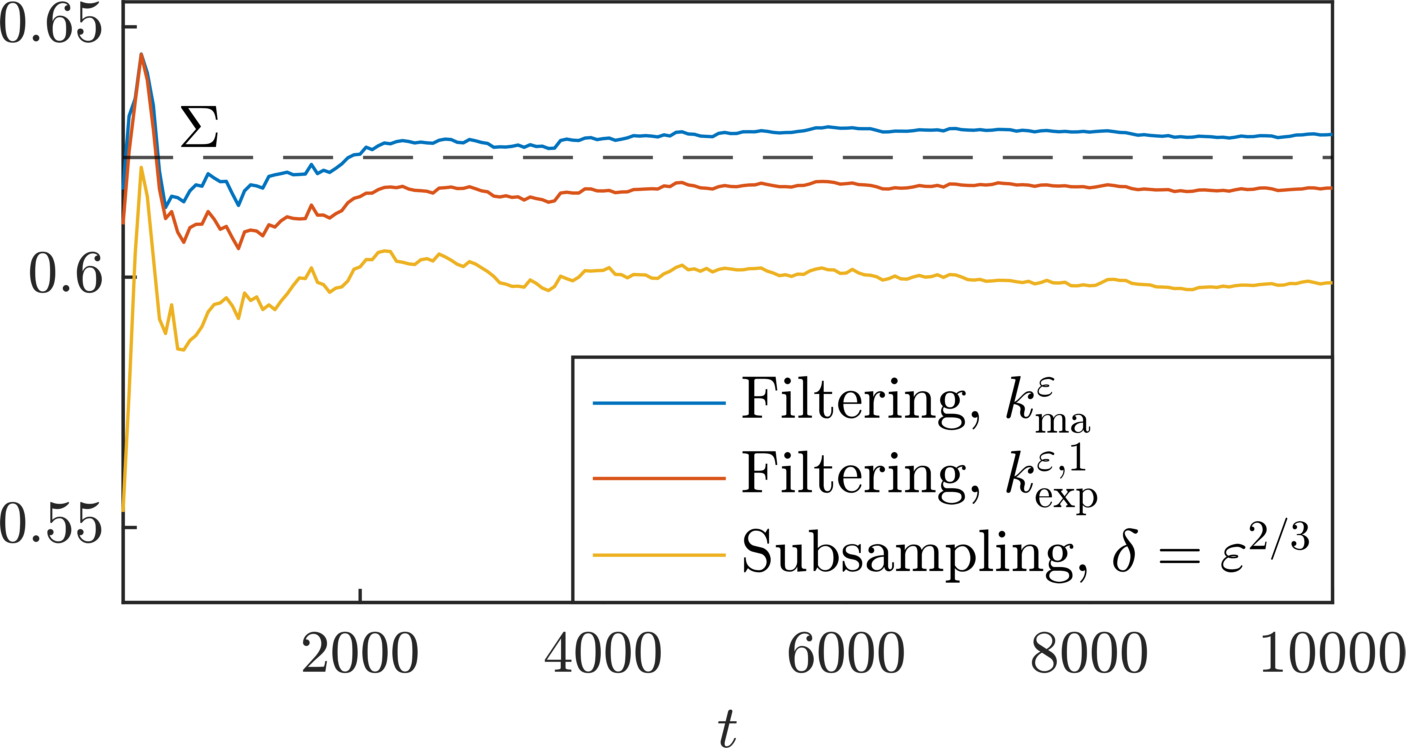}
	\caption{Estimation of the diffusion coefficient of an effective Ornstein--Uhlenbeck equation. Top: Numerical results at final time $T = 10^4$ for variable $\sigma = 1, 0.75, 0.5$ (per row), for $\epl = 0.2, 0.1, 0.05$ (blue, red, and yellow lines respectively), and for variable filtering/subsampling width $\delta$ (horizontal axis in all figures). Results for the estimators $\widehat \Sigma$: Comparison between filtering with moving average and exponential filters (first two columns) and subsampling (last column). Bottom: Convergence with respect to $t \in [0, 10^4]$ of the two estimators based on filtered data with $\delta = \epl$, and of the subsampling estimator with $\delta = \epl^{2/3}$, for fixed $\sigma = 1$ and $\epl = 0.05$. Remark: The legend on top is valid for all plots, except the last row.}
	\label{fig:Diffusion}
\end{figure}
\begin{figure}[t!]
	\centering
	\hspace{2cm}\includegraphics[]{Figures/legend_epsilon}\vspace{0.25cm}
	\begin{tabular}{m{1.3cm}m{4cm}m{4cm}m{4cm}}		
		\vspace{-1.2cm}$\sigma=1$ & \includegraphics[]{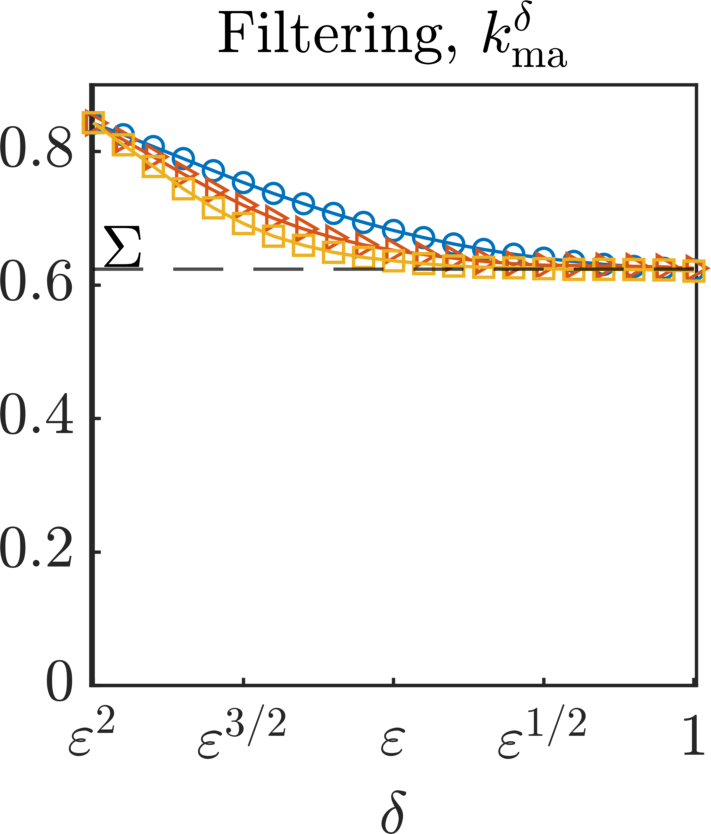} & \includegraphics[]{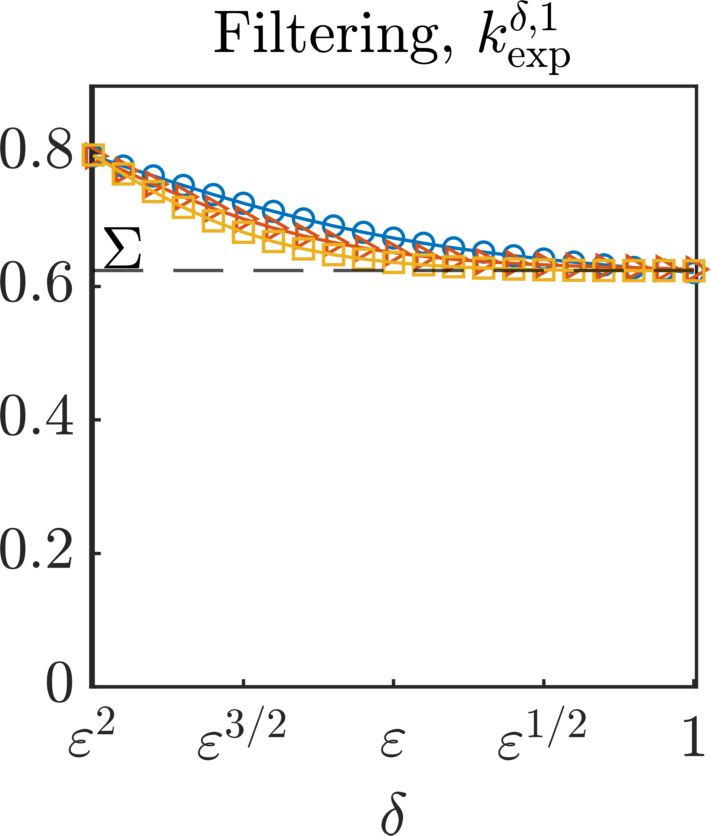}  & \includegraphics[]{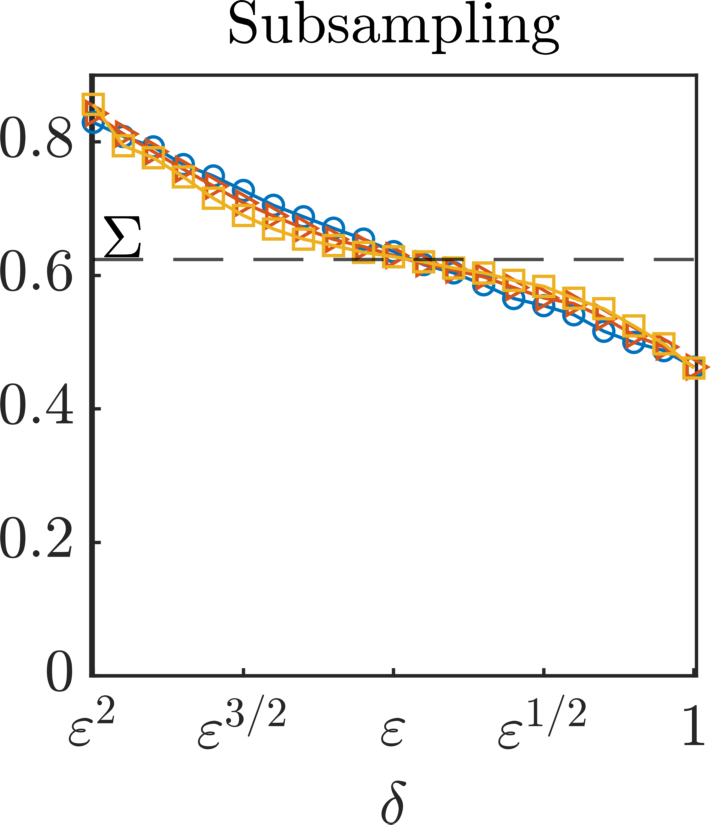} \\ 
		\vspace{-1.2cm}$\sigma=0.75$ & \includegraphics[]{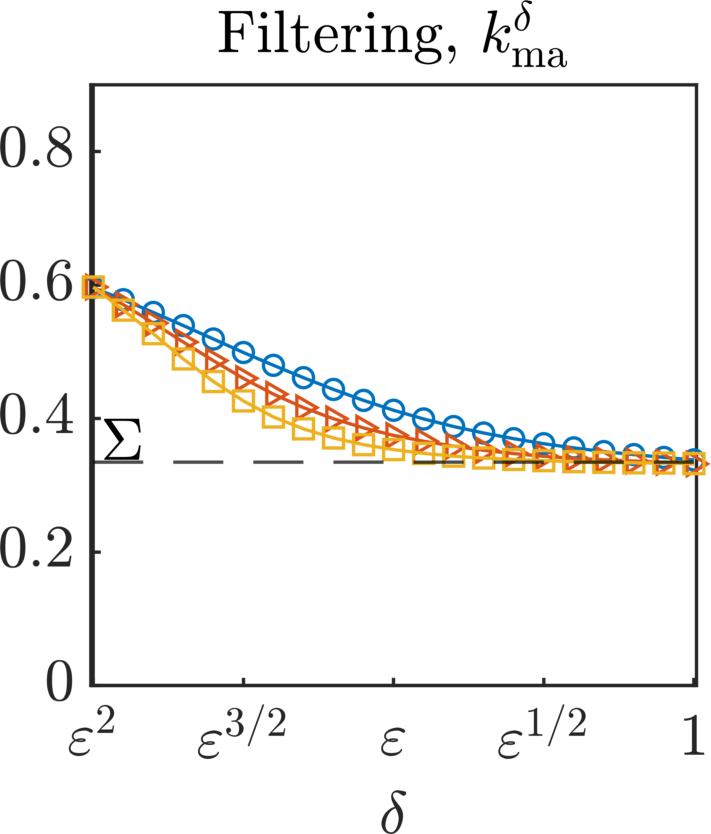} & \includegraphics[]{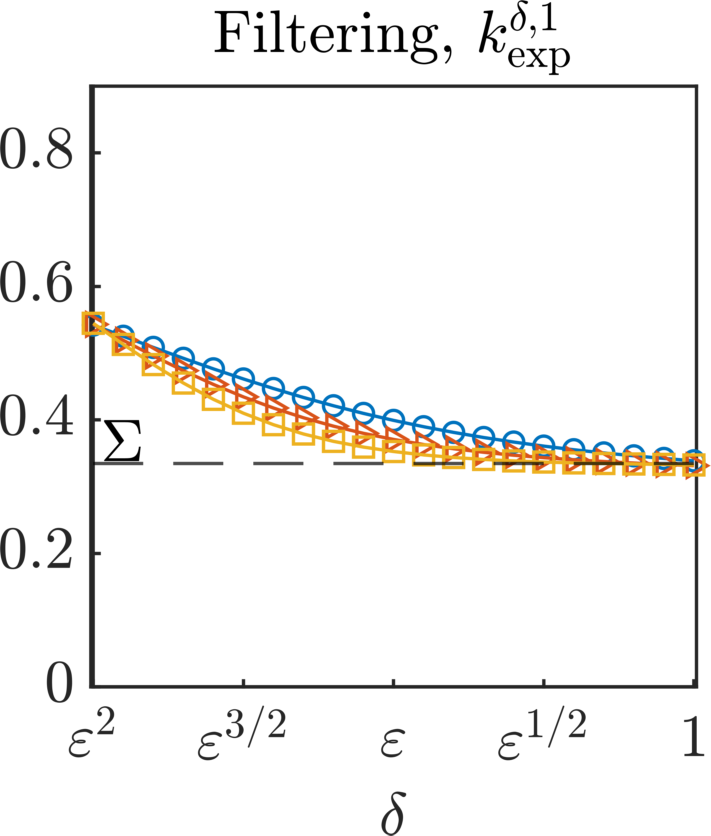}  & \includegraphics[]{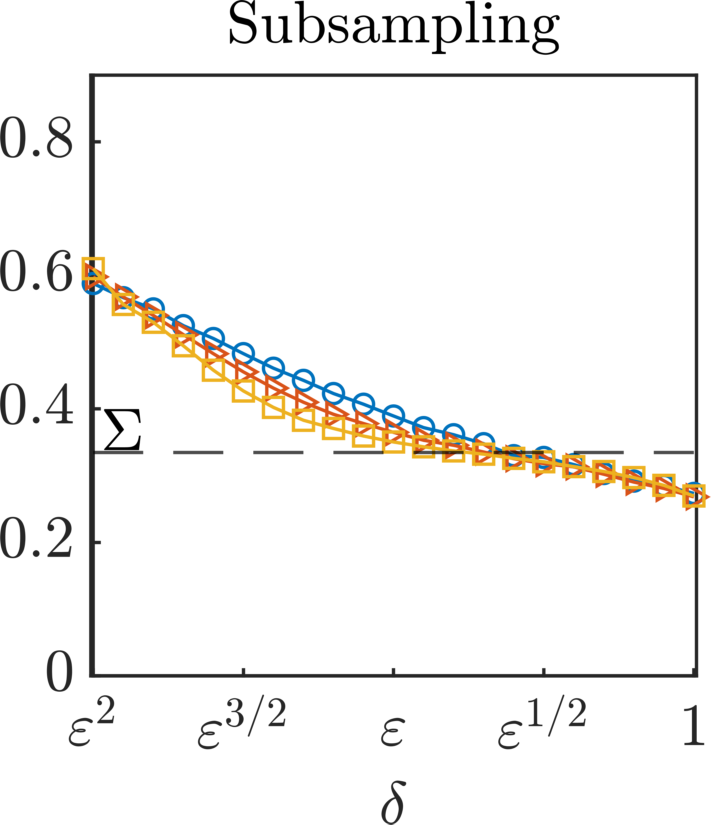} \\ 
		\vspace{-1.2cm}$\sigma=0.5$ & \includegraphics[]{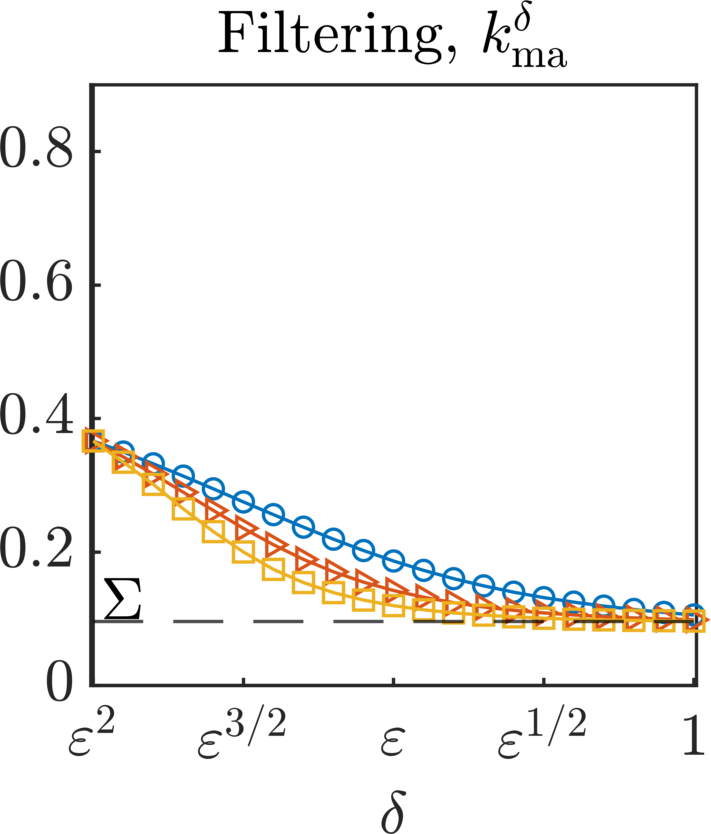} & \includegraphics[]{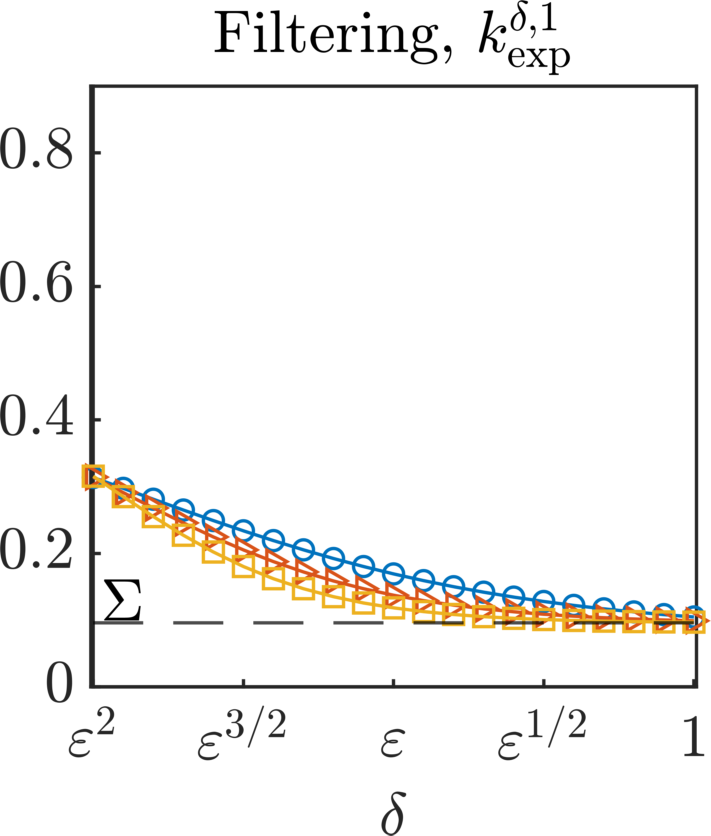}  & \includegraphics[]{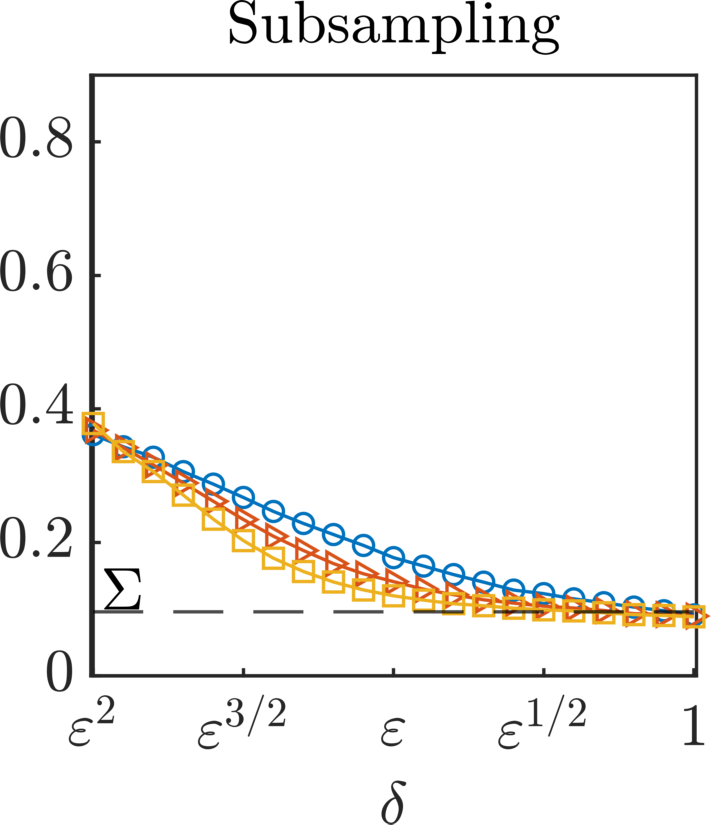} \\ 
	\end{tabular}
	
	\vspace{0.25cm}
	\includegraphics[]{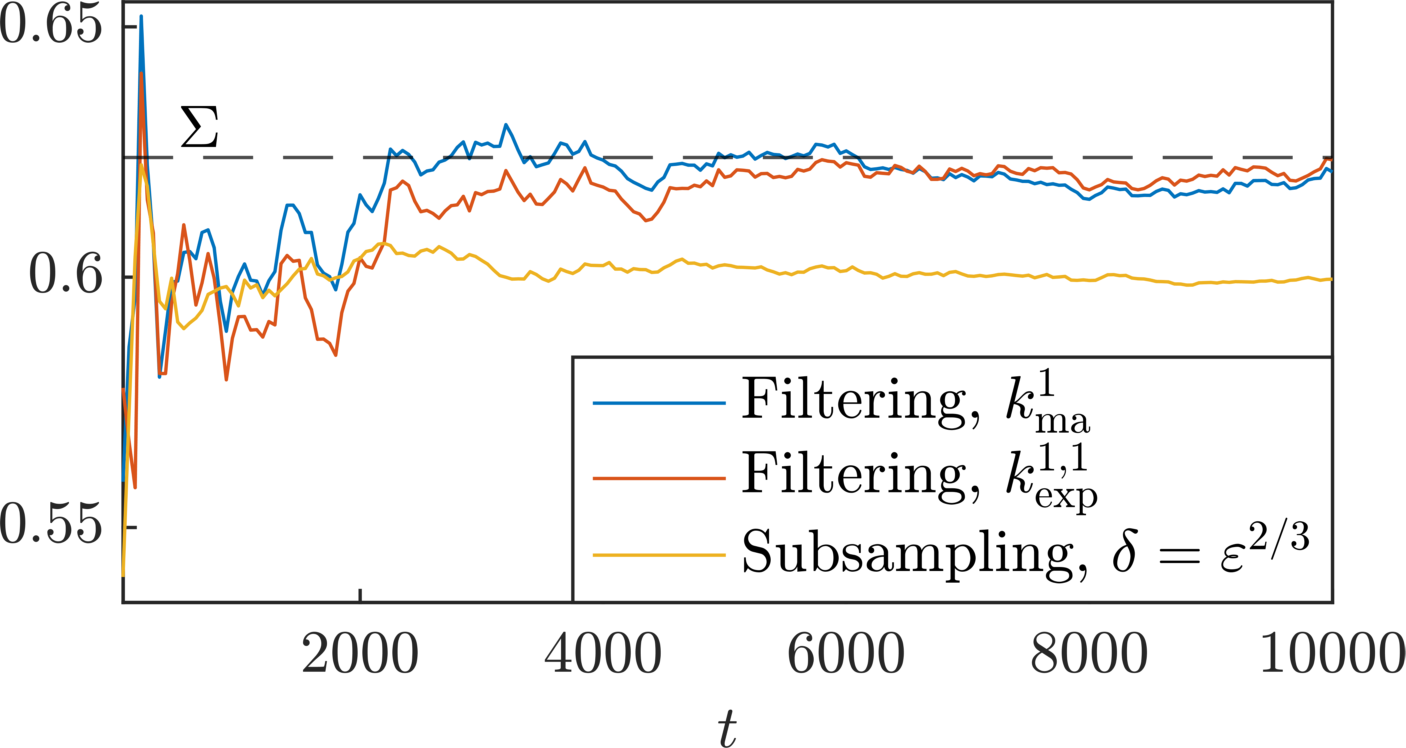}
	\caption{Estimation of the diffusion coefficient of an effective Ornstein--Uhlenbeck equation. Top: Numerical results at final time $T = 10^4$ for variable $\sigma = 1, 0.75, 0.5$ (per row), for $\epl = 0.2, 0.1, 0.05$ (blue, red, and yellow lines respectively), and for variable filtering/subsampling width $\delta$ (horizontal axis in all figures). Results for the estimators $\widetilde \Sigma$: Comparison between filtering with moving average and exponential filters (first two columns) and subsampling (last column). Bottom: Convergence with respect to $t \in [0, 10^4]$ of the two estimators based on filtered data with $\delta = 1$, and of the subsampling estimator with $\delta = \epl^{2/3}$, for fixed $\sigma = 1$ and $\epl = 0.05$. Remark: The legend on top is valid for all plots, except the last row.}	\label{fig:DiffusionNew}
\end{figure}

We first consider the one-dimensional equation \eqref{eq:SDE_MS_Lang_d1} with $N = 1$, with the slow scale potential $V(x) = x^2/2$, and with the fluctuating potential $p(y) = \sin(y)$. In this case, the effective model is an Ornstein--Uhlenbeck equation, and it is simple to verify that the homogenization coefficient $\mathcal K > 0$ is given by
\begin{equation}\label{eq:K1d}
\mathcal K = \frac{L^2}{C_\eta^+ C_\eta^-}, \qquad C_\eta^{\pm} = \int_0^L \exp\left(\pm \frac{p(y)}{\sigma}\right) \dd y,
\end{equation}
where $L = 2\pi$ is the period of $p$. In this scenario, we can compute exact values for the effective drift and diffusion coefficients to assess the accuracy of our inference method. We then compare numerically the accuracy of the estimators 
\begin{enumerate}
	\item\label{it:DriftOU} $\widehat A_{\mathrm{ma}}^\delta$, $\widehat A_{\mathrm{exp}}^{\delta,1}$, and $\widehat A_{\mathrm{sub}}^{\delta}$ of the effective drift coefficient obtained by employing data preprocessed with the moving average filter of width $\delta$, the exponential filter with $\delta$ and shape parameter $\beta = 1$, and subsampling with period $\delta$,
	\item $\widehat \Sigma_{\mathrm{ma}}^\delta$, $\widehat \Sigma_{\mathrm{exp}}^{\delta,1}$, and $\widehat \Sigma_{\mathrm{sub}}^{\delta}$ of the effective diffusion coefficient obtained with the same methods, respectively,
	\item $\widetilde \Sigma_{\mathrm{ma}}^\delta$, $\widetilde \Sigma_{\mathrm{exp}}^{\delta,1}$, and $\widetilde \Sigma_{\mathrm{sub}}^{\delta}$ of the effective diffusion coefficient based on the corresponding drift estimators given in \ref{it:DriftOU}.
\end{enumerate}
We consider $T = 10^4$ and generate data $X^\epl = (X^\epl(t), 0 \leq t \leq T)$ from the  multiscale model with drift coefficient $\alpha = 1$ and for a variable diffusion coefficient $\sigma = 0.5, 0.75, 1$, so that the homogenization coefficient $\mathcal K \approx 0.19, 0.45, 0.62$ respectively. Moreover, we consider scale-separation parameters $\epl = 0.2, 0.1, 0.05$ to observe convergence with respect to the homogenization limit. Data is generated with the Euler--Maruyama method with fixed time step $\Delta_t = \epl_{\min}^3$, where $\epl_{\min} = 0.05$ is the smallest value we employ for the scale-separation parameter. With this choice, we have the twofold advantage of introducing negligible numerical errors which do not compromise the validity of our results, and of capturing well the fast-scale oscillations. The filtering/subsampling widths are set to $\delta = \epl^\zeta$ for $\zeta = i / 10$ for $i = 0,1,\ldots,20$ to observe robustness with respect to preprocessing. Let us remark that subsampling-based estimators are asymptotically unbiased only for $\zeta \in (0, 1)$, and that the theory for both filtering kernels is different in case $\zeta = 0$, i.e., when the filtering width is independent of $\epl$.

Numerical results, given in \cref{fig:Drift,fig:Diffusion,fig:DiffusionNew}, demonstrate that
\begin{enumerate}[label=(\alph*)]
	\item \cref{fig:Drift}: The two filtering techniques yield estimators of the drift coefficient of comparable accuracy across all parameters $\sigma$, $\epl$ and $\delta$, and they are both more robust than subsampling when varying the parameters $\sigma$ and $\delta$. We observe that robustness with respect to $\delta$ is particularly improved for higher values of $\sigma$. For the two filtering methodologies asymptotic unbiasedness with respect to $\epl$ seems to hold in practice. Finally, convergence with respect to $t \in [0,T]$, verified with $\sigma = 1$, $\delta = 1$ for the filtering methods and $\delta = \epl^{2/3}$ (conjectured optimal in \cite{PaS07}) for subsampling is similar for the three methods. 
	\item\label{it:diff} \cref{fig:Diffusion}: The estimators $\widehat \Sigma_{\mathrm{ma}}^\delta$, $\widehat \Sigma_{\mathrm{exp}}^{\delta,1}$, and $\widehat \Sigma_{\mathrm{sub}}^{\delta}$ of $\Sigma$ have similar accuracy across all values of $\epl$, $\delta$, and $\sigma$. Again, convergence with respect to $\epl$ seems to be in practice respected. Let us remark that for these estimators $\delta = 1$ is not a viable choice (for neither the two filtering methods, nor subsampling). Convergence in time is therefore demonstrated for $\sigma =1$, for $\delta = \epl$ for the two filtering methods, and with $\delta = \epl^{2/3}$ for subsampling. With these choices, the moving average filter introduced in this paper seems to slightly outperform the concurrent methods.
	\item \cref{fig:DiffusionNew}: The estimators $\widetilde \Sigma_{\mathrm{ma}}^\delta$ and $\widetilde \Sigma_{\mathrm{exp}}^{\delta,1}$ show enhanced accuracy with respect to the corresponding estimators of \ref{it:diff} across all values of $\epl$, $\sigma$, and $\delta$. Convergence with respect to $\epl$ seems to be respected for all methods. Convergence in time is demonstrated for $\sigma =1$, for $\delta = 1$ for the two filtering methods, and with $\delta = \epl^{2/3}$ for subsampling. With these choices, the moving average filter introduced here seems to show a faster time transient towards the effective diffusion coefficient with respect to the concurrent methods. We remark that the diffusion estimators identified by the ``hat'' seem to converge faster with respect to $t$ than the ones identified with the ``tilde''.
\end{enumerate}

\subsection{The Semi-parametric Setting}\label{sec:Poly}

\begin{figure}[t!]
	\centering
	\hspace{0.5cm}\includegraphics[]{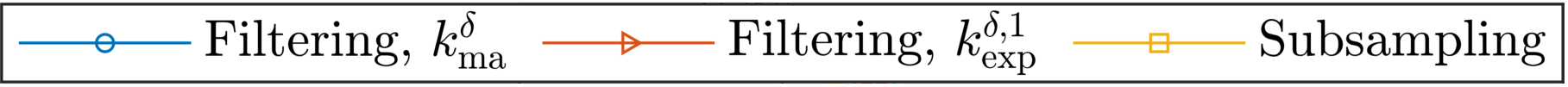}
	
	\vspace{0.25cm}
	\includegraphics[]{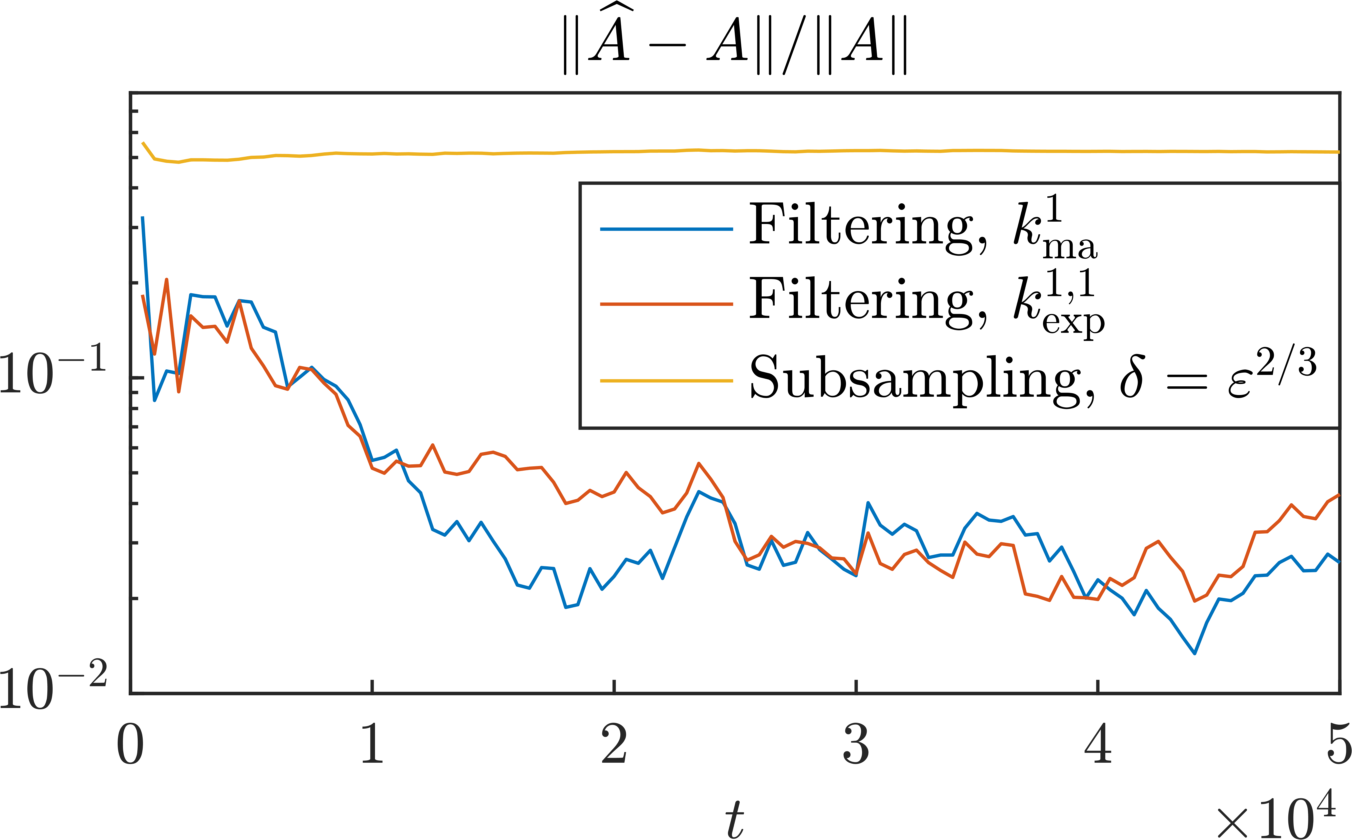}
	\includegraphics[]{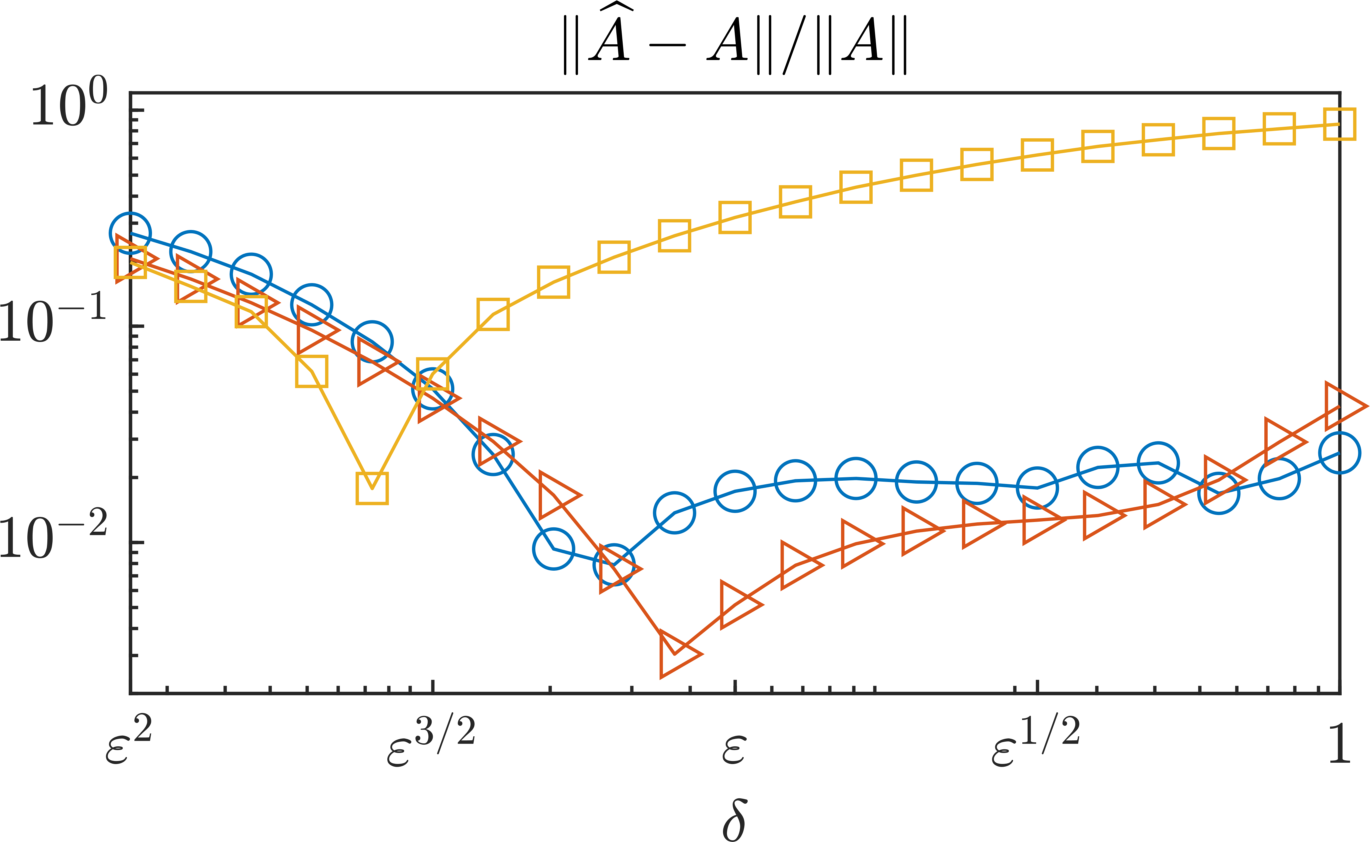}
	
	\vspace{0.25cm}
	\begin{tabular}{ccc}
		\includegraphics[]{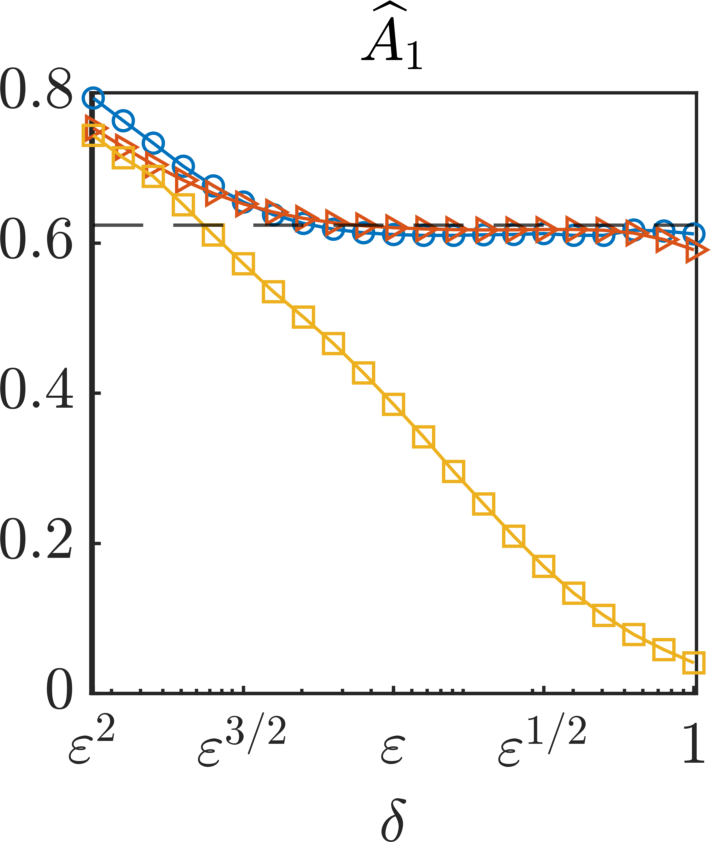} &
		\includegraphics[]{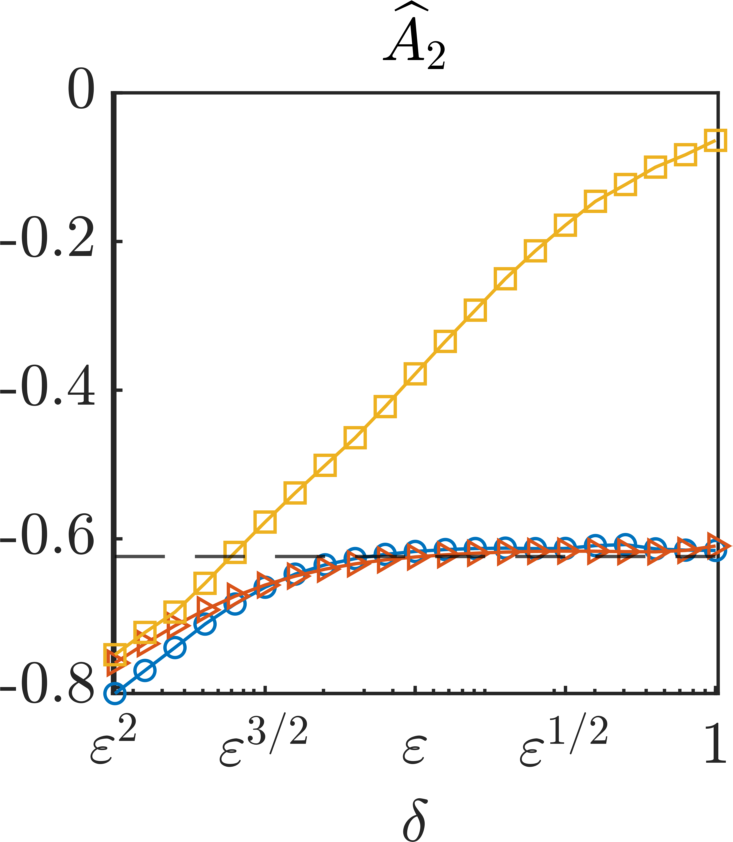} & 
		\includegraphics[]{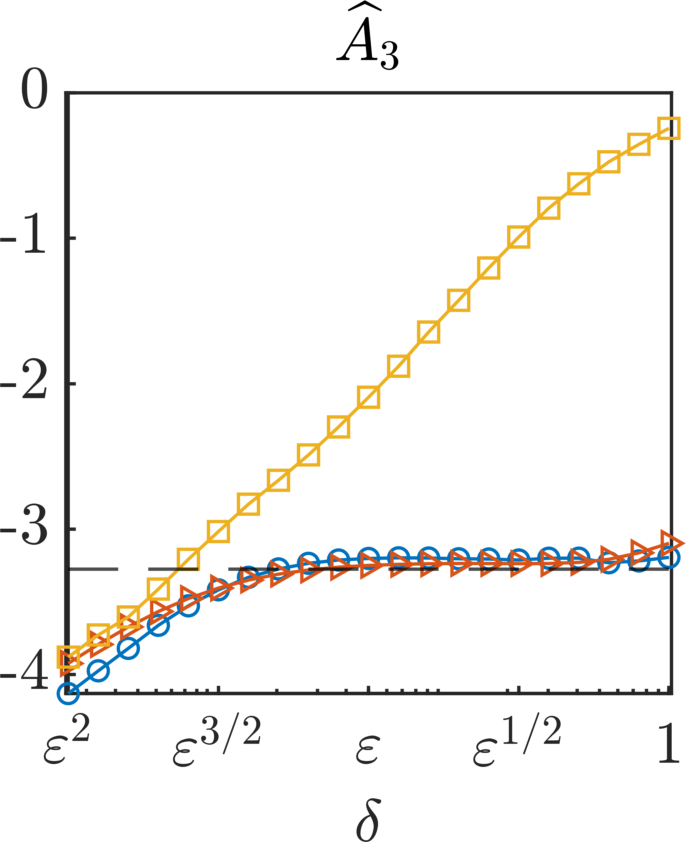} \\
		\includegraphics[]{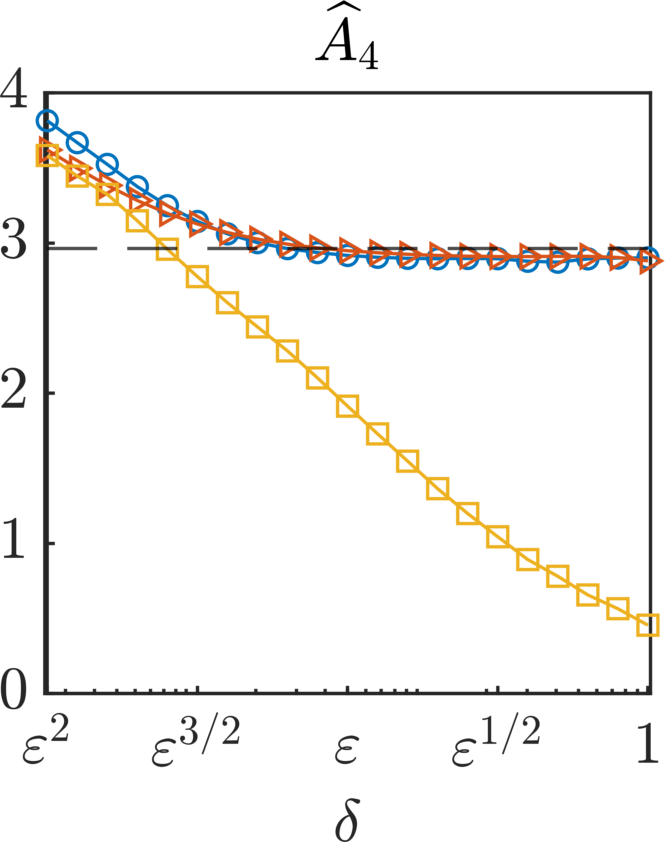} &
		\includegraphics[]{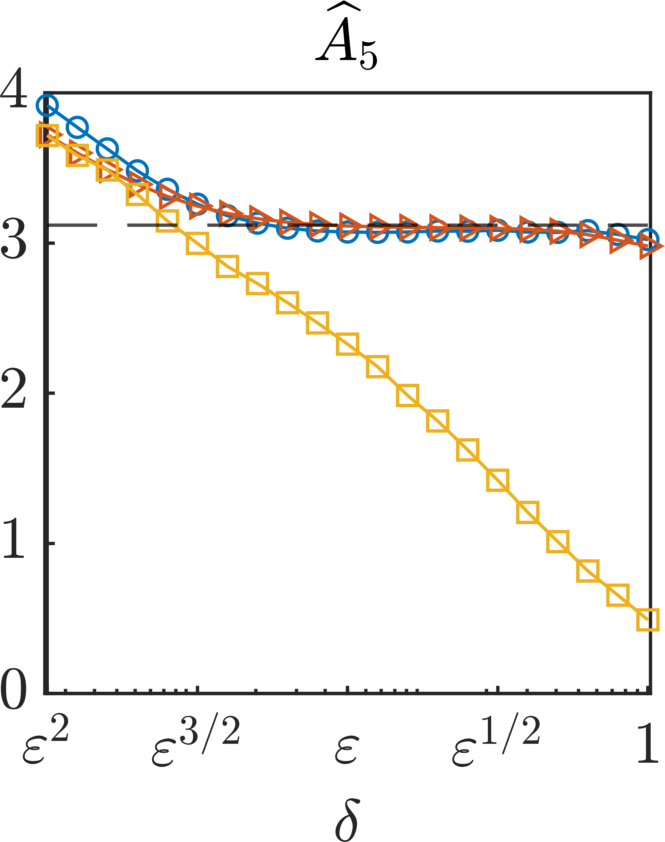} & 
		\includegraphics[]{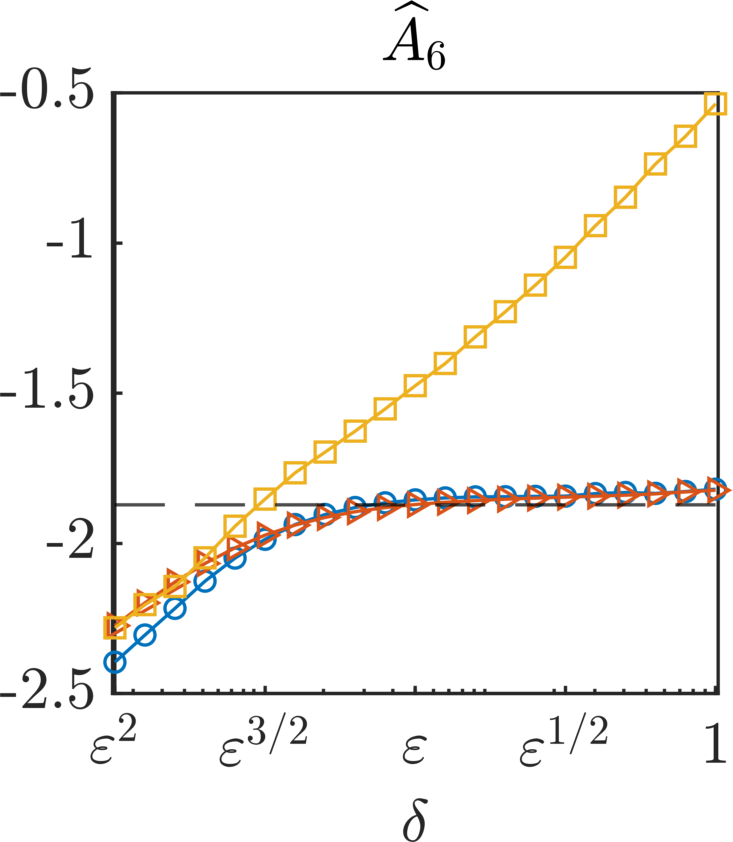}
	\end{tabular}
	
	\vspace{0.25cm}
	\begin{tabular}{cccc}
		\raisebox{1.5cm}{\includegraphics[]{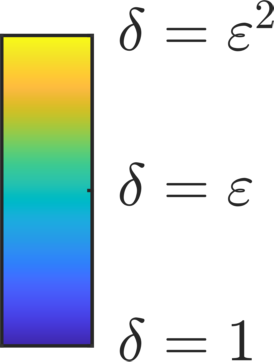}} &
		\includegraphics[]{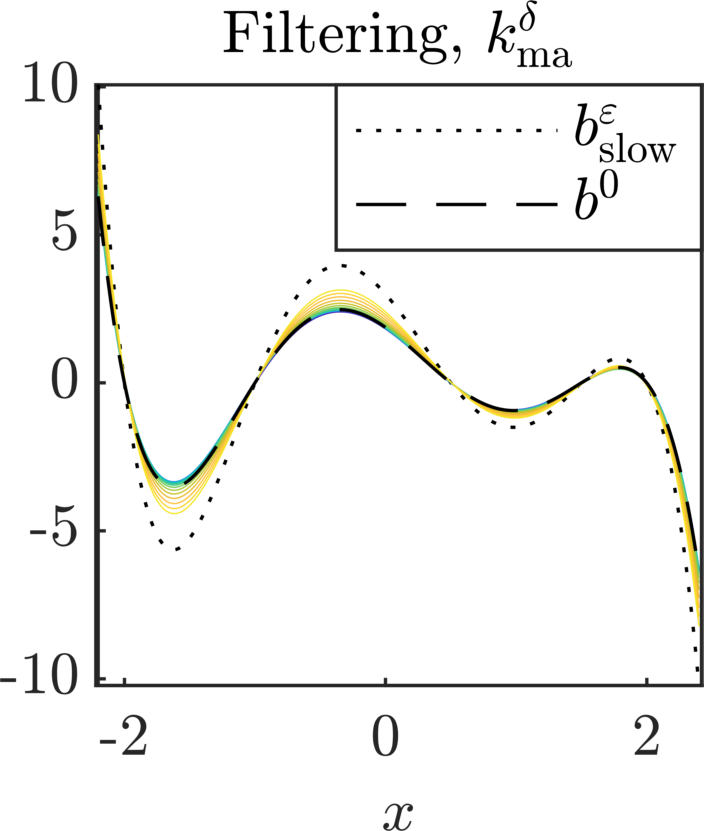} &
		\includegraphics[]{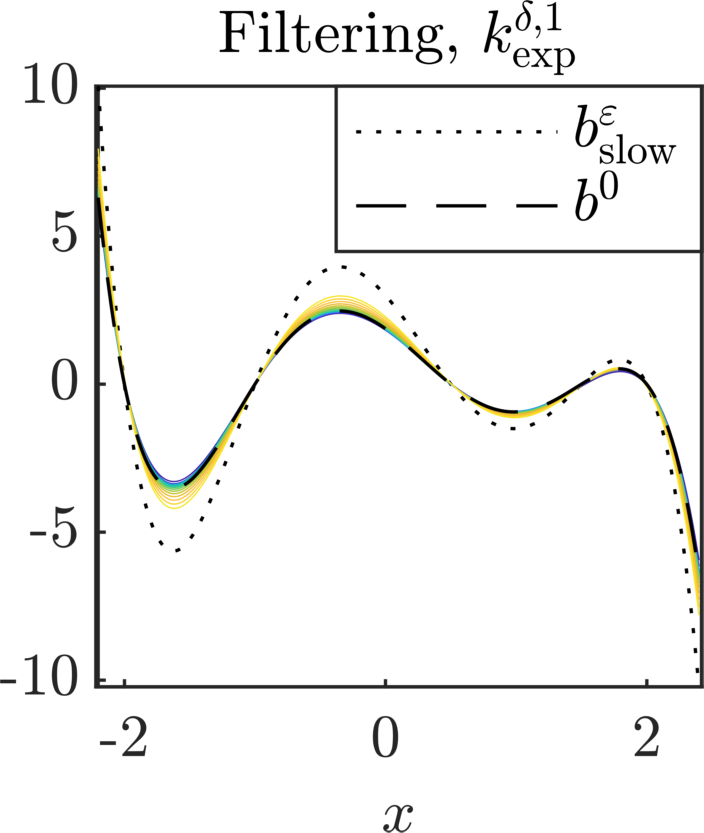} & 
		\includegraphics[]{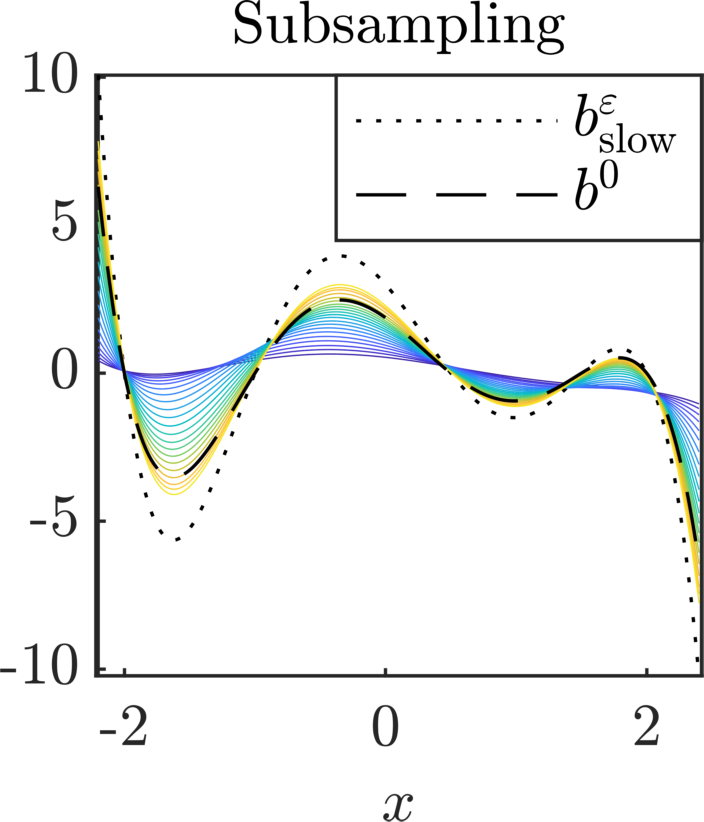} 
	\end{tabular}
	\caption{Estimation of the drift coefficient in the one-dimensional semi-parametric setting. First row: on the left, we show the evolution of the relative error with respect to $t \in [0, 5\cdot 10^4]$, and on the right the dependence of the relative error on the filtering/subsampling width $\delta \in [\epl^2, 1]$. Second and third rows: Dependence on $\delta$ of the estimators for the components $A_i$, $i=1,\ldots,6$ of the effective drift coefficient obtained with filtered data with both kernels, and with subsampling. Fourth row: Dependence on the filtering/subsampling width $\delta$ of the estimated drift function with the same three methodologies. Remark: The legend on top is valid for all plots, except the last row.}
	\label{fig:Poly_Drift}
\end{figure}
\begin{figure}
	\centering
	\hspace{0.5cm}\includegraphics[]{Figures/legend_methods}
	
	\vspace{0.25cm}
	\includegraphics[]{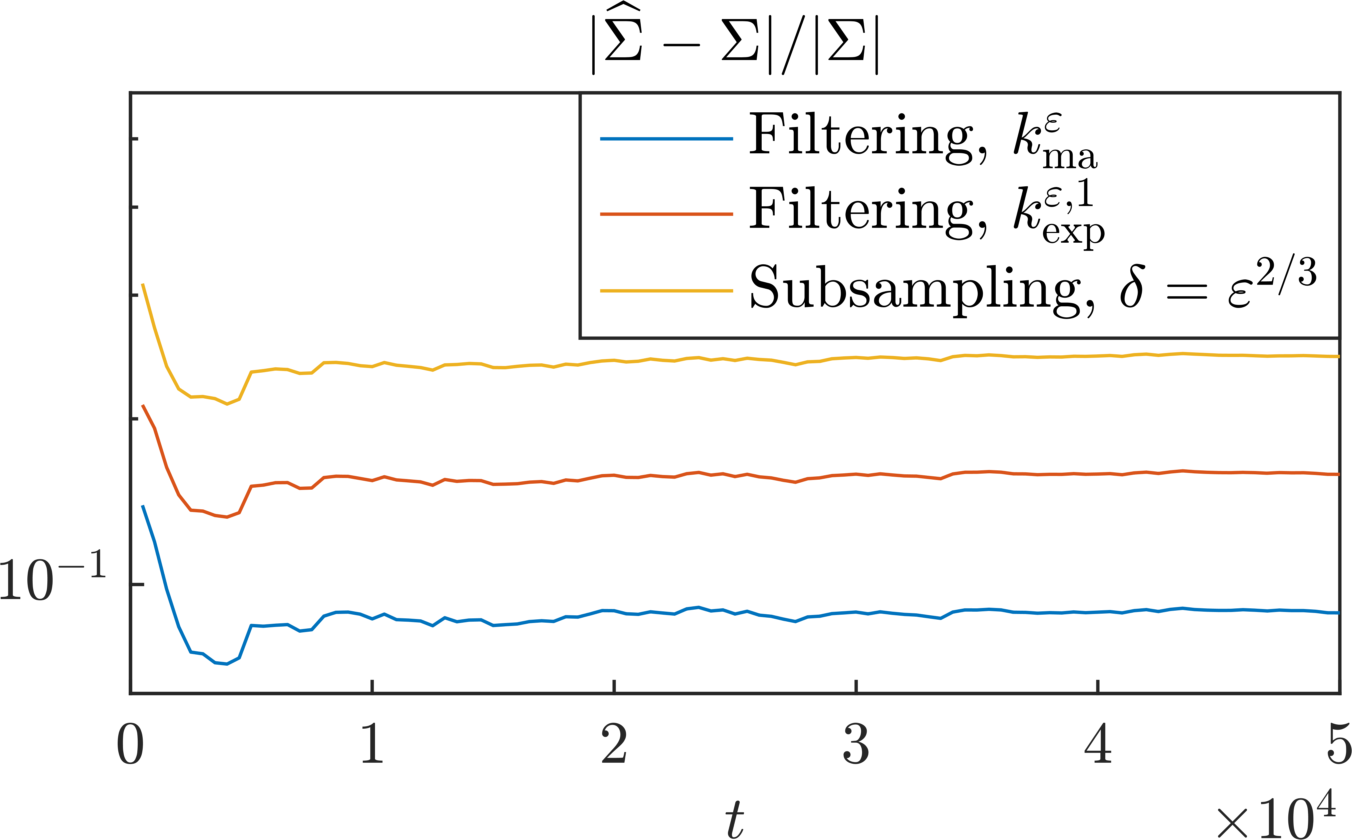}
	\includegraphics[]{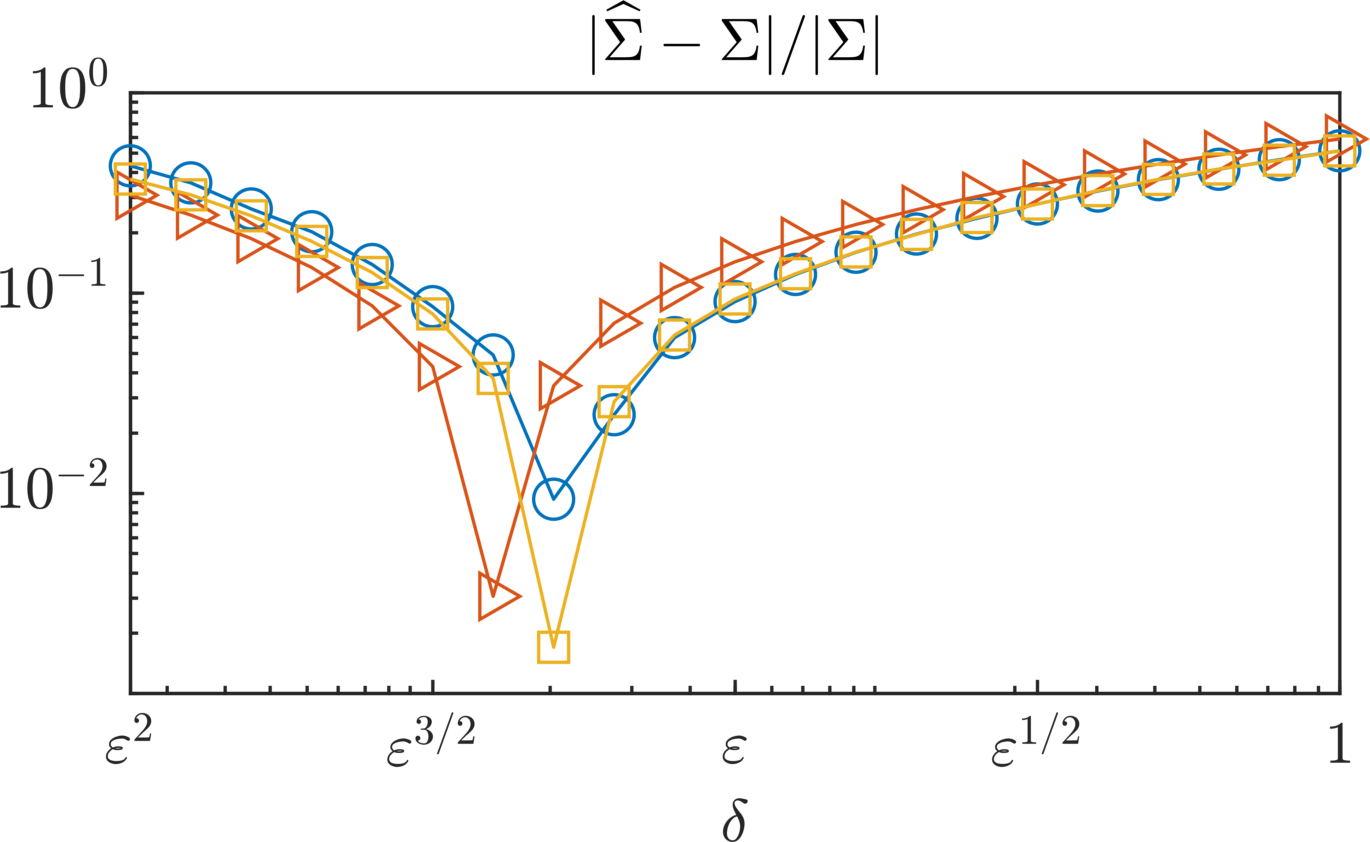}
	
	\vspace{0.25cm}
	\includegraphics[]{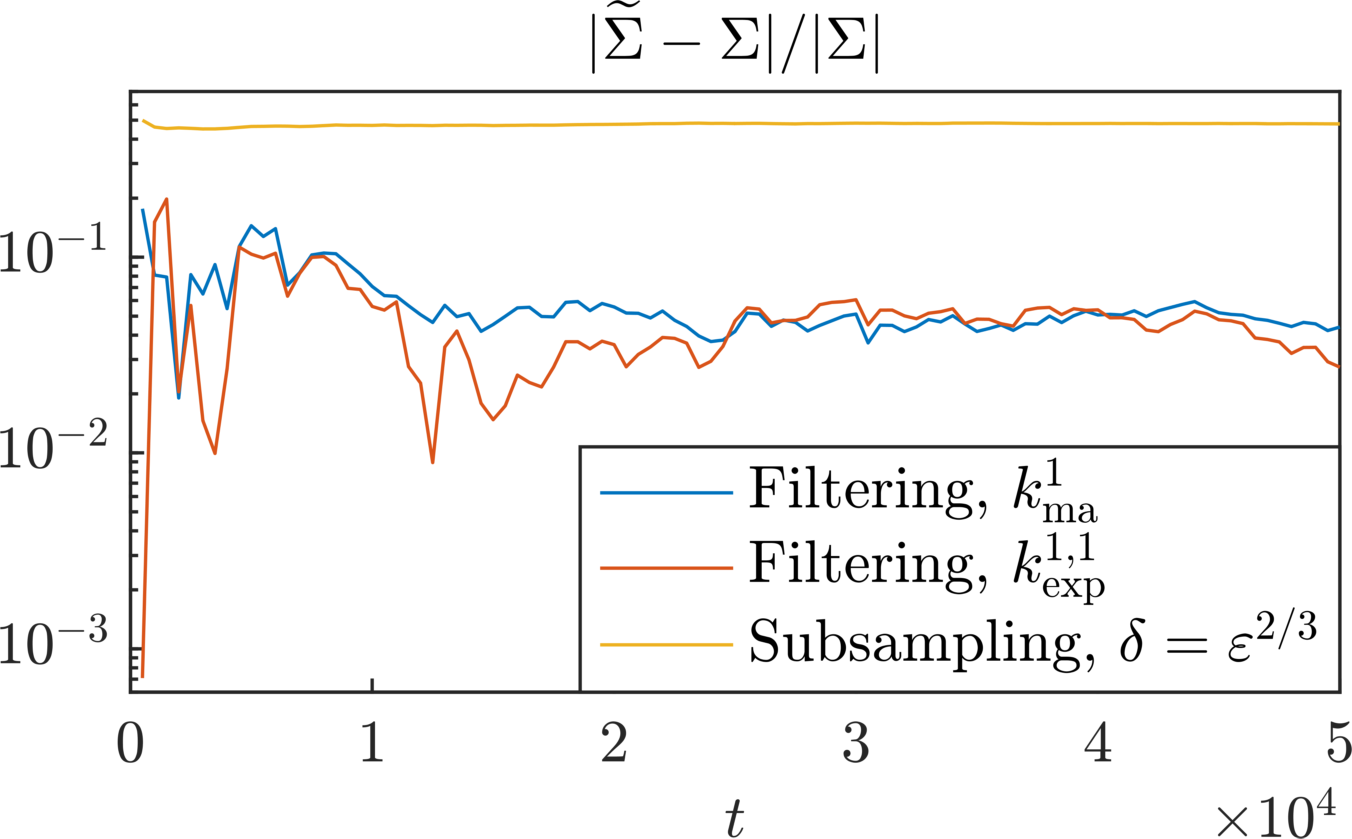}
	\includegraphics[]{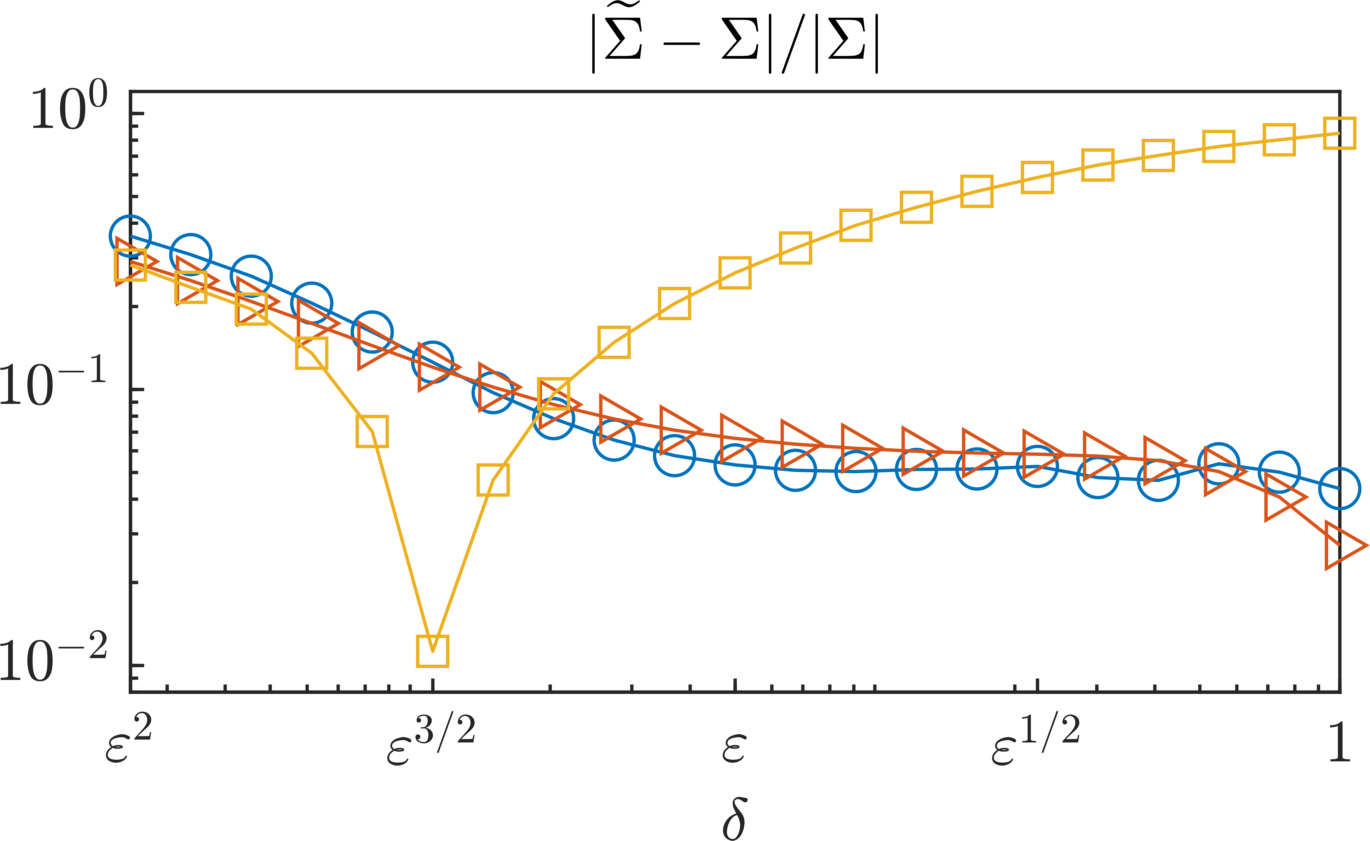}
	
	\vspace{0.25cm}
	\begin{tabular}{cc}
		\includegraphics[]{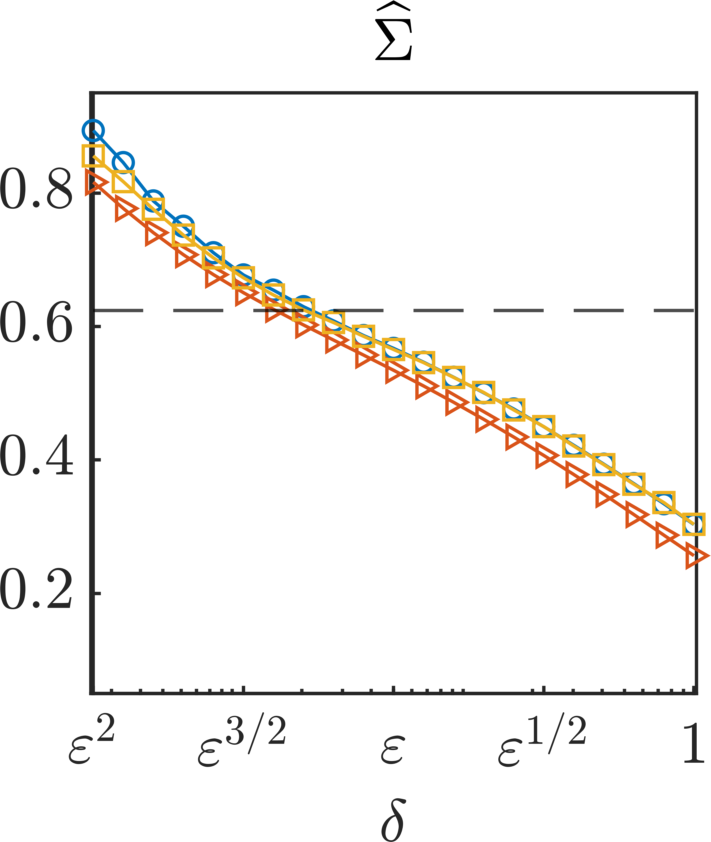} &
		\includegraphics[]{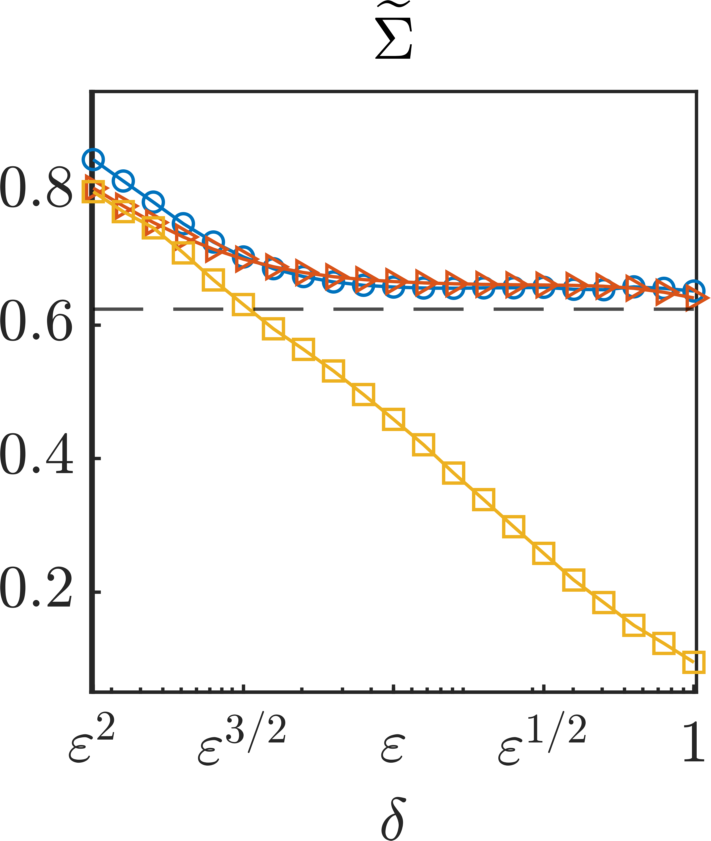}
	\end{tabular}
	\caption{Estimation of the diffusion coefficient in the one-dimensional semi-parametric setting. First ($\widehat \Sigma$) and second ($\widetilde \Sigma$) row: on the left, we show the evolution of the relative error with respect to $t \in [0, 5\cdot 10^4]$, and on the right the dependence of the relative error with respect to the filtering/subsampling width $\delta \in [\epl^2, 1]$. Third row: Dependence on $\delta$ of the estimators $\widehat \Sigma$ and $\widetilde \Sigma$ of the effective diffusion coefficient obtained with filtered data with both kernels, and with subsampling.}
	\label{fig:Poly_Diffusion}
\end{figure}

We now consider the semi-parametric setting for a one-dimensional multiscale Langevin equation of the form \eqref{eq:SDE_MS_Lang_d1}. In particular, we consider the number of parameters $N = 6$ and define $V\colon \R \to \R^N$ as
\begin{equation}
V(x) = \begin{pmatrix} \dfrac{x^6}{6} & \dfrac{x^5}{5} & \dfrac{x^4}{4} & \dfrac{x^3}{3} & \dfrac{x^2}{2} & x \end{pmatrix}^\top.
\end{equation}
The slow-scale potential is premultiplied by the six dimensional drift coefficient $\alpha \in \R^6$ 
\begin{equation}
\alpha = \begin{pmatrix} 1 & -1 & -5.25 & 4.75 & 5 & -3	\end{pmatrix}^\top.
\end{equation}
With this choice, the slow-scale potential $\alpha \cdot V$ has three stable points. Moreover, we choose the fast-scale potential as $p = \sin(y)$, the diffusion coefficient $\sigma = 1$, and the multiscale parameter $\epl = 0.05$. We then wish to infer the effective drift and diffusion coefficients $A \in \R^6$ and $\Sigma > 0$ from synthetic data $X^\epl = (X^\epl(t), 0 \leq t \leq T)$ with $T = 5 \cdot 10^4$, generated with the Euler--Maruyama method with time step $\Delta_t = \epl^3$. In this case, the homogenization coefficient $\mathcal K \approx 0.62$. We then infer the effective drift and diffusion coefficients $A \in \R^6$ and $\Sigma > 0$ which define the homogenized equation \eqref{eq:SDE_Hom_Lang_d1}. Similarly to \cref{sec:NumExp_OU}, we compare the two filtering methodologies (moving average and exponential kernels), and subsampling. Moreover, we compute for all strategies the effective drift estimator $\widehat A$, and the effective diffusion estimators $\widehat \Sigma$ and $\widetilde \Sigma$. 

Numerical results, given in \cref{fig:Poly_Drift,fig:Poly_Diffusion}, demonstrate that
\begin{enumerate}[label=(\alph*)]
	\item \cref{fig:Poly_Drift}: The six-dimensional effective drift coefficient is estimated accurately by both filtering-based methodologies, which yield comparable results both in terms of time convergence and of robustness with respect to the filtering width $\delta$. At final time, both estimators have relative errors of magnitude $10^{-2}$, and all six components of the drift coefficient are accurately retrieved. We note that for $\delta= 1$, i.e., the go-to implementation when $\epl$ is unknown, the moving average estimator appears to be slightly better than the one obtained with the exponential filter.  Subsampling, conversely, does not enable to retrieve the drift coefficient accurately and strongly depends on the subsampling width $\delta$. We remark that that the optimal value for $\delta$ appears to be $\delta \approx \epl^{3/2}$, which is surprising in view of the convergence result of \cite{PaS07}. Finally, we notice that for all values of $\delta$ the estimated drift function is visually almost identical to the effective drift, and is clearly differentiated from the slow component of the multiscale drift.
	\item \cref{fig:Poly_Diffusion}: The diffusion coefficient is estimated more accurately by the estimator $\widetilde \Sigma$ than $\widehat \Sigma$ when employing filtered data. Indeed, for both filtering kernels the estimator $\widetilde \Sigma$ is very robust with respect to the filtering width $\delta$, and results are very accurate in case $\delta = 1$, the go-to implementation when the scale-separation parameter $\epl$ is unknown. Conversely, the estimator $\widehat \Sigma$ strongly depends on the filtering/subsampling width. We note that choosing $\delta = \epl$ the moving average estimator $\widehat \Sigma_{\mathrm{ma}}^\epl$ outperforms the corresponding estimator $\widehat \Sigma_{\mathrm{exp}}^{\epl,1}$ based on exponential filtering. For subsampling, the two estimators are equivalent in terms of accuracy, and are extremely dependent of the subsampling width $\delta$. Equivalently to the drift estimator, the best inference results seem to be given by $\delta \approx \epl^{3/2}$, and not at the conjectured optimal value $\epl^{2/3}$.
\end{enumerate}

\subsection{A Two-dimensional Example}\label{sec:NumExp_2d}

\begin{figure}[t!]
	\centering
	\begin{tabular}{ccc}
		\includegraphics[]{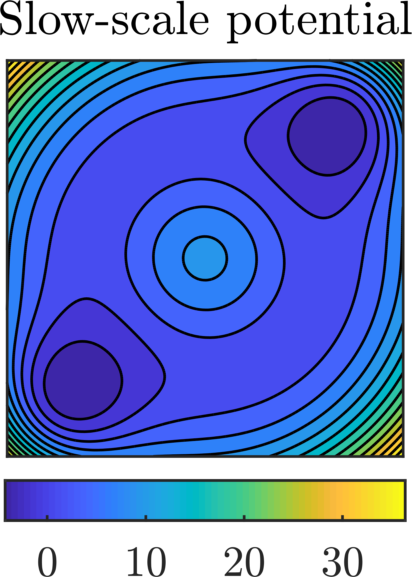} & \includegraphics[]{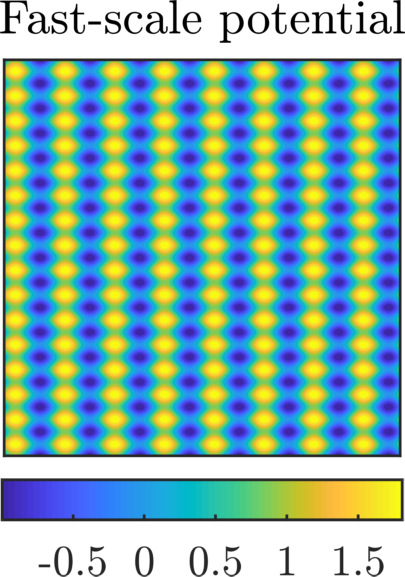} & \includegraphics[]{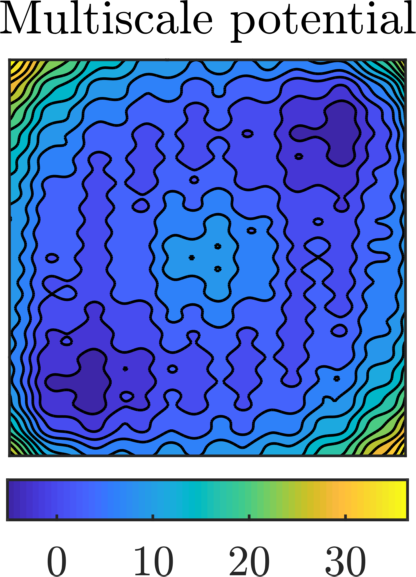}
	\end{tabular}
	\caption{Slow-, fast-, and multiscale potentials for the two-dimensional example of \cref{sec:NumExp_2d}, depicted here in the square $(-2.5, 2.5)^2$.}
	\label{fig:TwoD_Potential}
\end{figure}
\begin{figure}[t!]
	\centering
	\hspace{0.25cm}\includegraphics[]{Figures/legend_methods}
	
	\vspace{0.25cm}
	\includegraphics[]{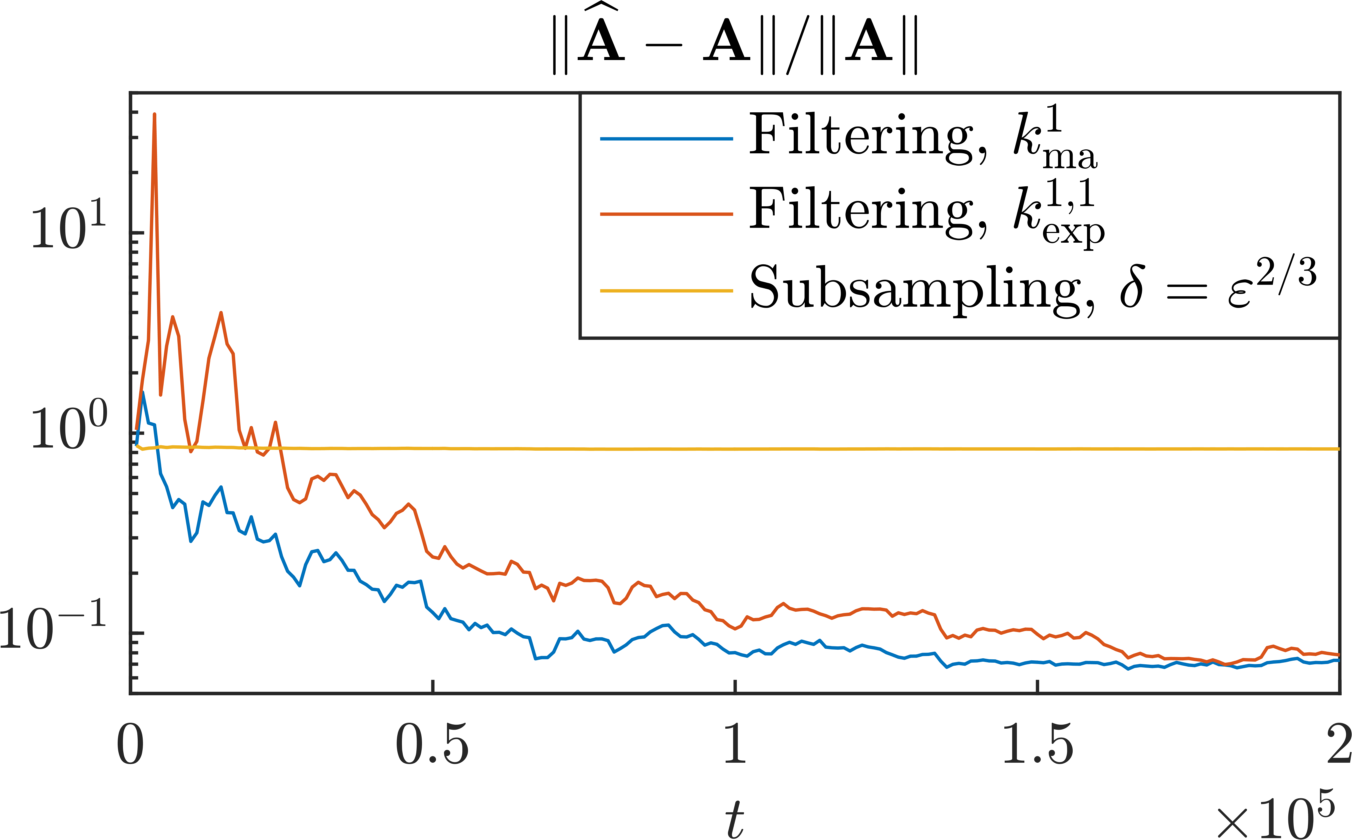}
	\includegraphics[]{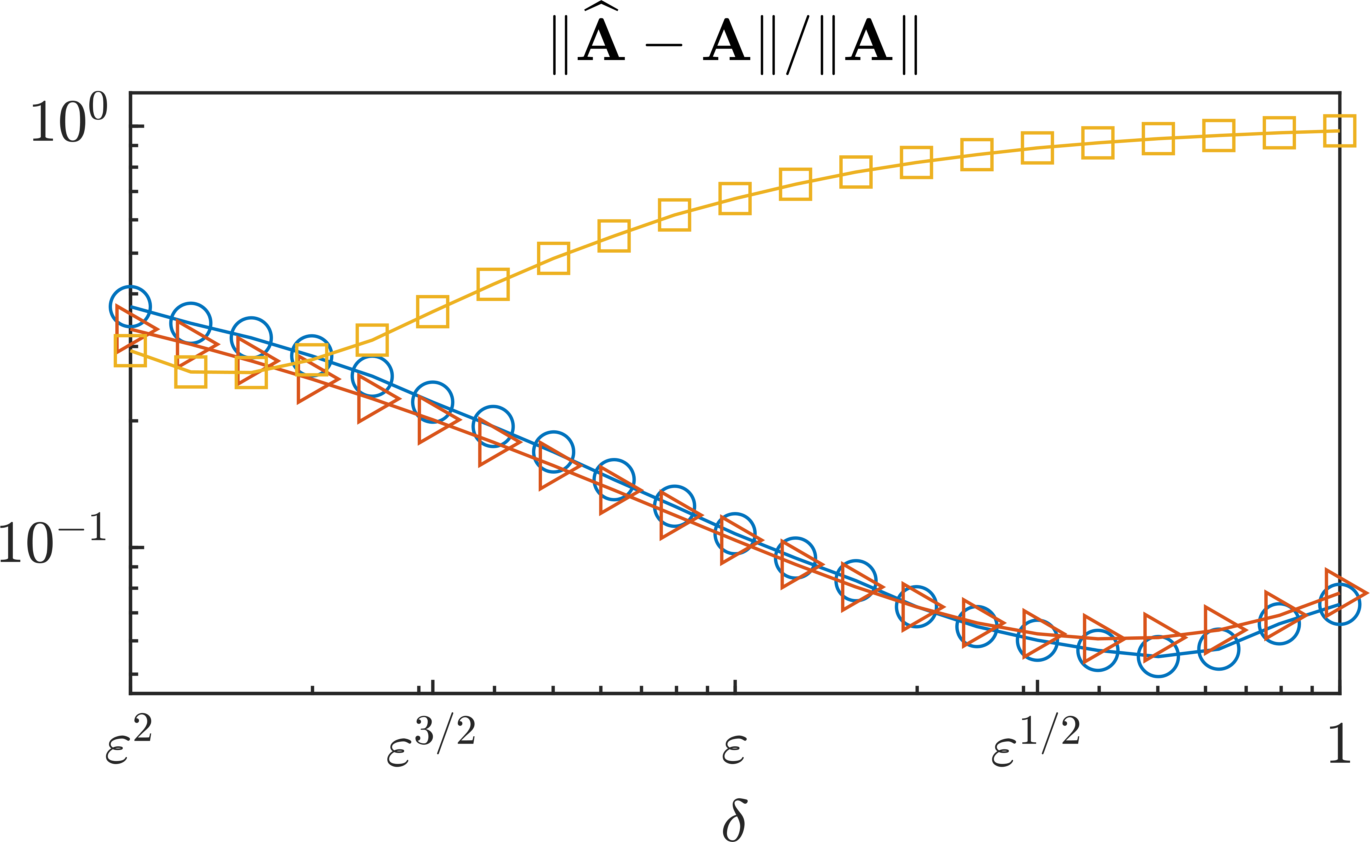}
	
	\vspace{0.2cm}
	\begin{tabular}{cccc}
		\includegraphics[]{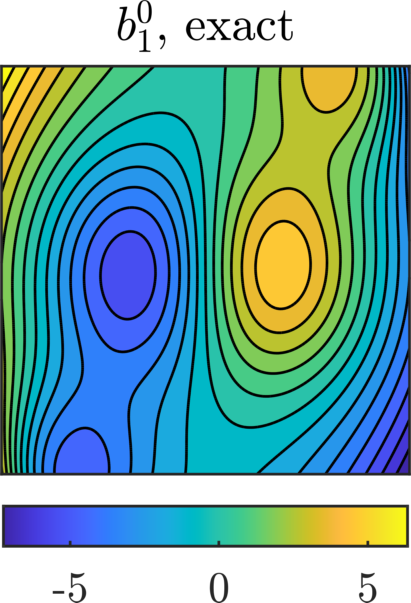} & \includegraphics[]{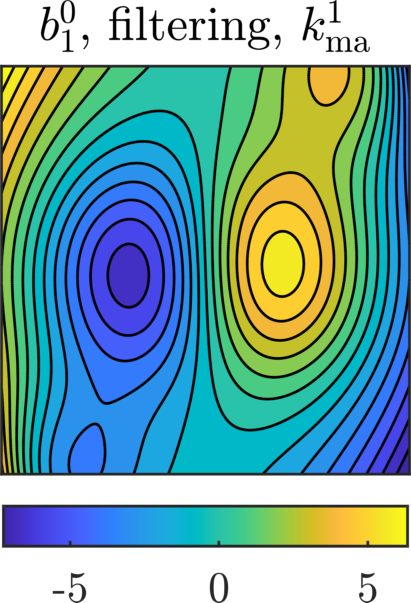} & \includegraphics[]{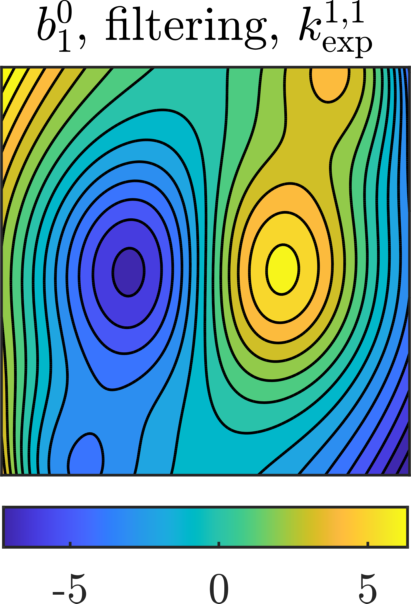} & \includegraphics[]{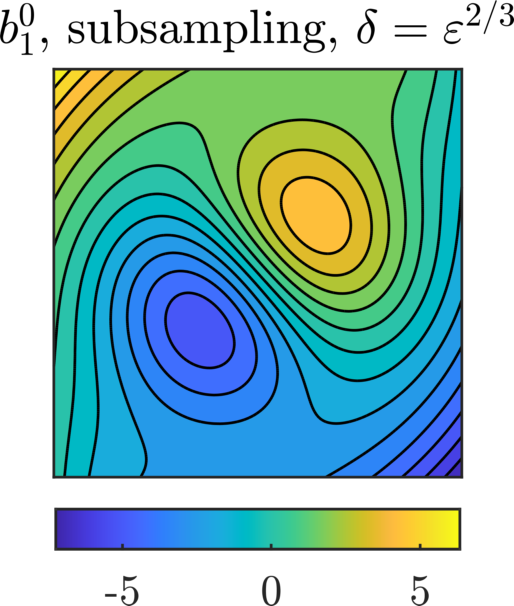} \vspace{0.3cm} \\
		\includegraphics[]{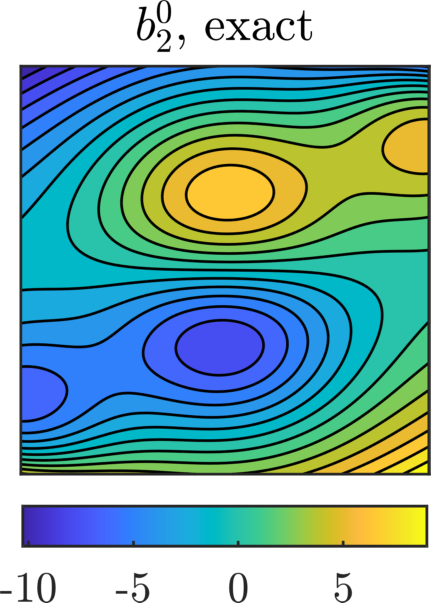} & \includegraphics[]{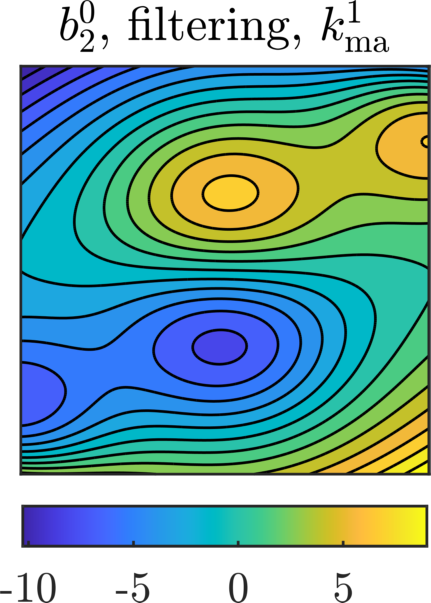} & \includegraphics[]{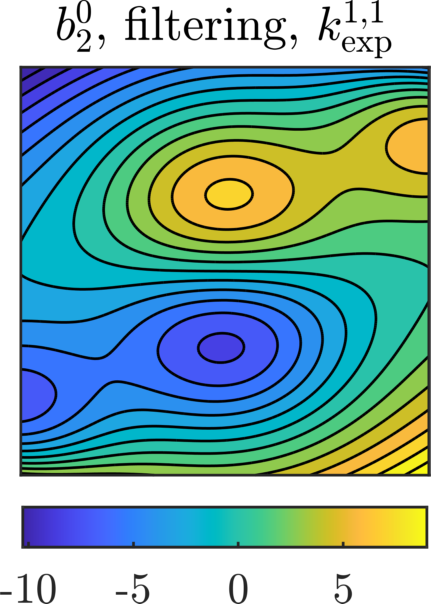} & \includegraphics[]{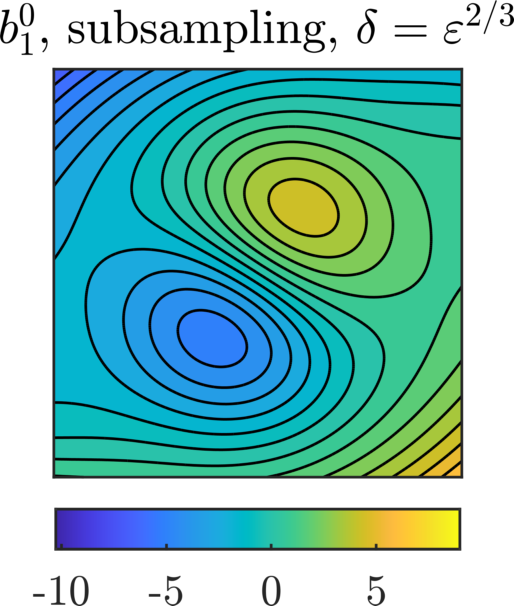}
	\end{tabular}
	\caption{Estimation of the effective drift coefficient for the two-dimensional example of \cref{sec:NumExp_2d}. First row: Dependence of the relative error with respect to $t \in [0,T]$ (left) and to the subsampling/filtering width $\delta$ (right) for both filtering methods and subsampling. Second and third row: Graphical representation in the square $(-2.5, 2.5)^2$ of the estimated effective drift function at final time for the same three methods.}
	\label{fig:TwoD_Drift}
\end{figure}
\begin{figure}
	\centering
	\hspace{0.25cm}\includegraphics[]{Figures/legend_methods}
	
	\vspace{0.25cm}
	\includegraphics[]{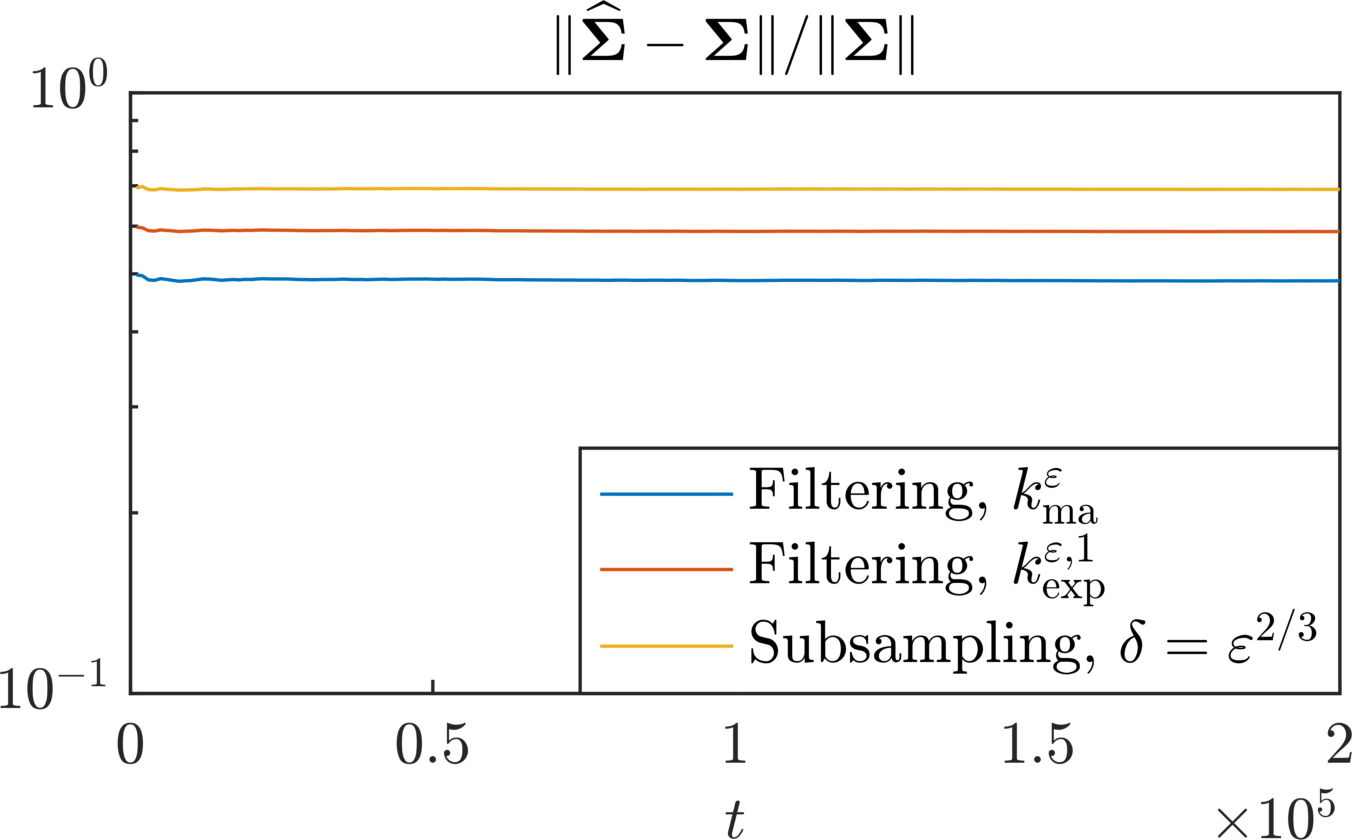}
	\includegraphics[]{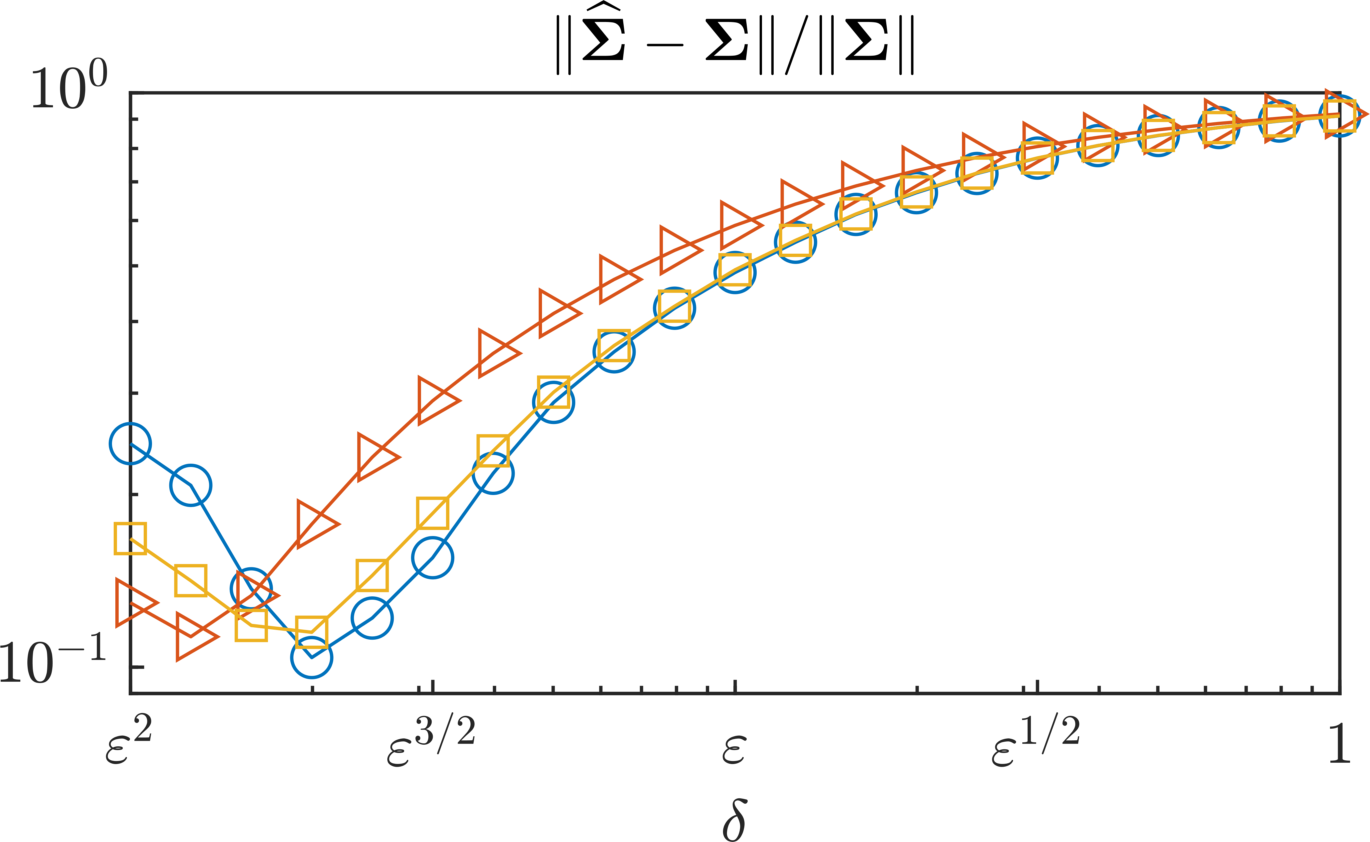}
	
	\vspace{0.25cm}
	\includegraphics[]{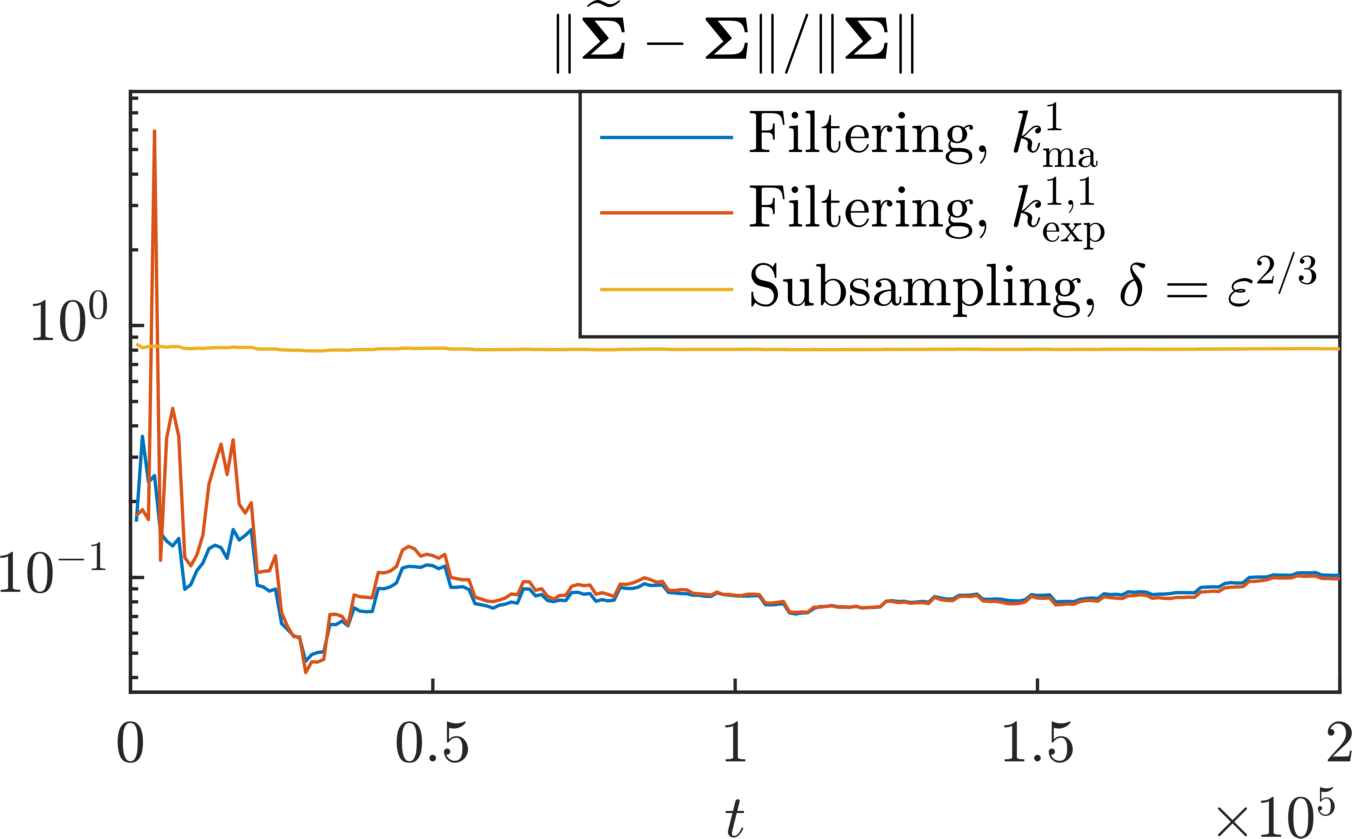}
	\includegraphics[]{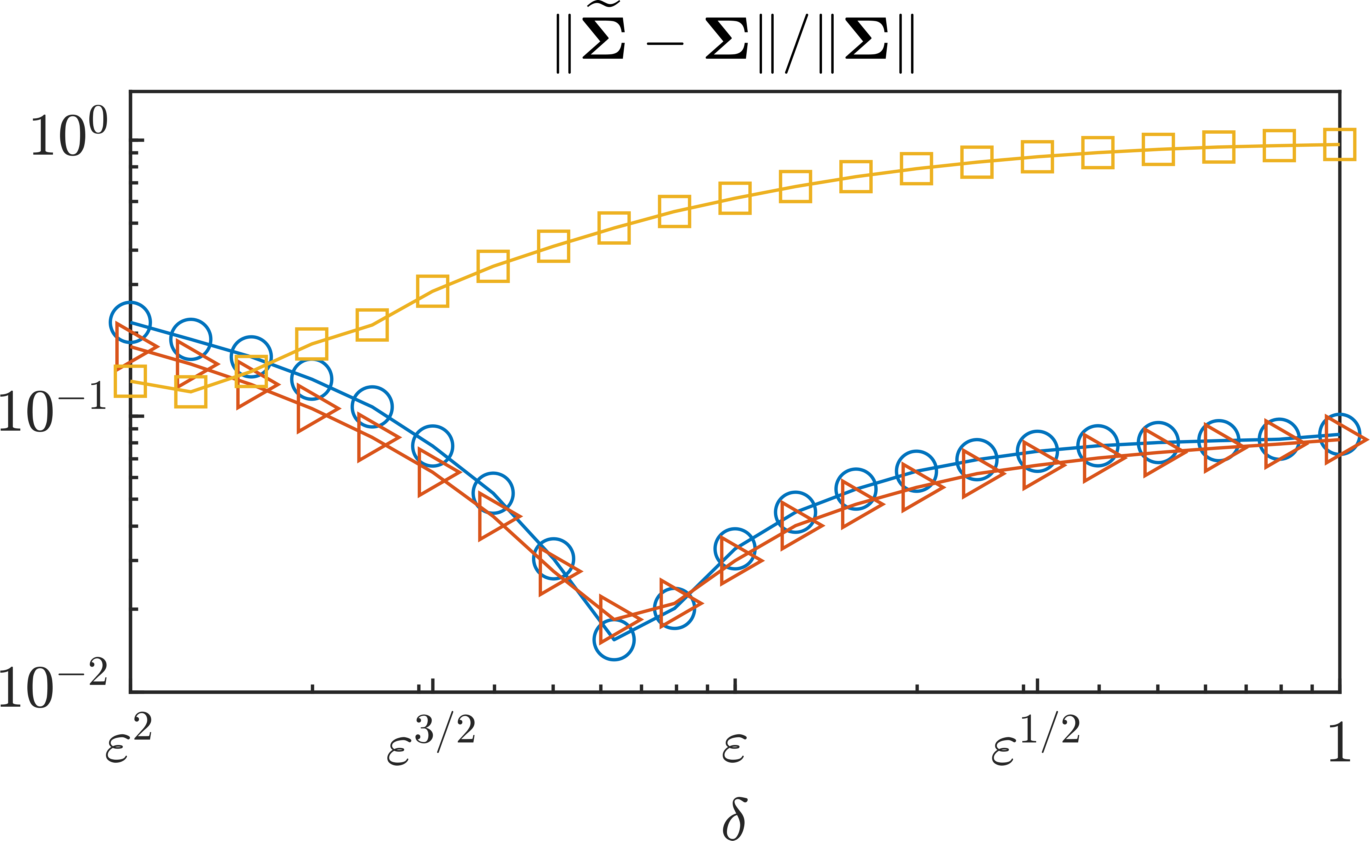}
	
	\vspace{0.25cm}
	\begin{tabular}{ccc}
		\includegraphics[]{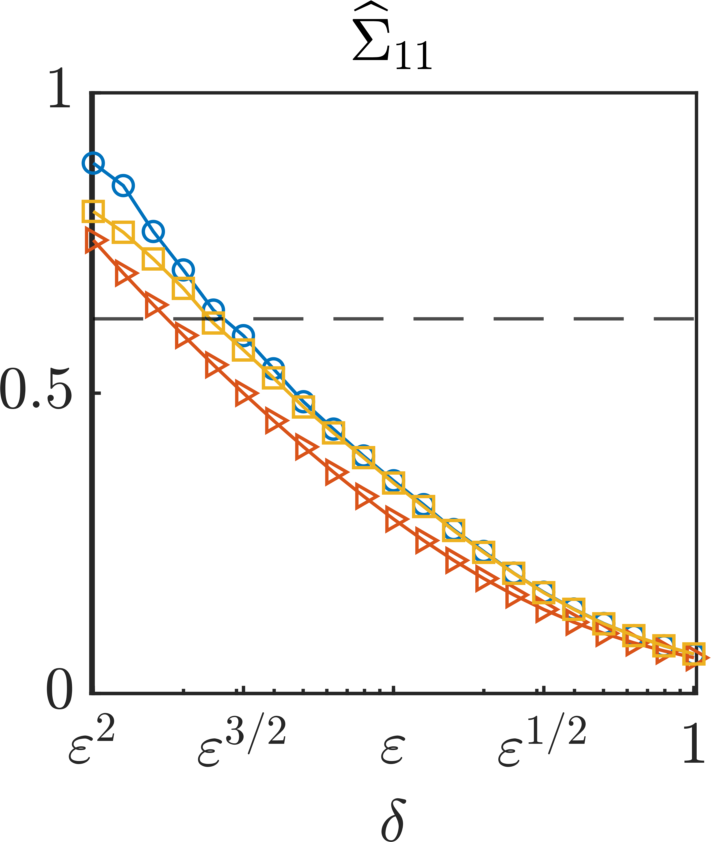} &
		\includegraphics[]{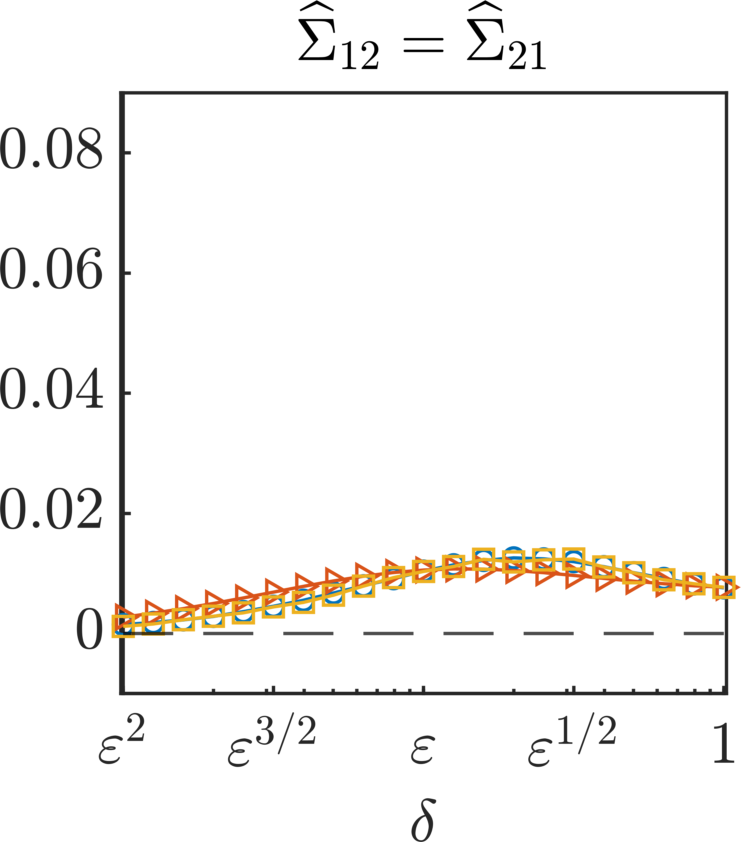} &
		\includegraphics[]{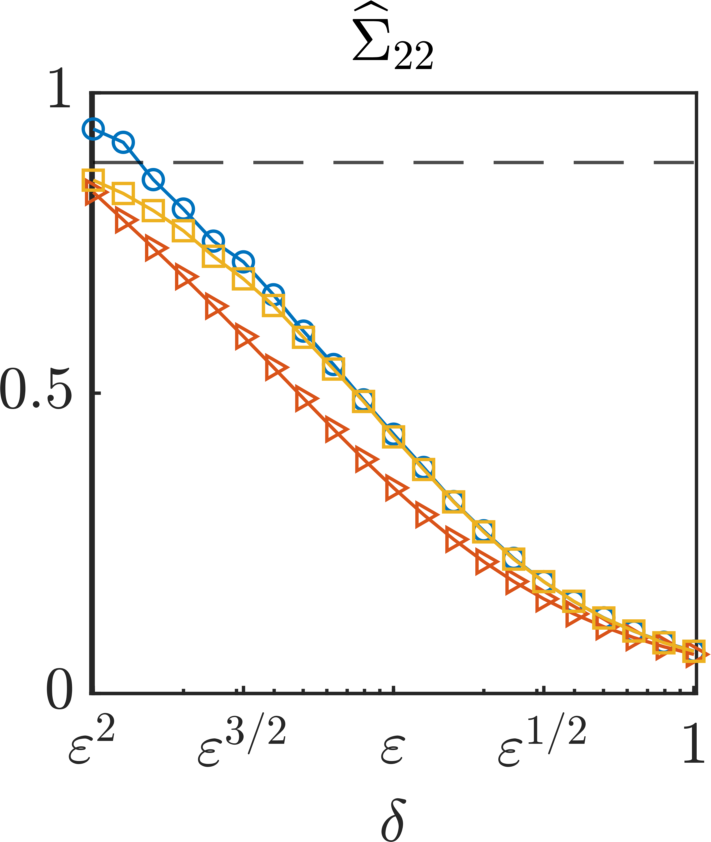} \\ 
		\includegraphics[]{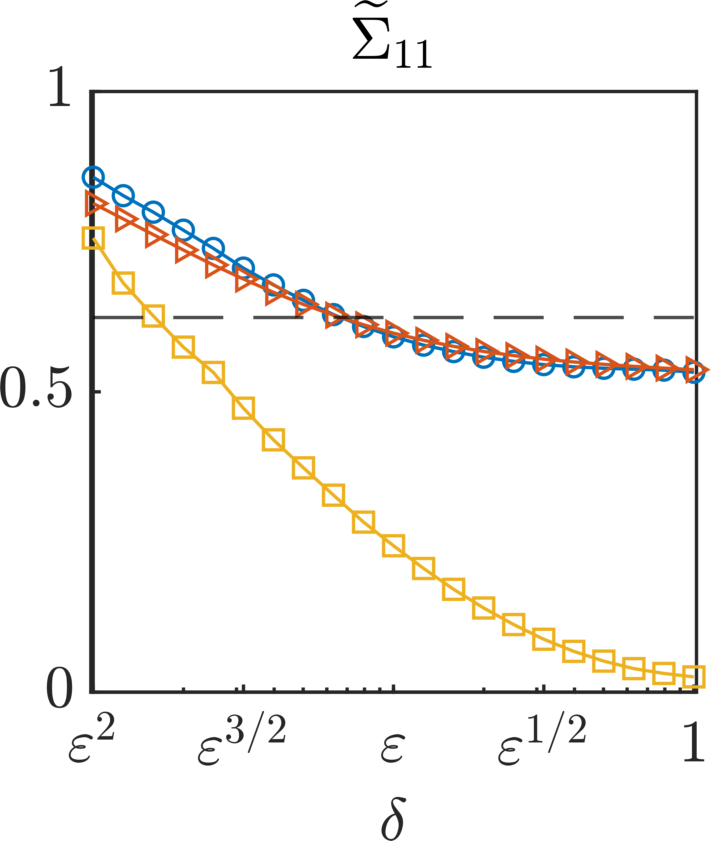} &
		\includegraphics[]{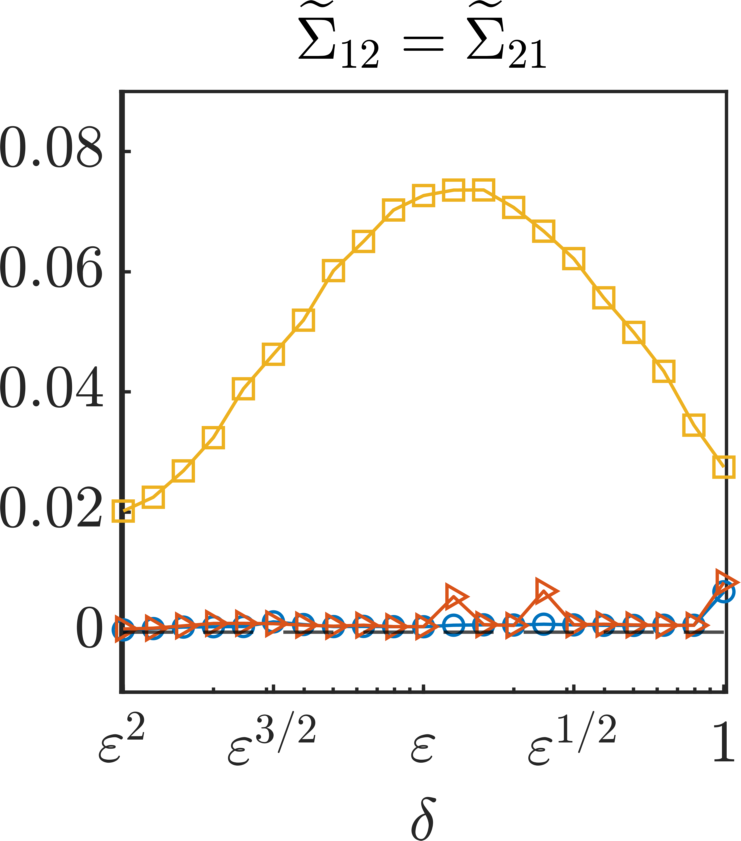} &
		\includegraphics[]{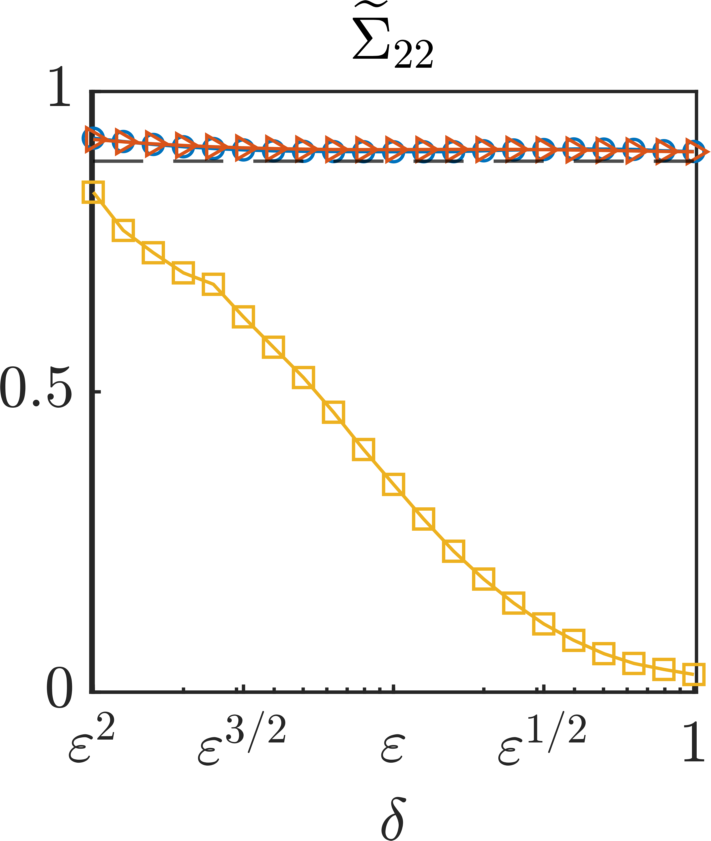}
	\end{tabular}
	\caption{Estimation of the effective diffusion coefficient for the two-dimensional example of \cref{sec:NumExp_2d}. First and second row: Dependence of the relative error of the estimators $\widehat \Sigma$ (first row) and $\widetilde \Sigma$ (second row) with respect to $t \in [0,T]$ (left) and to the subsampling/filtering width $\delta$ (right) for both filtering methods and subsampling. Third and fourth row: Sensitivity of the estimators of the entries of the diffusion matrix $\Sigma$ with respect to $\delta$, with both ``hat'' (third row) and ``tilde'' (fourth row) estimators, and for the same three methodologies as above.}
	\label{fig:TwoD_Diffusion}
\end{figure}

As a last numerical example, we consider a two-dimensional SDE ($d = 2$) of the form \eqref{eq:SDE_MS_Lang}. In particular, we let $N = 4$ and define
\begin{equation}
\begin{alignedat}{2}
V_1(x) &= \exp\left(-\norm{x - x_1}^2\right), \quad &&V_2(x) = \exp\left(-\norm{x - x_2}^2\right), \\
V_3(x) &= \exp\left(-\norm{x}^2\right), &&V_4(x) = \frac14 \norm{x}^4,
\end{alignedat}
\end{equation}
where $x_1 = (2, 2)^\top$, $x_2 = (-2, -2)^\top$. The exact drift coefficient in the multiscale dynamics is defined by $\alpha_1 = \alpha_2 = -15$, $\alpha_3 = 10$ and $\alpha_4 = 1$. We choose the fast-scale periodic potential $p\colon \R^2\to \R$ as
\begin{equation}
p(y) = \sin\left(y_1\right) + \sin^2\left(y_2\right),
\end{equation}
and let the diffusion coefficient $\sigma = 1$ and the scale-separation parameter $\epl = 0.1$. Since the fast-scale potential can be decomposed as $p(y) = p_1(y_1) + p_2(y_2)$, the homogenization coefficient $\mathcal K$ is diagonal and its diagonal components can be computed employing the one-dimensional formula \eqref{eq:K1d}. In particular, we have
\begin{equation}
\mathcal K \approx \begin{pmatrix} 0.62 & 0 \\ 0 & 0.88 \end{pmatrix}.
\end{equation}
The slow, fast, and multiscale potential functions are represented in \cref{fig:TwoD_Potential}. The slow-scale potential presents two wells around the points $x_1$ and $x_2$, a local maximum in the origin, and diverges outside any ball large enough and centered in the origin. The superposition of the slow and fast-scale potentials (evaluated in $y = x/\epl$) perturbs the slow-scale potential and is responsible for an infinity of non-negligible local minima. We note that due to the local minima, the local maximum in the origin, and the two-dimensional setup, transitions between the potential wells are rare. This compromises the accuracy of the inference results, especially for the drift coefficient, unless final time is taken large enough. 

We set $T = 2\cdot 10^5$ and generate synthetic observations $X^\epl = (X^\epl(t), 0\leq t \leq T)$ by integrating \eqref{eq:SDE_MS_Lang} with the Euler--Maruyama method with time step $\Delta_t = \epl^3$. We then estimate the effective drift coefficients $\{A_i \in \R^{2\times 2}\}_{i=1}^{4}$ and the effective diffusion matrix $\Sigma \in \R^{2\times 2}$ employing data filtered with the moving average and exponential kernels, and with subsampling for a comparison. For the drift coefficient, we measure accuracy by computing the relative error on the $16$-dimensional vector obtained by stacking all coefficients of the four $2\times 2$ effective drift matrices. Numerical results, given in \cref{fig:TwoD_Drift,fig:TwoD_Diffusion}, demonstrate that
\begin{enumerate}[label=(\alph*)]
	\item \cref{fig:TwoD_Drift}: The drift estimator obtained with both filtering methodologies is extremely accurate, given the complexity of the setting and the high-dimensionality of the coefficient. In particular, for all values of $\delta \in [\epl,1]$ we obtain relative errors below $10\%$. Moreover, implementing both filtering methodologies with $\delta = 1$, i.e., when the scale-separation parameter is unknown, yields quasi-optimal results. We remark that the estimator $\widehat A^1_{\mathrm{ma}}$ obtained with the moving average kernel and $\delta = 1$ appears to converge sensibly faster with respect to $t \in [0, T]$ than the corresponding estimator $\widehat A^{1,1}_{\mathrm{exp}}$ obtained with the exponential kernel. Conversely, the relative error for subsampling is dramatically higher, and subsampling should not in our opinion be employed in this high-dimensional setting. Always commenting on \cref{fig:TwoD_Drift}, we note that the drift function estimated with both filtering methods at final time is visually almost indistinguishable from the exact effective drift function. 
	\item \cref{fig:TwoD_Diffusion}: Likewise the numerical experiments of the previous sections, the diffusion estimators $\widehat \Sigma$ obtained with both filtering kernels and subsampling is not robust with respect to the filtering width, with the moving average kernel that seems to perform slightly better than the exponential kernel for $\delta = \epl$, and than subsampling when $\delta = \epl^{3/2}$. The estimators $\widetilde \Sigma$ obtained with the two filtering methods are instead extremely accurate at identifying both the diagonal components -- especially $\Sigma_{22}$, and the zero off-diagonal elements. The subsampling-based estimator $\widetilde \Sigma_{\mathrm{sub}}$, instead, suffers from the lack of accuracy of the drift estimator $\widehat A^\delta_{\mathrm{sub}}$, and is not reliable. Always commenting on \cref{fig:TwoD_Diffusion}, we remark that convergence with respect to $t \in [0,T]$ of the estimators $\widetilde \Sigma^1_{\mathrm{ma}}$ and $\widetilde \Sigma^{1,1}_{\mathrm{exp}}$ is similar, with the moving average filter seemingly less prone to instabilities for small $t$.
\end{enumerate}

\section{Asymptotic Unbiasedness of the Estimators}\label{sec:Proof}

In this section we present the proof of \cref{thm:Drift,thm:Diffusion,thm:Diffusion_tilde,cor:Diffusion_tilde}, i.e., the results of asymptotic unbiasedness for our filtering-based estimators. In \cite{AGP21} the proofs of convergence are obtained with the kernel $k_{\mathrm{exp}}^{\delta,1}$ by noticing that the original trajectory $X^\epl$ and its filtered version $Z_{\mathrm{exp}}^{\delta,\beta,\epl}$ are solution of an hypoelliptic system of Itô SDEs. For higher values of $\beta > 1$, the system describing the evolution of $X^\epl$ and $Z_{\mathrm{exp}}^{\delta,\beta,\epl}$ is not a Itô system due to the presence of a memory term. In case we consider the moving average kernel $k_{\mathrm{ma}}^\delta$ which we study in this paper, the memory term simplifies to a constant delay. Hence, the evolution of the filtered trajectory $Z_{\mathrm{ma}}^{\delta,\epl}$ can be coupled with the original trajectory $X^\epl$ through the system of stochastic delay differential equations (SDDEs)
\begin{equation}\label{eq:System}
\begin{aligned}
\d X^\epl(t) &= -\alpha \cdot V'(X^\epl(t)) \dd t - \frac1\epl p'\left(\frac{X^\epl(t)}\epl\right) \dd t + \sqrt{2\sigma} \dd W(t), \\
\d Z_{\mathrm{ma}}^{\delta,\epl}(t) &= -\frac1\delta (X^\epl(t-\delta) - X^\epl(t)) \dd t.
\end{aligned}
\end{equation}
To be precise, the system above is a combination of a Itô SDE and a delasy ordinary differential equation driven by a stochastic signal. Due to the theory of homogenization (see \cite[Chapter 3]{BLP78}, or \cite[Chapter 18]{PaS08}, or the proof of \cite[Lemma 3.9]{AGP21}), if $\delta$ is independent of $\epl$, the solution $(X^\epl,Z_{\mathrm{ma}}^{\delta,\epl})$ converges in law as random variables in $\mathcal C^0([0,T],\R^2)$ to the solution $(X^0,Z_{\mathrm{ma}}^{\delta,0})$ of the system
\begin{equation}\label{eq:System_hom}
\begin{aligned}
\d X^0(t) &= -A \cdot V'(X^0(t)) \dd t + \sqrt{2\Sigma} \dd W(t), \\
\d Z_{\mathrm{ma}}^{\delta,0}(t) &= -\frac1\delta (X^0(t-\delta) - X^0(t)) \dd t.
\end{aligned}
\end{equation}
In the following, we first focus on ergodic properties of the couples $(X^\epl, Z_{\mathrm{ma}}^{\delta,\epl})$ and $(X^0, Z_{\mathrm{ma}}^{\delta,0})$ evolving according to \eqref{eq:System} and \eqref{eq:System_hom}, respectively. Then, we employ the invariant measures and the Fokker--Planck equations derived through the ergodicity theory to prove asymptotic unbiasedness. We remark that the strategy we adopt is similar to the one of \cite{AGP21}. Still, different techniques need to be employed due to the delay in the second equation of the systems \eqref{eq:System} and \eqref{eq:System_hom}.

\subsection{Ergodic Properties}	

It is well-known (see, e.g., \cite{PaS07,AGP21}) that $X^\epl$ is geometrically ergodic with invariant measure $\mu^\epl$ on $\R$, whose density $\rho^\epl$ satisfying $\mu^\epl(\d x) = \rho^\epl(x) \dd x$ takes the Gibbs form
\begin{equation}
\rho^\epl(x) = \frac1{C^\epl_\mu} \exp\left(-\frac{V_\epl(x)}{\sigma}\right), \quad C^\epl_\mu = \int_\R  \exp\left(-\frac{V_\epl(x)}{\sigma}\right) \dd x,
\end{equation}
where 
\begin{equation}
V_\epl(x) \defeq \alpha \cdot V(x) + p\left(\frac{x}{\epl}\right).
\end{equation}
Moreover, an analogous result holds true for the homogenized process $X^0$, which is geometrically ergodic with invariant measure $\mu^0$ on $\R$, whose density $\rho^0$ satisfying $\mu^0(\d x) = \rho^0(x) \dd x$ is given by
\begin{equation}
\rho^0(x) = \frac1{C^0_\mu} \exp\left(-\frac{A \cdot V(x)}{\Sigma}\right), \quad C^0_\mu = \int_\R  \exp\left(-\frac{A \cdot V(x)}{\Sigma}\right) \dd x.
\end{equation}

We now introduce a similar result of ergodicity for the couples $(X^\epl,Z_{\mathrm{ma}}^{\delta,\epl})$ and $(X^0,Z_{\mathrm{ma}}^{\delta,0})$ satisfying \eqref{eq:System} and \eqref{eq:System_hom}, respectively, i.e., for the multiscale process and its filtered version. 

\begin{proposition}\label{prop:FokkerPlanck} Under \cref{as:regularity}, the solution $(X^\epl, Z_{\mathrm{ma}}^{\delta,\epl})$ of \eqref{eq:System} is ergodic, and the density $\widetilde \rho^\epl$ of its invariant measure $\widetilde \mu^\epl$ on $\R^2$, such that $\widetilde \mu^\epl(\d x, \d z) = \widetilde \rho^\epl(x, z) \dd x \dd z$, satisfies
	\begin{equation} \label{eq:FokkerPlanck}
	\begin{aligned}
	&\sigma \partial_{xx}^2 \widetilde \rho^\epl(x, z) + \partial_x\left(V_\epl'(x) \widetilde \rho^\epl(x,z)\right) + \frac1\delta \partial_z \left(\left(\int_\R y \psi^\epl(y \mid x, z) \dd y - x\right)\widetilde \rho^\epl(x, z)\right) = 0,\\
	&\int_\R\int_\R \widetilde \rho^\epl(x, z) \dd x \dd z = 1,
	\end{aligned}
	\end{equation}
	where, if $X^\epl(0) \sim \mu^\epl$, it holds
	\begin{equation}
	\int_\R y \psi^\epl(y \mid x, z) \dd y = \E\left[X^\epl(0) \mid X^\epl(\delta)=x, Z_{\mathrm{ma}}^{\delta,\epl}(\delta) = z\right],
	\end{equation}
	i.e., $\psi^\epl(\cdot \mid x,z)$ is the conditional density of $X^\epl(0)$ given $X^\epl(\delta)=x$ and $Z_{\mathrm{ma}}^{\delta,\epl}(\delta)=z$. Moreover, if $\delta$ is independent of $\epl$, the solution $(X^0,Z_{\mathrm{ma}}^{\delta,0})$ of \eqref{eq:System_hom} is ergodic, and the density $\widetilde \rho^0$ of its invariant measure $\widetilde \mu^0$ on $\R^2$, such that $\widetilde \mu^0(\d x, \d z) = \widetilde \rho^0(x, z) \dd x \dd z$, satisfies
	\begin{equation} \label{eq:FokkerPlanck_hom}
	\begin{aligned}
	&\Sigma \partial_{xx}^2 \widetilde \rho^0(x, z) + \partial_x\left(A \cdot V'(x) \widetilde \rho^0(x,z)\right) + \frac1\delta \partial_z \left(\left(\int_\R y \psi^0(y \mid x, z) \dd y - x\right)\widetilde \rho^0(x, z)\right) = 0,\\
	&\int_\R\int_\R \widetilde \rho^0(x, z) \dd x \dd z = 1,
	\end{aligned}
	\end{equation}
	where, if $X^0(0) \sim \mu^0$, it holds
	\begin{equation}
	\int_\R y \psi^0(y \mid x, z) \dd y = \E\left[X^0(0) \mid X^0(\delta)=x, Z_{\mathrm{ma}}^{\delta,0}(\delta) = z\right],
	\end{equation}
	i.e., $\psi^0(\cdot \mid x,z)$ is the conditional density of $X^0(0)$ given $X^0(\delta)=x$ and $Z_{\mathrm{ma}}^{\delta,0}(\delta)=z$.
\end{proposition}
\begin{proof}
	In order to prove that the joint process $(X^\epl, Z_{\mathrm{ma}}^{\delta,\epl})$ is ergodic, we show that it admits a unique invariant measure. If $X^\epl(0)$ is distributed accordingly to its invariant measure $\mu^\epl$, which exists due to \cref{as:regularity}, then the processes $(X^\epl(s), 0 \le s \le \delta)$ and $(X^\epl(s), t - \delta \le s \le t)$ are equally distributed for all $t \ge \delta$. Hence, the two-dimensional random variables $\left( X^\epl(\delta), \frac1\delta \int_0^\delta X^\epl(s) \dd s \right)$ and $\left( X^\epl(t), \frac1\delta \int_{t-\delta}^t X^\epl(s) \dd s \right)$ are equal in law for all $t \ge \delta$. Recalling that the joint process $\left( X^\epl(t), \frac1\delta \int_{t-\delta}^t X^\epl(s) \dd s \right)$ is the solution $(X^\epl(t), Z_{\mathrm{ma}}^{\delta,\epl}(t))$ of the system \eqref{eq:System}, it follows that the invariant measure $\widetilde \mu^\epl$ on $\R^2$ is the law of the random variable $\left( X^\epl(\delta), \frac1\delta \int_0^\delta X^\epl(s) \dd s \right)$. The uniqueness of the invariant measure $\widetilde \mu^\epl$ is then a direct consequence of the uniqueness of the invariant measure $\mu^\epl$ for the process $X^\epl$ since the joint measure $\widetilde \mu^\epl$ is uniquely determined by its marginal $\mu^\epl$. Moreover, the Fokker--Planck equation for the one-time PDF related to an SDDE with a single fixed delay is well-known (see, e.g., \cite{GLL99, Fra05, LoK19}) and the stationary equation \eqref{eq:FokkerPlanck} for the density $\widetilde \rho^\epl$ of $\widetilde \mu^\epl$ is then obtained due to the particular form of the system \eqref{eq:System}. Finally, the results corresponding to the homogenized system \eqref{eq:System_hom} can be proved analogously.
\end{proof}

The following formulas, which will be employed in the proof of the main results, are then direct consequences of the Fokker--Planck equations obtained above.

\begin{lemma}\label{lem:Magic} Let $\widetilde \rho^\epl(x, z) = \rho^\epl(x) \varphi^\epl(z \mid x)$, where $\rho^\epl$ and $\widetilde \rho^\epl$ are the densities of the invariant measures $\mu^\epl$ and $\widetilde \mu^\epl$ of $X^\epl$ and $(X^\epl, Z_{\mathrm{ma}}^{\delta,\epl})$, respectively, and where $\varphi^\epl$ is the conditional density of $Z_{\mathrm{ma}}^{\delta,\epl}$ given $X^\epl$. Then, if $X^\epl(0) \sim \mu^\epl$, it holds
	\begin{equation} \label{eq:magic}
	\sigma \int_\R\int_\R  V'(z) \rho^\epl(x) \partial_x \varphi^\epl(z \mid x) \dd x \dd z = \frac1\delta \E^{\widetilde \mu^\epl}\left[\left(X^\epl(\delta) - Z_{\mathrm{ma}}^{\delta,\epl}(\delta)\right) \left(X^\epl(\delta) - X^\epl(0\right)) V''(Z_{\mathrm{ma}}^{\delta,\epl}(\delta))\right].
	\end{equation}
	Moreover, if $\delta$ is independent of $\epl$ and writing $\widetilde \rho^0(x, z) = \rho^0(x)\varphi^0(z \mid x)$ for the density of the homogenized invariant measure $\widetilde \mu^0$ of $(X^0, Z_{\mathrm{ma}}^{\delta,0})$, it holds
	\begin{equation} \label{eq:magic_hom}
	\Sigma \int_\R\int_\R  V'(z) \rho^0(x) \partial_x \varphi^0(z \mid x) \dd x \dd z = \frac1\delta \E^{\widetilde \mu^0}\left[\left(X^0(\delta) - Z_{\mathrm{ma}}^{\delta,0}(\delta)\right) \left(X^0(\delta) - X^0(0)\right) V''(Z_{\mathrm{ma}}^{\delta,0}(\delta))\right].
	\end{equation}
\end{lemma}
\begin{proof}
	We proceed similarly to the proof of \cite[Lemma 3.5]{AGP21}. Replacing the decomposition $\widetilde \rho^\epl(x, z) = \rho^\epl(x) \varphi^\epl(z \mid x)$ into the Fokker--Planck equation \eqref{eq:FokkerPlanck} gives
	\begin{equation}
	\partial_x \left( \sigma \rho^\epl(x) \partial_x \varphi^\epl(z \mid x) \right) + \partial_z \left( \frac1\delta \left( \int_\R y \psi^\epl(y \mid x,z) \dd y - x \right) \rho^\epl(x) \varphi^\epl(z \mid x) \right) = 0.
	\end{equation}
	We then multiply the equation above by a smooth function $f \colon \R^2 \to \R^N$, $f = f(x,z)$, and integrate first with respect to $x$ and $z$ and then by parts, obtaining
	\begin{equation}
	\sigma \int_\R \int_\R \partial_x f(x,z) \rho^\epl(x) \partial_x \varphi^\epl(z \mid x) \dd x \dd z = \frac1\delta \E^{\widetilde \mu^\epl} \left[ \partial_z f(X^\epl(\delta), Z_{\mathrm{ma}}^{\delta,\epl}(\delta))(X^\epl(\delta) - X^\epl(0)) \right].
	\end{equation}
	The choice $f(x,z) = (x - z)V'(z) + V(z)$ gives equation \eqref{eq:magic}. Finally, equation \eqref{eq:magic_hom} is obtained analogously employing the Fokker--Planck equation of the homogenized SDE \eqref{eq:FokkerPlanck_hom}.
\end{proof}

\subsection{Preliminary Results}

Let us first introduce the notation
\begin{equation}
\begin{aligned}
&\widetilde{\mathcal M}_\epl \defeq \E^{\widetilde \mu^\epl}\left[V'(Z_{\mathrm{ma}}^{\delta,\epl}) \otimes V'(X^\epl)\right], \qquad &&\widetilde{\mathcal M}_0 \defeq \E^{\widetilde \mu^0}\left[V'(Z_{\mathrm{ma}}^{\delta,0}) \otimes V'(X^0)\right], \\
&\mathcal M_\epl \defeq \E^{\mu^\epl}\left[V'(X^\epl) \otimes V'(X^\epl)\right],  &&\mathcal M_0 \defeq \E^{\mu^0}\left[V'(X^0) \otimes V'(X^0)\right],
\end{aligned}
\end{equation}
which is repeatedly employed below. Before presenting the main proofs, we introduce two auxiliary lemmas.

\begin{lemma} \label{lem:distance_XZ} 
	Under \cref{as:regularity}, it holds
	\begin{equation}\label{eq:difference_XZ}
	X^\epl(\delta) - Z_{\mathrm{ma}}^{\delta,\epl}(\delta) = \frac{\sqrt{2\sigma}}{\delta} \int_0^\delta t (1 + \Phi'(Y^\epl(t))) \dd W(t) + R(\epl,\delta),
	\end{equation}
	where $\Phi$ is the solution of the cell problem \eqref{eq:Cell} and where the remainder $R(\epl, \delta)$ satisfies for all $p\geq 1$ and a constant $C > 0$ independent of $\epl$ and $\delta$
	\begin{equation}\label{eq:difference_XZ_remainder}
	\E^{\widetilde \mu^\epl} \left[ \abs{R(\epl,\delta)}^p \right]^{1/p} \le C \left( \epl + \delta \right).
	\end{equation}
	Moreover, if $X^\epl(0)$ is stationary, i.e. $X^\epl(0) \sim \mu^\epl$, it holds
	\begin{align}
		&\E^{\widetilde \mu^\epl}\left[\abs{X^\epl - Z_{\mathrm{ma}}^{\delta,\epl}}^p\right]^{1/p} \leq C \left( \delta^{1/2} + \epl\right), \label{eq:bound_XZ}\\
		&\E^{\widetilde \mu^\epl}\left[\abs{Z_{\mathrm{ma}}^{\delta,\epl}}^p\right]^{1/p} \leq C. \label{eq:bound_Z}
	\end{align}
\end{lemma}
\begin{proof} Employing the decomposition (5.8) in \cite{PaS07} and due to \cite[Lemma 5.5, Propsition 5.8]{PaS07} we have for all $t \in [0,\delta]$
	\begin{equation} \label{eq:decompositon_X}
	X^\epl(\delta) = X^\epl(t) + \sqrt{2\sigma} \int_t^\delta (1 + \Phi'(Y^\epl(s))) \dd W(s) + R(\epl,\delta),
	\end{equation}
	where the remainder satisfies for all $p \ge 1$ and for a constant $C > 0$ independent of $\epl$ and $\delta$
	\begin{equation} \label{eq:bound_remainder}
	\E^{\widetilde \mu^\epl} \left[ \abs{R(\epl,\delta)}^p \right]^{1/p} \le C \left( \epl + \delta \right).
	\end{equation}	
	Therefore, we obtain 
	\begin{equation}
	X^\epl(\delta) - Z_{\mathrm{ma}}^{\delta,\epl}(\delta) = \frac1\delta \int_0^\delta (X^\epl(\delta) - X^\epl(t)) \dd t = \frac{\sqrt{2\sigma}}{\delta} \int_0^\delta \int_t^\delta (1 + \Phi'(Y^\epl(s))) \dd W(s) \dd t + R(\epl,\delta),
	\end{equation}
	which by Fubini's theorem yields
	\begin{equation}
	X^\epl(\delta) - Z_{\mathrm{ma}}^{\delta,\epl}(\delta) = \frac{\sqrt{2\sigma}}{\delta} \int_0^\delta t (1 + \Phi'(Y^\epl(t))) \dd W(t) + R(\epl,\delta),
	\end{equation}
	and proves \eqref{eq:difference_XZ} and \eqref{eq:difference_XZ_remainder}. By the Itô isometry, it holds
	\begin{equation} \label{eq:bound_mart_2}
	\E^{\widetilde \mu^\epl} \left[ \abs{\int_0^\delta t (1 + \Phi'(Y^\epl(t))) \dd W(t)}^p \right]^{1/p} \le C \delta^{3/2},
	\end{equation}
	which, together with \eqref{eq:difference_XZ}, \eqref{eq:difference_XZ_remainder} and the proof of \cref{prop:FokkerPlanck}, gives \eqref{eq:bound_XZ}. Finally, \eqref{eq:bound_Z} is proved by applying the triangle inequality and due to \eqref{eq:bound_XZ} and \cite[Corollary 5.4]{PaS07}.
\end{proof}

\begin{lemma} \label{lem:distance_MeplM0} 
	Under \cref{as:regularity} and if $\delta = \epl^\zeta$ with $\zeta \in (0,2)$, then it holds
	\begin{equation}
	\lim_{\epl \to 0} \widetilde{\mathcal M}_\epl = \mathcal M_0.
	\end{equation}
\end{lemma}
\begin{proof} By the triangle inequality, we have
	\begin{equation}
	\norm{\widetilde{\mathcal M}_\epl - \mathcal M_0} \leq \norm{\widetilde{\mathcal M}_\epl - \mathcal M_\epl} + \norm{\mathcal M_\epl - \mathcal M_0}.
	\end{equation}
	The first term vanishes as $\epl \to 0$ due to \cref{lem:distance_XZ}, \cite[Corollary 5.4]{PaS07} and since $V'$ is Lipschitz under \cref{as:regularity}. The second term vanishes due to the theory of homogenization as $\epl\to 0$.
\end{proof}

\subsection{Proof of the Main Results}

We can now prove our main results, i.e., \cref{thm:Drift,thm:Diffusion,thm:Diffusion_tilde,cor:Diffusion_tilde}.

\begin{proof}[Proof of \cref{thm:Drift}]
	Following the proof of \cite[Theorem 3.12]{AGP21}, we have
	\begin{equation}
	\widehat A_{\mathrm{ma}}^\delta(X^\epl, T) = \alpha + I_1 - I_2,
	\end{equation}
	where
	\begin{equation}
	\begin{aligned}
	I_1 &= \frac1T M_{\mathrm{ma}}^\delta(X^\epl,T)^{-1}\int_0^T \frac1\epl p'\left(\frac{X^\epl(t)}{\epl}\right) V'(Z_{\mathrm{ma}}^{\delta,\epl}(t)) \dd t, \\
	I_2 &= \frac{\sqrt{2\sigma}}{T} M_{\mathrm{ma}}^\delta(X^\epl,T)^{-1} \int_0^T V'(Z_{\mathrm{ma}}^{\delta,\epl}(t)) \dd W(t),
	\end{aligned}
	\end{equation}
	and where
	\begin{equation}
	\lim_{T\to\infty} I_2 = 0,
	\end{equation}
	uniformly in $\epl$ by \cref{lem:distance_XZ} and the strong law of large numbers for martingales. Considering $I_1$, due to \cref{as:regularity} the ergodic theorem and an integration by parts yield
	\begin{equation}
	\lim_{T\to\infty} I_1 = -\alpha + \widetilde{\mathcal M}_\epl^{-1} \sigma \int_\R\int_\R V'(z) \rho^\epl(x) \partial_x \varphi^\epl(z \mid x) \dd x \dd z,
	\end{equation}
	where $\varphi^\epl(z \mid x)$ is defined in \cref{lem:Magic}, which also implies
	\begin{equation}
	\lim_{T\to\infty} I_1 = -\alpha + \mathcal A^\epl(\delta),
	\end{equation}
	where
	\begin{equation} \label{eq:A_epl_delta}
	\mathcal A^\epl(\delta) = \frac1\delta \widetilde{\mathcal M}_\epl^{-1} \E^{\widetilde \mu^\epl} \left[ (X^\epl(\delta) - Z_{\mathrm{ma}}^{\delta,\epl}(\delta)) (X^\epl(\delta) - X^\epl(0)) V''(Z_{\mathrm{ma}}^{\delta,\epl}(\delta)) \right].
	\end{equation}
	It remains to show that
	\begin{equation}
	\lim_{\epl \to 0} \mathcal A^\epl(\delta) = A,
	\end{equation}
	for which we consider two cases, corresponding to $\delta$ independent of $\epl$ and $\delta = \epl^\zeta$ with $\zeta \in (0,2)$, respectively.
	
	\textit{Case 1: $\delta$ independent of $\epl$.} In this case, the theory of homogenization yields
	\begin{equation}
	\lim_{\epl \to 0} \mathcal A^\epl(\delta) = \frac1\delta \widetilde{\mathcal M}_0^{-1} \E^{\widetilde \mu^0} \left[ (X^0(\delta) - Z_{\mathrm{ma}}^{\delta,0}(\delta)) (X^0(\delta) - X^0(0)) V''(Z_{\mathrm{ma}}^{\delta,0}(\delta)) \right],
	\end{equation}
	so that applying \cref{lem:Magic} for the homogenized equation backwards we have
	\begin{equation}
	\lim_{\epl \to 0} \mathcal A^\epl(\delta) = \widetilde{\mathcal M}_0^{-1} \Sigma \int_\R\int_\R V'(z) \rho^0(x) \partial_x \varphi^0(z \mid x) \dd x \dd z.
	\end{equation}
	An integration by parts then gives
	\begin{equation}
	\lim_{\epl \to 0} \mathcal A^\epl(\delta) = \widetilde{\mathcal M}_0^{-1} \widetilde{\mathcal M}_0 A = A,
	\end{equation}
	which concludes \textit{Case 1}.
	
	\textit{Case 2: $\delta = \epl^\zeta$ with $\zeta \in (0, 2)$.} 
	Replacing formulas \eqref{eq:decompositon_X} with $t=0$ and \eqref{eq:difference_XZ} into \eqref{eq:A_epl_delta} gives
	\begin{equation}
	\begin{aligned}
	\mathcal A^\epl(\delta) &= \frac{2\sigma}{\delta^2} \widetilde{\mathcal M}_\epl^{-1} \E^{\widetilde \mu^\epl} \left[ \left( \int_0^\delta t (1 + \Phi'(Y^\epl(t))) \dd W(t) \right) \left( \int_0^\delta (1 + \Phi'(Y^\epl(t))) \dd W(t) \right) V''(Z_{\mathrm{ma}}^{\delta,\epl}(\delta)) \right] \\
	&\quad + \widetilde R_1(\epl,\delta),
	\end{aligned}
	\end{equation}
	where, due to \cref{lem:distance_XZ}, estimate \eqref{eq:bound_mart_2} and the fact that by the Itô isometry
	\begin{equation} \label{eq:bound_mart_1}
	\E^{\widetilde \mu^\epl} \left[ \abs{\int_t^\delta (1 + \Phi'(Y^\epl(s))) \dd W(s)}^p \right]^{1/p} \le C \delta^{1/2},
	\end{equation}
	it follows that the remainder satisfies
	\begin{equation} \label{eq:bound_remainder_tilde}
	\norm{\widetilde R_1(\epl,\delta)} \le C \left( \delta^{1/2} + \epl \delta^{-1/2} + \epl^2 \delta^{-1} \right).
	\end{equation}
	Moreover, since $V''$ is Lipschitz under \cref{as:regularity} and due to the triangle inequality, equation \eqref{eq:decompositon_X}, estimates \eqref{eq:difference_XZ_remainder}, \eqref{eq:bound_mart_1} and \cref{lem:distance_XZ}, it holds for all $t \in [0, \delta]$
	\begin{equation} \label{eq:bound}
	\E^{\widetilde \mu^\epl} \left[ \norm{V''(Z_{\mathrm{ma}}^{\delta,\epl}(\delta)) - V''(X^\epl(t))}^p \right]^{1/p} \le C \left( \epl + \delta^{1/2} \right),
	\end{equation}
	which for $\epl$ and $\delta$ sufficiently small is at most of order $\mathcal O\left(\norm{\widetilde R_1(\epl,\delta)}\right)$. Hence, by the Itô isometry 
	\begin{equation}
	\begin{aligned}
	\mathcal A^\epl(\delta) &= \frac{2\sigma}{\delta^2} \widetilde{\mathcal M}_\epl^{-1} \E^{\widetilde \mu^\epl} \left[ \left( \int_0^\delta t (1 + \Phi'(Y^\epl(t))) \dd W(t) \right) \left( \int_0^\delta (1 + \Phi'(Y^\epl(t))) V''(X^\epl(t)) \dd W(t) \right) \right] \\
	&\quad + \widetilde R_2(\epl,\delta) \\
	&= \frac{2\sigma}{\delta^2} \widetilde{\mathcal M}_\epl^{-1} \int_0^\delta t \E^{\widetilde \mu^\epl} \left[ (1 + \Phi'(Y^\epl(t)))^2 V''(X^\epl(t)) \right] \dd t + \widetilde R_2(\epl,\delta),
	\end{aligned}
	\end{equation}
	where due to \eqref{eq:bound_remainder_tilde} and \eqref{eq:bound} the remainder satisfies
	\begin{equation}
	\norm{\widetilde R_2(\epl,\delta)} \le C \left( \delta^{1/2} + \epl \delta^{-1/2} + \epl^2 \delta^{-1} \right).
	\end{equation}
	Repeating the last part of the proof of \cite[Lemma 3.17]{AGP21}, we then obtain
	\begin{equation}
	\begin{aligned}
	\mathcal A^\epl(\delta) &= \frac{2\sigma \mathcal K}{\delta^2} \widetilde{\mathcal M}_\epl^{-1} \E^{\mu^0} [V''(X^0)] \int_0^\delta t \dd t + \widetilde R_2(\epl,\delta) \\
	&= \Sigma \widetilde{\mathcal M}_\epl^{-1} \E^{\mu^0} [V''(X^0)] + \widetilde R_2(\epl,\delta).
	\end{aligned}
	\end{equation}
	Finally, since $\delta = \epl^\zeta$ with $\zeta \in (0,2)$, by \eqref{eq:bound_remainder_tilde} and due to \cref{lem:distance_MeplM0} we obtain
	\begin{equation}
	\lim_{\epl \to 0} \mathcal A^\epl(\delta) = \Sigma \mathcal M_0^{-1} \E^{\mu^0} [V''(X^0)],
	\end{equation}
	and an integration by parts gives
	\begin{equation}
	\lim_{\epl \to 0} \mathcal A^\epl(\delta) = \Sigma \mathcal M_0^{-1} \frac1\Sigma \mathcal M_0 A = A, 
	\end{equation}
	which proves \textit{Case 2} and therefore concludes the proof.
\end{proof}

\begin{proof}[Proof of \cref{thm:Diffusion}] Due to \cref{as:regularity} the ergodic theorem gives
	\begin{equation} \label{eq:limtTdiffusion}
	\lim_{T\to\infty}\widehat\Sigma_{\mathrm{ma}}^\delta(X^\epl, T) = \frac1\delta \E^{\widetilde \mu^\epl}\left[\left(X^\epl(\delta) - Z_{\mathrm{ma}}^{\delta,\epl}(\delta)\right) \left(X^\epl(\delta) - X^\epl(0)\right)\right].
	\end{equation}
	Following step-by-step \textit{Case 2} of the proof of \cref{thm:Drift} with the value $1$ instead of $V''(Z_{\mathrm{ma}}^{\delta,\epl}(\delta))$, and without the pre-multiplication by $\widetilde {\mathcal M}_\epl^{-1}$, we obtain the desired result.
\end{proof}

\begin{remark} It is clear from the proof of \cref{thm:Diffusion} that it is theoretically not possible to choose $\delta$ independent of $\epl$ in the computation of $\widehat \Sigma_{\mathrm{ma}}^{\delta}(X^\epl, T)$. Let $N = 1$ and $V(x) = x^2/2$, so that $X^0$ is an Ornstein-Uhlenbeck process. In this case, the process $X^0$ is a Gaussian process such that at stationarity $X^0 \sim \mathcal {GP}(0, \mathcal C(t,s))$ where
	\begin{equation} \label{eq:covarianceOU}
	\mathcal C(t,s) = \frac{\Sigma}{A} e^{-A\abs{t-s}}.
	\end{equation}
	By \eqref{eq:limtTdiffusion} and \eqref{eq:covarianceOU} we can therefore explicitly compute
	\begin{equation}
	\lim_{\epl \to 0} \lim_{T \to \infty} \widehat \Sigma_{\mathrm{ma}}^\delta(X^\epl,T) = \frac1\delta \E^{\widetilde \mu^0}\left[\left(X^0(\delta) - Z_{\mathrm{ma}}^{\delta,0}(\delta)\right) \left(X^0(\delta) - X^0(0)\right)\right] = \frac{1 - e^{-\delta A}}{\delta A} \Sigma,
	\end{equation}
	so that $\widehat \Sigma_{\mathrm{ma}}^\delta(X^\epl, T)$ is asymptotically unbiased only if $\delta \to 0$.
\end{remark}

\begin{proof}[Proof of \cref{thm:Diffusion_tilde}] 
	Notice notice that due to \eqref{eq:failure_classic} we have
	\begin{equation}
	\frac{\langle X^\epl \rangle_T}{2T} = \sigma,
	\end{equation}
	which together with \eqref{eq:hypothesis_convergence_drift} and by \eqref{eq:failure_classic} implies
	\begin{equation}
	\lim_{\epl \to 0} \lim_{T \to \infty} \widetilde \Sigma(X^\epl,T) = \frac{\alpha^\top A}{(\alpha^\top \alpha)} \sigma.
	\end{equation}
	Finally, since $A = \mathcal K \alpha$ and $\Sigma = \mathcal K \sigma$ we obtain
	\begin{equation}
	\lim_{\epl \to 0} \lim_{T \to \infty} \widetilde \Sigma(X^\epl,T) = \mathcal K \sigma = \Sigma,
	\end{equation}
	which is the desired result.
\end{proof}

\begin{proof}[Proof of \cref{cor:Diffusion_tilde}] 
	The desired results follow directly from \cref{thm:Drift}, \cite[Theorems 3.12 and 3.18]{AGP21}, \cite[Theorem 3.5]{PaS07} and \cref{thm:Diffusion_tilde}. We remark that the limit in \cite[Theorem 3.18]{AGP21} holds true also a.s. and the proof of \cite[Theorem 3.5]{PaS07} can be modified (see \cref{rem:proofPS}) such that hypothesis \eqref{eq:hypothesis_convergence_drift} is satisfied.
\end{proof}

\begin{remark} \label{rem:proofPS}
	The proof of \cite[Theorem 3.5]{PaS07} can be modified in order to show that the estimator $\widehat A_{\mathrm{sub}}(X^\epl,T)$ given in \eqref{eq:DriftEstimatorSub} satisfies
	\begin{equation} \label{eq:limitPS}
	\lim_{\epl \to 0} \lim_{T \to \infty} \widehat A_{\mathrm{sub}}(X^\epl,T) = A, \quad \text{a.s.}
	\end{equation}
	Due to \cref{as:regularity}\ref{as:regularity_diss} and by the ergodic theorem we have
	\begin{equation}
	\lim_{T \to \infty} \widehat A_{\mathrm{sub}}(X^\epl,T) = - \frac{\E^{\mu^\epl}[V'(X^\epl(0))(X^\epl(\delta) - X^\epl(0))]}{\E^{\mu^\epl}[V'(X^\epl(0))^2]}.
	\end{equation}
	We then employ \cite[Lemma 5.2]{APZ22} with $f$ the identity function and $\Delta = \delta$ and we notice that the martingale
	\begin{equation}
	M^\epl(t) \defeq \sqrt{2\sigma} \int_0^t (1 + \Phi'(Y^\epl(s))) \dd W(s),
	\end{equation}
	where $\Phi$ is defined in \eqref{eq:Cell} and $Y^\epl(s) = X^\epl(s)$, is such that $M^\epl(0) = 0$. Therefore, we obtain
	\begin{equation}
	\lim_{T \to \infty} \widehat A_{\mathrm{sub}}(X^\epl,T) = A + \widetilde R(\epl,\delta),
	\end{equation}
	where the remainder satisfies for a constant $C > 0$ independent of $\epl$ and $\delta$
	\begin{equation}
	\abs{\widetilde R(\epl,\delta)} \le C \left( \epl \delta^{-1} + \delta^{1/2} \right).
	\end{equation}
	Finally, since $\delta = \epl^{\zeta}$ with $\zeta \in (0,1)$ we deduce the desired result \eqref{eq:limitPS}.
\end{remark}

\section{Conclusion}\label{sec:Conclusion}

In this work, we introduced a novel methodology for inferring effective diffusions from observations of multiscale dynamics based on filtering the data with moving averages. Asymptotic unbiasedness is rigorously proved by originally exploiting an ergodicity result for SDDEs. Our method is robust, easy to implement, computationally uninvolved, and outperforms the standard technique of subsampling on a range of test cases. Moreover, the performances are comparable to a similar class of estimators that we introduced in our previous work \cite{AGP21}. The accuracy of our methodology in the multiscale, multi-dimensional, and highly-parametrised case is surprisingly high in view of its simplicity and low computational involvement.

We believe that future developments could go in the direction of:
\begin{enumerate}
	\item Deriving asymptotically unbiased estimators for the diffusion coefficient which are robust in practice and do not rely on the drift estimator,
	\item Extending the filtered data methodology to multiscale SDEs with non-constant diffusion terms, or for drift functions which do not depend linearly on the parameters,
	\item Extending the analysis to the non-parametric case.
\end{enumerate} 

\subsection*{Acknowledgements} 

We thank the anonymous reviewer whose comments helped improve and clarify this manuscript. We are grateful to Assyr Abdulle and Grigorios A. Pavliotis for invaluable support and advice over the years.

\bibliographystyle{siamnodash}
\bibliography{anmc}

\end{document}